\documentclass[12pt, twoside]{article}
\usepackage{amsmath,amsthm,amssymb}
\usepackage{times}
\usepackage{enumerate}

\def\titlerunning#1{\gdef\titrun{#1}}
\makeatletter
\def\author#1{\gdef\autrun{\def\and{\unskip, }#1}\gdef\@author{#1}}
\def\address#1{{\def\and{\\\hspace*{18pt}}\renewcommand{\thefootnote}{}%
\footnote {#1}}%
\markboth{\autrun}{\titrun}}
\makeatother
\def\email#1{e-mail: #1}
\def\subjclass#1{{\renewcommand{\thefootnote}{}%
\footnote{\emph{Mathematics Subject Classification (2010):} #1}}}
\def\keywords#1{\par\medskip
\noindent\textbf{Keywords.} #1}

\newtheorem{thm}{Theorem}[section]

\newtheorem{Lemma}[thm]{Lemma}
\newtheorem{Proposition}[thm]{Proposition}
\newtheorem{Theorem}[thm]{Theorem}
\newtheorem{Corollary}[thm]{Corollary}
\newtheorem*{nntheorem}{Theorem}
\newtheorem{lemma}[thm]{Lemma}
\newtheorem{proposition}[thm]{Proposition}
\newtheorem{theorem}[thm]{Theorem}
\newtheorem{corollary}[thm]{Corollary}

\theoremstyle{definition}

\newtheorem{definition}[thm]{Definition}

\newtheorem{remark}[thm]{Remark}

\newtheorem{example}[thm]{Example}

\numberwithin{equation}{section}

\newenvironment{tabsection}{}{}

\usepackage{amssymb,amsmath,amscd,dsfont}
\usepackage[all,cmtip,line]{xy}
\usepackage{slashed}
\newcommand{\blabel}[1]{\setlength\fboxrule{.0em}\fbox{\ensuremath{\scriptstyle #1}}}
\newcommand{\op}[1]{\ensuremath{\operatorname{#1}}}
\newcommand{\wt}[1]{\ensuremath{\widetilde{#1}}}
\newcommand{\wh}[1]{\ensuremath{\widehat{#1}}}
\newcommand{\ol}[1]{\ensuremath{\overline{#1}}}
\newcommand{\ul}[1]{\ensuremath{\underline{#1}}}
\newcommand{\cA}{\ensuremath{\mathcal{A}}}
\newcommand{\cC}{\ensuremath{\mathcal{C}}}
\newcommand{\cE}{\ensuremath{\mathcal{E}}}
\newcommand{\cF}{\ensuremath{\mathcal{F}}}
\newcommand{\cG}{\ensuremath{\mathcal{G}}}
\newcommand{\cH}{\ensuremath{\mathcal{H}}}
\newcommand{\cK}{\ensuremath{\mathcal{K}}}
\newcommand{\cL}{\ensuremath{\mathcal{L}}}
\newcommand{\cP}{\ensuremath{\mathcal{P}}}
\newcommand{\cT}{\ensuremath{\mathcal{T}}}
\newcommand{\cU}{\ensuremath{\mathcal{U}}}
\newcommand{\cX}{\ensuremath{\mathcal{X}}}
\newcommand{\fz}{\ensuremath{\mathfrak{z}}}
\newcommand{\fg}{\ensuremath{\mathfrak{g}}}
\newcommand{\dd}{\op{\tt{d}}}
\newcommand{\R}{\ensuremath{\mathbb{R}}}
\newcommand{\N}{\ensuremath{\mathbb{N}}}
\newcommand{\Z}{\ensuremath{\mathbb{Z}}}
\newcommand{\id}{\ensuremath{\operatorname{id}}}
\newcommand{\pr}{\ensuremath{\operatorname{pr}}}
\newcommand{\ev}{\ensuremath{\operatorname{ev}}}
\newcommand{\ad}{\ensuremath{\operatorname{ad}}}
\newcommand{\cat}[1]{\ensuremath{\mathsf{\mathop{#1}}}}
\newcommand{\Aut}{\ensuremath{\operatorname{Aut}}}
\newcommand{\Hom}{\ensuremath{\operatorname{Hom}}}
\newcommand{\im}{\ensuremath{\operatorname{im}}}
\newcommand{\tx}[1]{\ensuremath{\text{#1}}}
\newcommand{\per}{\ensuremath{\operatorname{per}}}
\newcommand{\se}{\ensuremath{\nobreak\subseteq\nobreak}}
\newcommand{\from}{\ensuremath{\nobreak\colon\nobreak}}
\renewcommand{\to}{\ensuremath{\nobreak\rightarrow\nobreak}}
\newcommand{\rrarrow}{\hspace{.05cm}\mbox{\put(0,-2){
$\rightarrow$}\put(0,2){ $\rightarrow$}\hspace{.5cm}}}
\newcommand{\rrrarrow}{\hspace{.05cm}\mbox{\,\put(0,-3){
$\rightarrow$}\put(0,1){ $\rightarrow$}\put(0,5){
$\rightarrow$}\hspace{.5cm}}}
\newcommand{\rrrrarrow}{\mbox{\,\put(0,-3){ $\rightarrow$}\put(0,.5){
$\rightarrow$}\put(0,4){ $\rightarrow$}\put(0,7.5){
$\rightarrow$}\hspace{.45cm}}}

\usepackage{mathrsfs} \usepackage{hyperref} \usepackage{graphicx}
\usepackage{caption}
\usepackage{wrapfig} \usepackage{ifthen} \usepackage[utf8]{inputenc}
\usepackage[usenames,dvipsnames]{color}

\hypersetup{
 urlcolor = red,
 colorlinks = true,
 linkcolor = blue,
 citecolor = blue,
 linktocpage = true,
 pdftitle = {Integrating central extensions of Lie algebras via Lie 2-groups},
 pdfauthor = {Chenchang Zhu, Christoph Wockel},
 bookmarksopen = true,
 bookmarksopenlevel = 1,
 unicode = true,
 hypertexnames =false 
}

\usepackage[all,cmtip,line]{xy} \CompileMatrices
\renewcommand{\rrarrow}{\rightrightarrows}

\newcommand{\bB}{\ensuremath{\mathbf{B}}}
\newcommand{\highlight}[1]{{#1}}
\newcommand{\Per}[1]{\ensuremath{\im(\per_{#1})}}
\newcommand{\disc}[1]{\ensuremath{\underline{#1}}}
\newcommand{\locsm}[1]{\ensuremath{\wt{#1}}} \newcommand{\emptycomment}[1]{}

\newcommand{\cccomment}[1]{} 
\renewcommand{\dd}{\op{\tt{d}}_{\op{gp}}}
\newcommand{\dsimp}{\ensuremath{\op{\delta}_{\op{gp}}}}
\newcommand{\xmod}{\ensuremath{A\xrightarrow{\xmodmap}B}}
\newcommand{\xmoda}{\ensuremath{A}} \newcommand{\xmodb}{\ensuremath{B}}
\newcommand{\xmodmap}{\ensuremath{\mu}}
\newcommand{\psing}{\op{\partial}_{\op{sing}}}

\newcommand{\aone}[1]{\ensuremath{\locsm{\alpha}}\ifthenelse{\equal{#1}{}}{}{\ensuremath{({#1})}}}
\newcommand{\aoneprime}[1]{\ensuremath{\locsm{\alpha}'}\ifthenelse{\equal{#1}{}}{}{\ensuremath{({#1})}}}
\newcommand{\Aone}[1]{\ensuremath{\locsm{A}}\ifthenelse{\equal{#1}{}}{}{\ensuremath{({#1})}}}
\newcommand{\takeCAREofTHIS}{\ensuremath{\locsm{\alpha}}({i^{-1}g})}
\newcommand{\atwo}[1]{\ensuremath{\locsm{\beta}}\ifthenelse{\equal{#1}{}}{}{\ensuremath{({#1})}}}
\newcommand{\atwoprime}[1]{\ensuremath{\locsm{\beta}'}\ifthenelse{\equal{#1}{}}{}{\ensuremath{({#1})}}}

\newcommand{\cd}{\check{\delta}} \newcommand{\bt}{\mathbf{t}} %
\newcommand{\bs}{\mathbf{s}} %
\newcommand{\bid}{\mathbf{id}} %

\renewcommand{\cU}{\mathcal{U}}

\usepackage{xspace} \newcommand{\tgp}{Lie 2-group\xspace}
\newcommand{\tgps}{Lie 2-groups\xspace}

\newcommand{\g}{\ensuremath{\frak{g}}} 
\newcommand{\z}{\ensuremath{\frak{z}}}

\newcommand{\tphi}{\wt{\phi}}

\newcommand{\sSt}{\ensuremath{\cat{sSt}}}
\newcommand{\eSt}{\ensuremath{\cat{\acute{e}St}}}

\begin{document}

%%%%% To ease editing, add:

\baselineskip=17pt

%%%%%%%%%%%%%%%%

%% In the running head, give an abbreviation of the title. 
\titlerunning{Integrating central extensions}

\title{Integrating central extensions of Lie algebras via
  Lie 2-groups}

\author{Christoph Wockel
\and 
Chenchang Zhu}

\date{}

\maketitle

\address{C. Wockel: Fachbereich Mathematik, Universit\"at Hamburg (Germany);  \email{christoph@wockel.eu}
\and
C. Zhu: Mathematisches Institut, Universit\"at G\"ottingen (Germany); \email{chenchang.zhu@gmail.com}}


%%%%%%%%

\begin{abstract}
 The purpose of this paper is to show how central extensions of (possibly
 infinite-dimensional) Lie algebras integrate to central extensions of \'etale
 \tgps in the sense of \cite{Getzler09Lie-theory-for-nilpotent-Linfty-algebras, Henriques08Integrating-Lsb-infty-algebras}. In finite dimensions, central extensions of Lie algebras integrate to
 central extensions of Lie groups, a fact which is due to the vanishing of
 $\pi_{2}$ for each finite-dimensional Lie group. This fact was used by
 Cartan (in a slightly other guise) to construct the simply connected Lie
 group associated to each finite-dimensional Lie algebra.
 
 In infinite dimensions, there is an obstruction for a central extension of Lie
 algebras to integrate to a central extension of Lie groups. This obstruction
 comes from non-trivial $\pi_{2}$ for general Lie groups. We show that this
 obstruction may be overcome by integrating central extensions of Lie algebras
 not to Lie groups but to central extensions of \'etale \tgps. As an
 application, we obtain a generalization of Lie's Third Theorem to
 infinite-dimensional Lie algebras.

%% Keywords are optional
\end{abstract}

\setcounter{tocdepth}{2} \tableofcontents

\section{Introduction}\label{sect:integrating_central_extensions_of_lie_algebras}

\begin{tabsection}
 Central extensions of Lie algebras and their integrability are closely related
 to Lie's Third Theorem. In fact, one can use the integration theory of central
 extensions of Lie algebras
 \cite{Neeb02Central-extensions-of-infinite-dimensional-Lie-groups} to decide
 (under some mild requirements) whether a given Lie algebra is the Lie algebra
 of a Lie group. In finite dimensions, this is always the case due to the
 vanishing of $\pi_{2}$ for each finite-dimensional Lie group. In infinite
 dimensions, $\pi_{2}$ does not always vanish and leads to Lie algebras which
 do not integrate to Lie groups
 \cite{EstKorthagen64Non-enlargible-Lie-algebras,DouadyLazard66Espaces-fibres-en-algebres-de-Lie-et-en-groupes}.
 A similar phenomenon occurs when integrating finite-dimensional Lie algebroids
 to Lie groupoids
 \cite{Pradines68Troisieme-theoreme-de-Lie-les-groupoi-des-differentiables,CrainicFernandes03Integrability-of-Lie-brackets}.
 In this case it is $\pi_{2}$ of the leaves that restricts the integrability of a
 Lie algebroid.

 On the other hand, the theory of higher Lie group(oid)s has been much
 developed recently. Already in the early twentieth century, 2-groups were
 studied by Whitehead and his followers under various terms, such as crossed
 modules. They are also studied from the perspective of ``gr-champ'' (i.e.
 stacky groups) by Breen \cite{Breen90Bitorseurs-et-cohomologie-non-abelienne}. More recently, various versions of
 2-groups, with different strictness assumptions, have been studied as models
 for the string group
 \cite{BaezCransStevensonSchreiber07From-loop-groups-to-2-groups,Schommer-Pries10Central-Extensions-of-Smooth-2-Groups-and-a-Finite-Dimensional-String-2-Group}
 using a method of categorification, initiated by Baez’ school
 \cite{BaezLauda04Higher-dimensional-algebra.-V.-2-groups}. However to treat
 all higher group(oid)s in various categories all together, the most efficient
 method is to apply Duskin and Glenn's idea of Kan complexes
 \cite{Duskin79Higher-dimensional-torsors-and-the-cohomology-of-topoi:-the-abelian-theory,Glenn82Realization-of-cohomology-classes-in-arbitrary-exact-categories}. This method has the advantage that it easily
 gives the concept of Lie $n$-group(oid)s for all $n\in \N$ if we take our
 category to be the one of manifolds, or the concept of topological
 $n$-group(oid)s for all $n\in \N$ if we change our category to be the one of
 topological spaces (see e.g. \cite{Henriques08Integrating-Lsb-infty-algebras,
 Zhu09n-groupoids-and-stacky-groupoids} for such a treatment and the additional
 requirement of the taken categories). This allows us to treat many theories
 uniformly without repeating proofs. Moreover, starting from Getzler's work on
 integration of nilpotent $L_\infty$-algebras \cite{Getzler09Lie-theory-for-nilpotent-Linfty-algebras}, this sort of Lie
 $n$-group(oid)s have been widely used in many integration problems, for
 example the integration of general $L_\infty$-algebras of Henriques
 \cite{Henriques08Integrating-Lsb-infty-algebras}, the integration of finite dimensional Lie algebroid
 \cite{TsengZhu06Integrating-Lie-algebroids-via-stacks}, the integration of
 Courant algebroid \cite{Li-BlandSevera11Integration-of-Exact-Courant-Algebroids,MehtaTang10From-double-Lie-groupoids-to-local-Lie-2-groupoids,ShengZhu11Higher-Extensions-of-Lie-Algebroids-and-Application-to-Courant-Algebroids}. Thus people have a general belief
 that these Lie $n$-group(oid)s are the correct objects of a certain (higher)
 category that corresponds, via integration, to the one of various
 infinitesimal objects, for example, $L_\infty$-algebras,
 $L_\infty$-algebroids.

 Using this sort of Lie 2-groups, we study in this article the integration of another sort of
 infinitesimal object, namely infinite-dimensional Lie algebras and their
 central extensions. We obtain a version of Lie's Third Theorem
 asserting that each locally exponential Lie algebra (see Definition
 \ref{def:loc-exp}) with topologically split center integrates to an \'etale
 \tgp. The same question was studied in
 \cite{Wockel08Categorified-central-extensions-etale-Lie-2-groups-and-Lies-Third-Theorem-for-locally-exponential-Lie-algebras}
 by a completely different and less powerful concept of Lie 2-group, since it
 only admits a notion of smoothness ``near the identity''. Note that this
 concept of ``locally smooth'' Lie 2-group is only known to be equivalent to
 the one mentioned above in very special cases
 \cite{WagemannWockel13A-Cocycle-Model-for-Topological-and-Lie-Group-Cohomology}
 that do not govern the situation we have in this paper.
 The value of the current article is that it extends the result of
 \cite{Wockel08Categorified-central-extensions-etale-Lie-2-groups-and-Lies-Third-Theorem-for-locally-exponential-Lie-algebras}
 to a {\em global} one. One can do this because we weaken (in a certain sense)
 the category of ``smooth 2-spaces'' from
 \cite{Wockel08Categorified-central-extensions-etale-Lie-2-groups-and-Lies-Third-Theorem-for-locally-exponential-Lie-algebras},
 which is nothing but the category of Lie groupoids with strict
 morphisms, to the bicategory of smooth stacks which is equivalent to
 the one of Lie groupoids with generalized
 morphisms and 2-morphisms. To obtain
 our result we then have enhanced the approach from
 \cite{Wockel08Categorified-central-extensions-etale-Lie-2-groups-and-Lies-Third-Theorem-for-locally-exponential-Lie-algebras}
 significantly because generalized morphisms are more involved than
 the strict ones. This effort is eventually rewarded by obtaining a
 globally smooth object integrating an infinite-dimensional Lie
 algebra.
\emptycomment{One can do this because we weaken (in a certain sense)
 the concept of  ``smooth 2-space'' from
 \cite{Wockel08Categorified-central-extensions-etale-Lie-2-groups-and-Lies-Third-Theorem-for-locally-exponential-Lie-algebras},
 which is nothing but a Lie groupoid, to the one of a smooth stack. To obtain
 our result we then have enhanced the approach from
 \cite{Wockel08Categorified-central-extensions-etale-Lie-2-groups-and-Lies-Third-Theorem-for-locally-exponential-Lie-algebras}
 significantly because morphisms between smooth stacks are more involved than
 the ones of Lie groupoids. This effort is eventually rewarded by obtaining a
 globally smooth object integrating an infinite-dimensional Lie algebra.}
\\
 
 We now go into some more detail. The basic idea behind the integration
 processes from
 \cite{Neeb02Central-extensions-of-infinite-dimensional-Lie-groups} (and
 \cite{CrainicFernandes03Integrability-of-Lie-brackets}) is to integrate
 prescribed curvature 2-forms along certain triangles (respectively homotopies
 between paths). Suppose that $\g$ is the Lie algebra of a simply connected Lie
 group $G$ and that $\omega\from \fg\times\fg\to\fz$ is a continuous Lie
 algebra cocycle. If the \emph{period homomorphism}
 $\pi_2(G) \xrightarrow{\per_{\omega}} \z$ (see \eqref{eq:period}) has discrete
 image, then $\omega$ has an integrating cocycle in the locally smooth Lie
 group cohomology $H^2_{\op{loc}}(G, \z/\per_{\omega}(\pi_2(G)))$, i.e., the
 differentiation homomorphism
 \begin{equation*}
  D\from  H^2_{\op{loc}}(G, \z/\per_{\omega}(\pi_2(G)))\to H_{c}^{2}(\fg,\fz)
 \end{equation*}
 has the class of $\omega$ in its range \cite[Corollary
 6.3]{Neeb02Central-extensions-of-infinite-dimensional-Lie-groups}. This is
 shown by integrating $\omega$ along some carefully chosen triangles, an idea
 which dates back to van Est
 \cite{Est58A-group-theoretic-interpretation-of-area-in-the-elementary-geometries.}.
 The procedure then reveals the obstruction against integration as a cocycle
 condition, which may also be viewed as an associativity constraint for
 enlarging a local group to a global one (cf.\
 \cite{Smith51The-complex-of-a-group-relative-to-a-set-of-generators.-I,Smith51The-complex-of-a-group-relative-to-a-set-of-generators.-II,Est62Local-and-global-groups.-I,Est62Local-and-global-groups.-II}).
 So one is naturally pushed to non-associative structures when searching for a
 general solution of the integration problem. Our \tgps are such structures,
 which provide at the same time the next higher coherence that the problem
 naturally has (cf.\ the discussion in \cite[Section
 2]{Wockel08Categorified-central-extensions-etale-Lie-2-groups-and-Lies-Third-Theorem-for-locally-exponential-Lie-algebras}).

 In this paper, we deal with the case when $\per_{\omega}(\pi_2(G))$ is {\em
 not} discrete. In this case, $\z/\per_{\omega}(\pi_2(G))$ does not exist as a
 Lie group any more. One of the natural substitutes for it is the \tgp
 $[\pi_{2}(G)\xrightarrow{\per_{\omega}}\fz]$\footnote{Another substitute would
 be the diffeological group $\fz/\per_{\omega}(\pi_{2}(G))$, see \cite[Remark
 7.1]{Wockel08Categorified-central-extensions-etale-Lie-2-groups-and-Lies-Third-Theorem-for-locally-exponential-Lie-algebras}}
 (see Example \ref{ex:Lie-2-groups}), which exists regardless of the
 discreteness of $\per_{\omega}(\pi_2(G))$. However, if
 $\per_{\omega}(\pi_2(G))$ is discrete, then
 $\fz/\per_{\omega}(\pi_2(G))$ is equivalent to a direct factor of
 $[\pi_{2}(G)\xrightarrow{\per_{\omega}}\fz]$  and thus
 $[\pi_{2}(G)\xrightarrow{\per_{\omega}}\fz]$ is the universal object taking
 over the r\^ole of $\z/\per_{\omega}(\pi_2(G))$. Moreover, extensions of $G$
 by $[\pi_{2}(G)\xrightarrow{\per_{\omega}}\fz]$ live in the category of \tgps
 and thus permit us to incorporate the non-associativity that we mentioned above.
 In this sense our treatment is a natural extension of the 
 procedure from
 \cite{Neeb02Central-extensions-of-infinite-dimensional-Lie-groups}. The price
 to pay for this freedom is that one has to work with group objects in smooth
 stacks (aka \tgps\footnote{Notice that when $n=2$, Lie $n$-groups described
 via Kan complexes that we mentioned earlier are proven
 \cite{Zhu09n-groupoids-and-stacky-groupoids} to be equivalent to group objects
 in smooth stacks. }) instead of group objects in smooth manifolds (aka Lie
 groups). This is technically more challenging but has similar underlying
 ideas.\\
 
 Section \ref{sec:cohomology} is concerned with setting up the theory of
 central extensions of Lie groups by abelian \tgps
 $[A\xrightarrow{\xmodmap}B]$, in particular to show how such extensions can be
 obtained from certain \v{C}ech cohomology classes. This section builds heavily
 on \cite[Section
 3]{Schommer-Pries10Central-Extensions-of-Smooth-2-Groups-and-a-Finite-Dimensional-String-2-Group}.
 It provides the conceptual background for understanding the constructions in
 the next section.
 
 Section \ref{sec:geo-cocycle} then presents the refinement of the
 aforementioned idea of integration along triangles. The problem that one has
 to overcome is that the cocycle condition that one has in
 $\fz/\per_{\omega}(\pi_{2}(G))$ makes many arguments implicitly work. One
 example for this is the smoothness of the multiplication of the central
 extension of Lie groups associated to an integrating cocycle for $\omega$ as
 discussed in Section \ref{sect:interpretation_of_the_construction_i}. All
 these implications now have to be built into the choices of the triangles, and
 this is the key point of Section \ref{sec:geo-cocycle}. In a certain sense,
 the essence of this construction is subsumed in Figures \ref{fig:Dreieck} and
 \ref{fig:Kugel}.
 
 Section \ref{sec:lie-third} then provides the differentiation process which
 justifies naming the construction of the previous section ``integration''. We
 restrict in this treatment to \'etale \tgps, for which the differentiation
 leads to ordinary Lie algebras. The main result on this is the following
 \begin{nntheorem}
  If $G$ is a simply connected Lie group with Lie algebra $\fg$, $\fz$ is a
  Mackey-complete locally convex space and $\omega\from \fg\times\fg\to\fz$ is
  a continuous Lie algebra cocycle. Then the differentiation homomorphism
  \begin{equation*}
   D\from  \op{Ext}(\ul{G},[\pi_{2}(G)\xrightarrow{\per_{\omega}}\fz])\to \op{Ext}(\fg,\fz)\cong H_{c}^{2}(\fg,\fz)
  \end{equation*}
  has $[\omega]$ in its image. Here $\ul{G}$ is the Lie group $G$ viewed as a
  Lie 2-group (see Example \ref{ex:Lie-2-groups}).
 \end{nntheorem}
 This then implies readily the following generalization of Lie's Third Theorem
 that our construction allows for.
 \begin{nntheorem}
  If $\fg$ is a locally convex locally exponential Lie algebra such that
  $\fz:=\fz(\fg)\se\fg$ is a complemented and Mackey-complete subspace, then
  there exists an \'etale \tgp $G$ with $L(G)\cong \fg$.
 \end{nntheorem}
 In the end, we provide some background on infinite-dimensional manifolds and
 their derived concepts of Lie groups, Lie groupoids and smooth stacks.

\end{tabsection}

\emptycomment{ In the setting of this paper, his results boil down to an exact
sequence
\begin{equation*}
 0\to H^{2}_{loc}(G,Z)\xrightarrow{D} H^{2}_{c}(\fg,\fz)\xrightarrow{[\omega]\mapsto
 (q{\circ}\per_{\omega})}\Hom(\pi_{2}(G),Z)
\end{equation*}
In particular, this shows that if we are interested in whether there exists a
locally smooth cocycle $\locsm{f}$ (for some choice of $\Gamma\se\fz$) with
$D[\wt{f}]=[\omega]$, then this is the case if and only if
$\Per{\omega}=\per_{\omega}(\pi_{2}(G))\se \fz$ is discrete. How to construct
such a cocycle in the case of discrete $\Per{\omega}$ will be explained in more
detail in the end Section \ref{sect:locally_smooth_cocycles}.}

\section*{Conventions}

\begin{tabsection}
 Unless stated otherwise, $G$ denotes throughout a 1-connected Lie group,
 modeled on a locally convex space, and $\fg$ denotes its Lie algebra.
 Moreover, $\fz$ stands for a Mackey-complete locally convex vector
 space\footnote{Mackey-complete locally convex spaces also run under the name
 convenient vector space, in particular each complete locally convex space is
 of this type, cf.\ \cite[Theorem
 I.2.14]{KrieglMichor97The-Convenient-Setting-of-Global-Analysis}},
 $\Gamma\se\fz$ is a discrete subgroup so that $Z:=\fz/\Gamma$ is a Lie group
 with universal covering morphism $q\from\fz\to Z$. In addition,
 $\omega \from \fg\times\fg\to\fz$ will always denote a continuous Lie algebra
 cocycle. Associated to this data is the \emph{period homomorphism}
 \begin{equation} \label{eq:period} \per_{\omega}\from
  \pi_{2}(G)\to \fz,\quad [\sigma]\mapsto \int_{\sigma}\omega^{l},
 \end{equation}
 where $\omega^{l}$ is the left-invariant 2-form on $G$ with
 $\omega^{l}(e)=\omega$ (cf.\ \cite[Section
 5]{Neeb02Central-extensions-of-infinite-dimensional-Lie-groups}). Note that
 under our assumptions $\pi_{2}(G)\cong H_{2}(G)$ by the Hurewicz homomorphism
 and we identify $\pi_{2}(G)$ with $H_{2}(G)$ throughout.
 
 We denote by $\Delta^{(n)}\se\R^{n}$ the standard $n$-simplex, viewed as a
 manifold with corners. By $C^{\infty}(\Delta^{(n)},G)$ we mean the manifold of
 smooth $n$-simplices in $G$ (see also Proposition
 \ref{prop:smoothness_of_integration}) and by $C^{\infty}_{*}(\Delta^{(n)},G)$
 those smooth $n$-simplices that are base-point preserving maps, where the
 base-point of $\Delta^{(n)}$ is $0$ and the base-point of $G$ is the identity.
 For a simplicial complex $\Sigma$ we will denote by
 $C^{\infty}_{pw}(\Sigma,G)$ the piece-wise smooth maps (cf.\ Remark
 \ref{rem:piece-wise-smooth-maps}). The simplicial manifold that shall play an
 important rôle in this paper is the classifying simplicial space
 $\bB G_{\bullet}=(G^{i})_{i\in \N_{0}}$ (with the product smooth structure and
 the convention $G^{0}:=*$) and the standard simplicial maps
 $p_{i}^{(n)}\from \bB G^{(n)}\to \bB G^{(n-1)}$ (cf.\ Example \ref{ex:nerve}).
 Moreover, $\mu\from A\to B$ is always a morphism between the abelian Lie
 groups $A$ and $B$.
\end{tabsection}

\section{Differentiable hypercohomology and its geometric correspondence}
\label{sec:cohomology}

\subsection{Differentiable hypercohomology}
\label{sect:differentiable_hypercohomology}

\begin{tabsection}
 The hypercohomology of complexes of sheaves on manifolds, action groupoids and
 complex stacks is explicitly studied for instance in
 \cite{Brylinski93Loop-spaces-characteristic-classes-and-geometric-quantization,
 Gomi05Equivariant-smooth-Deligne-cohomology} and \cite[\S
 A.2]{FelderHenriquesRossiZhu08A-gerbe-for-the-elliptic-gamma-function}. Here
 we extend it to the category of simplicial manifolds and relate it to our
 construction using a suitable covering constructed in Section
 \ref{sec:geo-cocycle}. We emphasize the \v{C}ech approach to differentiable
 hypercohomology and we are mostly interested in the simplicial manifold
 $\bB G_\bullet$ (i.e., the nerve of the Lie groupoid $G\rrarrow *$, see Example
 \ref{ex:nerve}).
 
 Recall that a {\em simplicial manifold} $X_\bullet$ is a functor
 $\Delta^{\op{op}}\to \cat{Man}$, where $\Delta$ is the standard simplex
 category of finite ordinal numbers $[n]$ and non-decreasing maps $[n]\to [m]$.
 This has the alternative description as a collection of manifolds $X_n$ and
 structure maps
 \begin{equation}\label{eq:fd}
  d^n_k: X_n \to X_{n-1} \;\text{(face maps)}\quad s^n_k: X_n \to X_{n+1} \; \text{(degeneracy maps)},\;\; k\in \{0, 1, 2,..., n\}
 \end{equation}
 that satisfy the usual coherence conditions (see for instance \cite[Chapter
 I.1]{GoerssJardine99Simplicial-homotopy-theory}).
\end{tabsection}

\begin{example}\label{ex:nerve}
 \begin{itemize}
  \item [a)] We can interpret each manifold as \emph{constant simplicial space}
        with $X_{n}=M$ for all $n$ and all structure maps to be the identity.
  \item [b)] Given a Lie groupoid\footnote{For a Lie groupoid we require source
        and target map to be surjective submersions in the sense of Appendix
        \ref{sec:grothendieck-pretop}}
        $\cG:= (G_1 \underset{\bt}{\overset{\bs}{\rrarrow}} G_0) $, we complete
        it to a simplicial manifold $\bB\cG_\bullet$, with
        \[ 
         \bB\cG_n = G_1 \times_{\bs, G_0, \bt} G_1 \times_{ \bs, G_0, \bt} \dots
         \times_{ \bs, G_0, \bt} G_1, \quad \text{($n$ copies of $G_1$)},
        \]
        for $n\ge 1$ and $\bB G_0=G_0$. The face maps, for $n\ge 2$, are given by
        \begin{equation*}
         d_k^{n}(g_1,{\dots},g_n ) =  
         \begin{cases}
         (g_2, {\dots}, g_n) & k=0 \\
         (g_1, {\dots}, g_k g_{k+1}, {\dots}, g_n) & 0<k<n \\
         (g_1, {\dots}, g_{n-1})  & k=n,
         \end{cases}
        \end{equation*}
        and the degeneracy maps, for $n \ge 1$, by
        \begin{equation*}
         s_k^{n}(g_1, \dots, g_n) =
         \begin{cases} (1_{\bt(g_1)}, g_1.
         \dots, g_n) & k=0 \\
         (g_1, \dots, g_k,  1_{\bt(g_{k+1})}, g_{k+1}, \dots, g_n) & 0<k<n, \\
         (g_1, \dots, g_n, 1_{\bs(g_n)}) & k=n.
         \end{cases}
        \end{equation*}
        Moreover, $d^1_0(g)=\bs(g)$, $d^1_1(g)=\bt (g)$, $s_0^0(x)=e(x)$, where
        $e: G_0\to G_1$ is the identity embedding.
        This construction is known as the \emph{nerve} of the Lie groupoid
        $\cG$. We call it $\bB\cG_\bullet$ because its geometric realization is
        the classifying space of $\cG$
        \cite{Segal68Classifying-spaces-and-spectral-sequences}. If
        $\cG=(G\rrarrow *)$, then we also denote $\bB\cG_{\bullet}$ by
        $\bB G_{\bullet}$.
 \end{itemize}
\end{example}

\begin{tabsection}
 For a simplicial manifold $X_\bullet$, a sheaf $\cF^\bullet$ on $X_\bullet$
 consists of sheaves $\cF^n$ on $X_n$ for all $n$ and morphisms
 $F^{\bullet}(\alpha)\from X_{\bullet}(\alpha)^{*}\cF^{n}\to \cF^{m}$ for each
 $\alpha\from [n]\to [m]$ such that
 $F^{\bullet}(\alpha {\circ} \beta)=F^{\bullet}(\alpha){\circ} F^{\bullet}(\beta)$
 \cite[\S 5.1.6]{Deligne74Theorie-de-Hodge.-III}. Alternatively these morphisms
 can also be described by morphisms
 $D^{n}_{k}\from (d_{k}^{n})^{*}\cF^{n-1}\to\cF^{n}$ and
 $S_{k}^{n}\from (s_{k}^{n})^{*}\cF^{n+1}\to\cF^{n}$ satisfying the
 corresponding compatibility conditions. Likewise, we define morphisms of
 sheaves as in \cite[\S 5.1.6]{Deligne74Theorie-de-Hodge.-III}. This then leads
 to the notion of a (bounded below) complex of sheaves on $X_{\bullet}$
 \begin{equation}\label{eq:sheaf-cx}
  \cF^{\bullet}_{*}=((\cF^{\bullet}_{n})_{n\in \N_{0}}, \cF_{n}^{\bullet}\xrightarrow{d_{n}}\cF^{\bullet}_{n+1}),
 \end{equation}
 (see also \cite[\S 3.2]{Gomi05Equivariant-smooth-Deligne-cohomology}).
\end{tabsection}

\begin{tabsection}
 A covering $\cU$ of a simplicial manifold $X_\bullet$ consists of a simplicial
 set $I^{\bullet}$ and a covering $(U^{(n)}_{i})_{i\in I^{(n)}}$ of $X_{n}$
 such that
 $X_{\bullet}(\alpha)(U^{(m)}_{i})\se U^{(n)}_{I^{\bullet}(\alpha)(i)}$ for
 each $\alpha\from [n]\to [m]$. One can demand less structure for a covering of
 a simplicial space, see
 \cite{Schommer-Pries10Central-Extensions-of-Smooth-2-Groups-and-a-Finite-Dimensional-String-2-Group}
 or
 \cite{WagemannWockel13A-Cocycle-Model-for-Topological-and-Lie-Group-Cohomology}.
 We demand all this structure for making the normalization arguments later on
 work. In particular, a covering induces another simplicial space
 $\cU_{\bullet}$ with $\cU_{n}:=\coprod_{i\in I^{n}}U^{(n)}_{i}$ and the
 inclusions induce a simplicial map $\cU_{\bullet}\to X_{\bullet}$. The
 following lemma shows that one can always extend coverings of $X_{n}$ to
 coverings of $X_{\bullet}$.
\end{tabsection}
	
\begin{lemma}\label{lem:refinements_to_simplicial_covers}
 If $X_{\bullet}$ is a simplicial manifold and $(U_{j})_{j\in J}$ is a covering of
 $X_{m}$, then there exists a covering $\cU$ of $X_{\bullet}$ such that
 $(U^{(m)}_{i\in I^{(m)}})$ is a refinement of $(U_{j})_{j\in J}$.
\end{lemma}

\begin{proof}
 We denote $\Delta(m,n):=\Hom_{\Delta}([m],[n])$ (note that this is a finite
 set). We first observe that $J$ determines a simplicial set $I^{\bullet}$ with
 $I^{(n)}:=J^{\Delta(m,n)}$ and with $\alpha\from [n]\to [n']$ getting mapped
 to
 \begin{equation*}
  \alpha^{\Delta}\from J^{\Delta(m,n')} \to J^{\Delta(m,n)},\quad 
  \alpha^{\Delta}((j_{f})_{f\in \Delta(m,n')})_{g}= j_{\alpha{\circ} g}.
 \end{equation*}
 Now for each $f\in \Delta(m,n)$ we pull back the covering $(U_{j})_{j\in J}$
 of $X_{m}$ to a covering of $X_{n}$ and take as a covering of $X_{n}$ the
 coarsest common refinement $(U_{i}^{(n)})_{i\in I^{(n)}}$ with
 \begin{equation*}
  U_{i}^{(n)}:= X_{\bullet}(f_{0})^{-1}(U_{j_{f_{0}}})\cap ...\cap X_{\bullet}(f_{s})^{-1}(U_{j_{f_{s}}})
 \end{equation*}
 (where $i=(j_{f_{0}},...,j_{f_{s}})$ with $s=|\Delta(m,n)|$) of all such
 coverings. To check that this is indeed a covering we have to show that for
 each $\alpha\from [n']\to [n]$ and each $i=(j_{f})_{f\in \Delta([m],[n])}$ we
 have
 \begin{align*}
  &X_{\bullet}(\alpha)(U^{(n)}_{i})\se U^{(n')}_{\alpha^{\Delta}(i)}\\\Leftrightarrow~ 
  &X_{\bullet}(\alpha)\big( X_{\bullet}(f_{0})^{-1}(U_{j_{f_{0}}})\cap\dots\cap X_{\bullet}(f_{s})^{-1}(U_{j_{f_{s}}}) \big)\se 
  X_{\bullet}(f)^{-1}(U_{j_{\alpha {\circ} f}}) \;\forall f\in \Delta(m,n')\\\Leftrightarrow~
  &X_{\bullet}(\alpha {\circ} f)\big( X_{\bullet}(f_{0})^{-1}(U_{j_{f_{0}}})\cap\dots \cap X_{\bullet}(f_{s})^{-1}(U_{j_{f_{s}}}) \big)\se
  U_{j_{\alpha {\circ} f}} \;\forall f\in \Delta(m,n').
 \end{align*}
 The latter is true since for each $f\in \Delta(m,n')$ we have that
 $\alpha {\circ} f =f_{i}$ for some $i$. To complete the proof we observe
 that the canonical map $I^{(m)}\to J$,
 $(j_{f})_{f\in \Delta(m,m)}\mapsto j_{\id_{[m]}}$ induces the corresponding
 refinement.
\end{proof}

\begin{tabsection}
 Let now $\cU$ be a covering of $X_{\bullet}$. Then we set
 \begin{equation}\label{hyper-check-tri-cx}
  {\check{C}}^{p, q, r}:= \prod_{i_{0},...,i_{q}\in I^{(p)}} \cF_{r}^{p}(U^{(p)}_{i_{0}}\cap...\cap U^{(p)}_{i_{q}}).
 \end{equation}
 We have $\dsimp(:=d_{1,0,0}): \check{C}^{p-1,q,r}\to \check{C}^{p, q,
   r}$ \cccomment{instead of $\check{C}^{p,q,r}\to \check{C}^{p+1, q,
     r}$ this causes the sign later on for $\phi_{3,0,-1}$.}
 defined by
 \begin{equation} \label{eq:delta}
  (\dsimp f)_{i_{0},...,i_{q}}= \sum_{k=0}^{p}(-1)^{k+p} D_{k}^{p}(\underbrace{f_{d_{k}^{p}(i_{0}),...,d_{k}^{p}(i_{q})}{\circ} d_{k}^{p}}_{\in (d_{k}^{p})^{*}\cF^{p-1}})
 \end{equation}
 and $\cd(:=d_{0,1,0}): \check{C}^{p,q,r} \to \check{C}^{p, q+1, r}$ defined by
 the \v{C}ech
 differential\begin{equation}\label{de:check-delta}
  (\check{\delta} f)_{i_{0},...,i_{q+1}}(x)=
  \sum_{l=0}^{q+1}(-1)^{l}f_{i_{0},...,\wh{i_{l}},...,i_{q+1}}(x).
 \end{equation}
 
 There is another differential
 $d(:=d_{0,0,1}): \check{C}^{p,q,r} \to \check{C}^{p,q,r+1}$, induced by the
 differential $d_{r}\from\cF_r \to \cF_{r+1}$ in the sheaf complex. Then
 $(\check{C}^{p,q,r}, \dsimp, \cd, d)$ is a triple complex and the total
 complex is $\check{C}^{N} :=\oplus_{N=p+q+r} \check{C}^{p,q,r}$ with the total
 differential $D_3= \dsimp + (-1)^{p} \cd + (-1)^{p+q} d$.
\end{tabsection}

\begin{definition}\label{def:hypercohomology}
 The \v Cech hypercohomology ${\check{H}}^n_\cU(X_\bullet, \cF^{\bullet}_{*})$
 of the complex of sheaves $\cF^{\bullet}_{*}$ on $X_{\bullet}$ with respect to
 the covering $\cU$ is the cohomology of the total complex of
 $(\check{C}^{p,q,r}, \dsimp, \cd, d)$. The group of $n$-cocycles of this
 triple complex is denoted
 $\check{Z}^{n}_{\cU}(X_{\bullet},\cF^{\bullet}_{*})$
 and its elements are called \emph{differentiable cocycles}.
\end{definition}

\begin{tabsection}
 On the other hand, one can define sheaf hypercohomology
 $H^n_{sh} (X_\bullet,\cF^{\bullet}_{*})$ to be the hyper derived functor of
 the functor
 \begin{equation*}
  \cF^{\bullet}_{*}\mapsto H^{0}\big( \Gamma(\cF^{\bullet}_{*}) \big),
 \end{equation*}
 where the section functor $\Gamma$ for a sheaf $\cF^{\bullet}$ on $X_{\bullet}$ is
 given by
 $\cF^{\bullet}\mapsto \ker(D_{0}-D_{1}\from \cF^{0}(X_{0})\to\cF^{1}(X_{1}))$.
 The hyper derived functor is then given by the usual construction using
 injective resolutions as in \cite[Chapter
 2]{Friedlander82Etale-homotopy-of-simplicial-schemes}, see also
 \cite{Deligne74Theorie-de-Hodge.-III,Conrad03Cohomological-Descent}. The
 following proposition generalizes the case of the relation between \v Cech
 cohomology and sheaf cohomology. Recall that a covering $\cU$ of $X_\bullet$
 is called $\cF$-{\em acyclic} if
 $H^{\geq 1} (U^{(n)}_{i_{0}}\cap...\cap U^{(n)}_{i_{q}}, \cF^n_r)=0$ for all
 finite subsets $\{i_{0},...,i_{q}\}\se I^{(n)}$ and all $n, r$. The reasoning
 of \cite[\S
 A.2]{FelderHenriquesRossiZhu08A-gerbe-for-the-elliptic-gamma-function} then
 carries over to show
\end{tabsection}

\begin{Proposition}\label{prop:deligne-cech}
 In the above setup, there is a morphism of abelian groups
 \[
  {\check{H}}^n_\cU(X_\bullet, \cF^{\bullet}_{*}) \to H^n_{sh}(X_\bullet, \cF^{\bullet}_{*}).
 \]
 In particular, if $\cU$ is an $\cF_{r}$-acyclic covering of $\cX$ for each $r$, then
 ${\check{H}}^n_\cU(\cX, \cF^{\bullet}_{*}) \cong H^n_{sh}(\cX, \cF^{\bullet}_{*})$.
\end{Proposition}

\begin{tabsection}
 Thus the \v Cech hypercohomology of $\cF^{\bullet}_{*}$, defined as the direct
 limit
 \begin{equation*}
  \check{H}^n(X_\bullet, \cF^{\bullet}_{*}):=  
  \underset{\longrightarrow}{\lim} \check{H}^n_\cU(X_\bullet,\cF^{\bullet}_{*}),
 \end{equation*}
 is isomorphic to the sheaf cohomology $H^n_{sh}(X_\bullet,\cF^{\bullet}_{*})$
 if each covering admits an $\cF$-acyclic refinement. When the existence of
 acyclic coverings is not guaranteed, one needs to take the limit over all
 hyper-coverings but not only coverings of $X_\bullet$ as explained in
 \cite{Friedlander82Etale-homotopy-of-simplicial-schemes}. Then the same result
 holds. However, each covering is in particular a hyper-covering and all our
 constructions will yield cocycles on usual coverings. In addition, the
 equivalences of cocycles that we construct will also live on usual coverings.
 Thus our constructions will lead to well-defined classes in \v{C}ech
 cohomology and thus also in sheaf cohomology.\\
 
 There is one additional condition on \v{C}ech cocycles, that simplifies
 computations a lot, which is the assumption that they are normalized. For this
 we consider
 \begin{equation*}
  \cF_{r,0}^{p}(U^{(p)}_{i_{0}}\cap...\cap U^{(p)}_{i_{q}}):=
  \bigcap _{l=0}^{q}  \ker((\check{s}^{q}_{l})^{*}),
 \end{equation*}
 with
 \begin{equation*}
  \check{s}^{q}_{l}\from U_{i_{0}}^{(p)}\cap...\cap U_{i_{q}}^{(p)}\xrightarrow{\cong} U^{(p)}_{i_{0}}\cap... \cap U_{i_{l}}\cap U_{i_{l}}\cap   ...\cap U^{(p)}_{i_{q+1}},
  \quad
  x_{i_{0},...,i_{q}}\mapsto x_{i_{0},...,i_{l-1},i_{l},i_{l},i_{l+1},...,i_{q}}
 \end{equation*}
 the standard \v{C}ech degeneracy maps and
 \begin{equation*}
  \check{C}_{0}^{p,q,r}:=\bigcap_{k=0}^{p-1}\ker(\sigma_{k})\cap \prod _{i_{0},...,i_{q}\in I^{(p)}}\cF_{r,0}^{p}(U^{(p)}_{i_{0}}\cap...\cap U^{(p)}_{i_{q}})
 \end{equation*}
 with
 \begin{equation*}
  \sigma_{k}\from \check{C}^{p,q,r} \to \check{C}^{p-1,q,r},\quad
  \sigma_k( f)_{i_{0},...,i_{q}}= S_{k}^{p-1}(f_{s_{k}^{p-1}(i_{0}),...,s_{k}^{p-1}(i_{q})} {\circ} s_{k}^{p-1} )
 \end{equation*}
 the degeneracy map induced by the degeneracies
 $s_{k}^{p}\from X_{p}\to X_{p+1}$ and
 $S_{k}^{p-1}\from (s_{k}^{p-1})^{*}\cF^{p}\to\cF^{p-1}$. It is clear that
 $(\check{C}_{0}^{p,q,r},\dsimp,\check{\delta},d)$ is a sub triple complex of
 $(\check{C}^{p,q,r},\dsimp,\check{\delta},d)$ and thus the cohomology of the
 associated total complex
 \begin{equation*}
  \check{H}^{n}_{0,\cU}(X_{\bullet},\cF^{\bullet}_{*})
 \end{equation*}
 comes equipped with a natural morphism
 \begin{equation*}
  \check{H}^{n}_{0}(X_{\bullet},\cF^{\bullet}_{*}):=\underset{\longrightarrow}{\lim}
  \check{H}^{n}_{0,\cU}(X_{\bullet},\cF^{\bullet}_{*})\to
  \check{H}^{n}(X_{\bullet},\cF^{\bullet}_{*})
 \end{equation*}
 from the normalized to the ordinary \v{C}ech cohomology. The normalized
 \v{C}ech cohomology is what we will work with in this article. For its
 conceptual interpretation we will first show that it actually agrees with the
 non-normalized \v{C}ech cohomology and thus with sheaf cohomology in many
 interesting cases.
\end{tabsection}

\begin{proposition}
 Suppose that $\cA=$
 \begin{equation*}
  \vcenter{\xymatrix{
  && A^{p,q+1}
  \ar@<1ex>[d]^{{\ol{i}_{k}^{p,q}}}
  \\
  A^{p-1,q} 
  \ar@<1ex>[rr]^{d^{p,q}_{k}} && 
  A^{p,q} 
  \ar@<1ex>[u]^{\ol{d}_{k}^{p,q+1}}
  \ar@<1ex>[d]^{{\ol{i}_{k}^{p,q-1}}}
  \ar@<1ex>[ll]^{i_{k}^{p-1,q}} 
  \ar@<1ex>[rr]^{d^{p+1,q}_{k}} && 
  A^{p+1,q}
  \ar@<1ex>[ll]^{i_{k}^{p,q}}
  \\
  && A^{p,q-1}
  \ar@<1ex>[u]^{\ol{d}_{k}^{p,q}}
  }}
 \end{equation*}
 is a bi-cosimplicial abelian group. Then the normalized cochains
 \begin{equation*}
  A_{0}^{p,q}=\bigcap_{k=0}^{p-1}\ker(i_{k}^{p-1,q})\cap\bigcap_{k=0}^{q-1}\ker{\ol{i}_{k}^{p,q-1}}
 \end{equation*}
 form a sub bi-cosimplicial abelian group $\cA_{0}$ and the inclusion of the
 associated total complexes $\op{Tot}(\cA_{0})\to \op{Tot}(\cA)$ induces an
 isomorphism in cohomology.
\end{proposition}

\begin{proof}
 We first define the vertically normalized cochains $\cA_{v}$ to be
 $A_{v}^{p,q}=\cap_{k=0}^{p-1}\ker(i_{k}^{p-1,q})$, which is also a sub
 bi-cosimplicial abelian group of $\cA$. Likewise, we define the horizontally
 normalized cochains $\cA_{h}$ to be given by
 $A_{h}^{p,q}=\cap_{k=0}^{q-1}\ker(\ol{i}_{k}^{p,q-1})$. Observe that
 $\cA_{0}=(\cA_{v})_{h}$. By the dual Dold-Kan correspondence \cite[Corollary
 8.4.3]{Weibel94An-introduction-to-homological-algebra} the cochain complex of
 a cosimplicial abelian group has the same cohomology as  its
 normalized sub complex. Thus $\op{Tot}(\cA_{v})\to \op{Tot}(\cA)$ and
 $\op{Tot}((\cA_{v})_{h})\to \op{Tot}(\cA_{v})$ induce isomorphisms in
 cohomology.
\end{proof}
 
\begin{Corollary}\label{cor:normalization}
 If $X_{\bullet}$ is a simplicial manifold and $\cF^{\bullet}_{*}$ is a complex
 of sheaves on $X_{\bullet}$, then the canonical morphism
 \begin{equation}\label{eqn:morphism_of_delta_functos}
  \check{H}^{n}_{0}(X_{\bullet},\cF^{\bullet}_{*})\to   
  \check{H}^{n}(X_{\bullet},\cF^{\bullet}_{*})
 \end{equation}
 is an isomorphism.
\end{Corollary}

\begin{remark}\label{rem:explicit_form_of_differentiable_cocycles}
 We spell out the definition of \v{C}ech hypercohomology from Definition
 \ref{def:hypercohomology} in the case that we are interested in for
 convenience and later reference. First note that for the constant simplicial
 manifold associated to a manifold $M$ (cf.\ Example \ref{ex:nerve}) the
 \v{C}ech hypercohomology is what is also called the non-abelian cohomology
 $\check{H}^2(M , A \xrightarrow{\xmodmap} B)$ with coefficients in the crossed
 module $\xmodmap\from A\to B$ (see for instance
 \cite{NikolausWaldorf11Four-Equivalent-Versions-of-Non-Abelian-Gerbes,Wockel09Principal-2-bundles-and-their-gauge-2-groups,BaezSchreiber07Higher-gauge-theory,Bartels06Higher-gauge-theory-I:-2-Bundles,Breen94On-the-classification-of-2-gerbes-and-2-stacks,Giraud71Cohomologie-non-abelienne,Dedecker60Sur-la-cohomologie-non-abelienne.-I}).

 In this article we will mainly be interested in the simplicial manifold
 $\bB G_\bullet$ the nerve of the infinite-dimensional Lie groupoid
 $G\rrarrow *$\footnote{If $\dim(G)<\infty$, then $\pi_{2}(G)=0$ and the
 integration procedure we consider here is covered by
 \cite{Neeb02Central-extensions-of-infinite-dimensional-Lie-groups}.} with the
 2-term complex of sheaves of germs of smooth functions with value in $A$ and
 $B$, where $A\xrightarrow{\xmodmap}B$ is a morphism of abelian Lie groups.
 
 Then our triple complex is
 \begin{equation}\label{eqn:triple-complex}
  \check{C}^{p,q,-1}= \prod_{i_{0},...,i_{q}\in I^{(p)}} C^\infty (U_{i_{0}}^{(p)}\cap\dots\cap U_{i_{q}}^{(p)}, A)\quad\text{ and }\quad 
  \check{C}^{p,q,0}= \prod_{i_{0},...,i_{q}\in I^{(p)}} C^\infty (U_{i_{0}}^{(p)}\cap\dots\cap U_{i_{q}}^{(p)}, B),
 \end{equation}
 with $d=\xmodmap_*$, i.e.,
 \begin{equation*}
  (\xmodmap_{*}f)_{i_{0},...,i_{q}}(x)=
  \xmodmap (f_{i_{0},...,i_{q}}(x)).
 \end{equation*}
 
 Then a 2-cocycle of the total differential in this triple complex is given by
 maps (we also write down the slightly more intuitive names of them that we
 will use in the geometric construction later on):
 \begin{align*}
  & \phi^{1,1,0}(=\gamma=(\gamma_{i,j}))\from U^{(1)}_{i}\cap
  U^{(1)}_{j}\to B, \forall {i,j\in I^{(1)}}\\
  & \phi^{1,2,-1}(=\eta=(\eta_{i,j,k}))\from U^{(1)}_{i}\cap
  U^{(1)}_{j}\cap U^{(1)}_{k}\to A, \forall {i,j,k\in I^{(1)}}\\
  & \phi^{2,0,0}(=F=(F_{i}))\from U^{(2)}_{i}\to B, \forall {i\in I^{(2)}}\\
  & \phi^{2,1,-1}(=\Phi=(\Phi_{i,j}))\from U^{(2)}_{i}\cap
  U^{(2)}_{j}\to A, \forall {i,j\in I^{(2)}}\\
  & \phi^{3,0,-1}(=\Theta=(\Theta_{i}))\from U^{(3)}_{i}\to A,
  \forall {i\in I^{(3)}}
 \end{align*}
 such that (when $r=0$)
 \begin{align} 
  d_{0,1,0}(\phi^{1,2,-1})&(=\check{\delta}(\eta))=0 \label{eq:120} \\
  d_{1,0,0}(\phi^{1,2,-1})+d_{0,1,0}(\phi^{2,1,-1})&(=\dsimp(\eta)\highlight{+}\check{\delta}(\Phi))=0\label{eq:210-120}\\
  d_{1,0,0}(\phi^{2,1,-1})-d_{0,1,0}(\phi^{3,0,-1})&(=\dsimp(\Phi)-\check{\delta}(\Theta))=0\label{eq:300-210}\\
  d_{1,0,0}(\phi^{3,0,-1})&(=\dsimp(\Theta))=0\label{eq:300}
 \end{align}
 and (involving mixings of $r=0$ and $r=1$)
 \begin{align} %
  \highlight{-}d_{0,1,0}(\phi^{1,1,0})-d_{0,0,1}(\phi^{1,2,-1})&(=\highlight{-}\check{\delta}(\gamma)-\xmodmap_{*}(\eta))=0 \label{eq:111-120}\\
  d_{1,0,0}(\phi^{1,1,0})+d_{0,1,0}(\phi^{2,0,0})-d_{0,0,1}(\phi^{2,1,-1})&(=\dsimp(\gamma)+\check{\delta}(F)-\xmodmap_{*}(\Phi))=0 \label{eq:201-210-111}\\
  d_{1,0,0} (\phi^{2,0,0})-  d_{0,0,1}(\phi^{3,0,-1})&(=\dsimp (F)-\xmodmap_{*}(\Theta))=0 \label{eq:300-201}
 \end{align}
 
 Two differentiable 2-cocycles
 $(\phi^{1,1,0},\phi^{1,2,-1},\phi^{2,0,0},\phi^{2,1,-1},\phi^{3,0,-1})$ and
 $(\tphi^{1,1,0},$ $\tphi^{1,2,-1},$ $\tphi^{2,0,0},$ $\tphi^{2,1,-1},$
 $\tphi^{3,0,-1})$ are called \emph{equivalent} if they differ by a coboundary,
 that is, there exists
 \begin{align*}
  &\psi^{1,0,0}(=\xi=(\xi_{i}))\from (U'')_{i}^{(1)}\to B\\
  &\psi^{1,1,-1}(=\rho=(\rho_{i,j}))\from (U'')_{i}^{(1)}\cap (U'')_{j}^{(1)}\to A\\
  &\psi^{2,0,-1}(=\sigma=(\sigma_{i})\from (U'')_{i}^{(2)}\to A
 \end{align*}
 such that
 \begin{align}
  -d_{0,1,0}(\psi^{1,1,-1})&=\phi^{1,2,-1}-\tphi^{1,2,-1}(=-\check{\delta}(\rho)=\eta-\eta')\label{110}\\
  d_{1,0,0}(\psi^{1,1,-1})+d_{0,1,0}(\psi^{2,0,-1})&=\phi^{2,1,-1}-\tphi^{2,1,-1}(=\dsimp(\rho)+\check{\delta}(\sigma)=\Phi-\Phi')\label{110-200}\\
  d_{1,0,0}(\psi^{2,0,-1})&=\phi^{3,0,-1}-\tphi^{3,0,-1}(=\dsimp(\sigma)=\Theta-\Theta')\label{200}\\
  -d_{0,1,0}(\psi^{1,0,0})+d_{0,0,1}(\psi^{1,1,-1})&=\phi^{1,1,0}-\tphi^{1,1,0}(=-\check{\delta}(\xi)+\xmodmap_{*}(\rho)=\gamma-\gamma') \label{101-110}\\
  d_{1,0,0}(\psi^{1,0,0})+d_{0,0,1}(\psi^{2,0,-1})&=\phi^{2,0,0}-\tphi^{2,0,0}(=\dsimp(\xi)+\xmodmap_{*}(\sigma)=F-F') \label{101-200}
 \end{align}
 holds on a common refinement\emptycomment{\footnote{To be precise, we need to take the
 common refinement to be a hyper-covering, that is
 $(U'')_{i}^{(1)}\cap (U'')_{j}^{(1)}$ needs to be replaced by a cover of it.
 However in our application in Section \ref{sec:geo-cocycle}, we can obtain a
 coboundary in this stricter sense. Thus for the application in the case of
 geometric cocycles and to avoid clumsy notations, we avoid taking a further
 refinement here.}} $\cU''$ of the two simplicial covers
 $(U_{i}^{(n)})_{i\in I^{(n)}}$ and $({U'}_{i}^{(n)})_{i\in (I')^{(n)}}$.
 \emptycomment{\cccomment{The collection
 $(\psi^{1,0,1},\psi^{1,1,0},\psi^{2,0,0})$ is then also called a
 \emph{coboundary} between
 $(\phi^{1,1,1},\phi^{1,2,0},\phi^{2,0,1},\phi^{2,1,0},\phi^{3,0,0})$ and
 $(\tphi^{1,1,1},\tphi^{1,2,0},\tphi^{2,0,1},\tphi^{2,1,0},\tphi^{3,0,0})$.
 Only $D(\psi^{1,0,1},\psi^{1,1,0},\psi^{2,0,0})$ is the coboundary. }}

\end{remark}

\subsection{From differentiable hypercohomology to Lie 2-groups}\label{sect:from_differentiable_hypercohomology_to_group_stacks}

\begin{tabsection}
 In this section we will describe how to construct Lie 2-groups form
 differentiable cocycles, similar to \cite[Theorem
 99]{Schommer-Pries10Central-Extensions-of-Smooth-2-Groups-and-a-Finite-Dimensional-String-2-Group}.
 
 We first introduce the concept of a group object in a bicategory and
 afterwards the corresponding notions of extensions and central extensions of
 \tgps. We will be brief on this, our main reference is
 \cite{Schommer-Pries10Central-Extensions-of-Smooth-2-Groups-and-a-Finite-Dimensional-String-2-Group}.
\end{tabsection}

\begin{definition}\label{defi:2gp}
 Let $\cat{C}$ be a bicategory with finite products. A {\em group object in
 $\cat{C}$} (or {\em $\cat{C}$-group,} for brevity) consists of the following
 data:
 \begin{itemize}
  \item an object $G$ in $\cat{C}$
  \item a list of 1-morphisms
        \begin{itemize}
         \item[] \makebox [100pt] [l] {$m: G \times G \to G$} (the {\em
               multiplication})
         \item[] \makebox [100pt] [l] {$u: * \to G$} (the {\em unit})
        \end{itemize}
        such that
        \begin{equation} \label{eq:inverse}
         (pr_1,m) : G \times G \to G \times G
        \end{equation}
        is an equivalence in $\cat{C}$.
  \item a list of invertible 2-morphisms
        \begin{align*}
         &a: m \circ (m \times \id) \Rightarrow m \circ (\id \times m) & & \quad \text{(the {\em associator})}\\
         \begin{split}
         &\ell: m \circ (u \times\id) \Rightarrow \id\\
         &r: m \circ (\id \times u) \Rightarrow \id
         \end{split} \Biggr\} & & \quad \text{(the {\em left,} resp.\ {\em
         right unit constraint})}
        \end{align*}
 \end{itemize}
 subject to the requirement that certain coherence conditions hold. A
 $\cat{ C}$-group is {\emph{strict}} if all the 2-morphisms above are identity
 2-morphisms.
 
 A {\em 1-morphism} $G\to H$ of $\cat{C}$-groups consists of a morphism
 $F\from G\to H$ in $\cat{C}$ and invertible 2-morphisms
 $F_{2}\from m_{H}{\circ}(F\times F)\Rightarrow F {\circ} m_{G}$ and
 $F_{0}\from F {\circ} u_{G}\Rightarrow u_{H}$ satisfying the corresponding
 coherence conditions. Likewise, a 2-morphism between 1-morphisms of
 $\cat{C}$-groups consists of a 2-morphism between the underlying 1-morphisms
 in $\cat{C}$ satisfying a certain coherence condition.
 
 We refer to \cite[Definition 41, 42,
 43]{Schommer-Pries10Central-Extensions-of-Smooth-2-Groups-and-a-Finite-Dimensional-String-2-Group},
 \cite[\S 4.3]{Blohmann08Stacky-Lie-groups} and \cite[p.\
 37]{BaezLauda04Higher-dimensional-algebra.-V.-2-groups} (the latter in the
 case that $\cat{C}$ is actually a strict 2-category) for the various coherence
 conditions mentioned above.
\end{definition}

\begin{definition}\label{def:Lie_2-group}
 Let $\sSt$ (respectively $\eSt$) be the bicategory of (respectively \'etale)
 Lie groupoids, i.e., objects are (respectively \'etale) Lie groupoids,
 1-morphisms are generalized morphisms and 2-morphisms are morphisms between
 generalized morphisms (see Appendix \ref{sec:grothendieck-pretop} for
 details). Then a group object in $\sSt$ is also called a \emph{\tgp}. The
 corresponding bicategory is denoted $\cat{\tgps}$. A Lie 2-group is \'etale if
 it is further a group object in $\eSt$.
\end{definition}

Notice that our notion of a Lie 2-group is equivalent to the notion of \tgp
 from \cite{Getzler09Lie-theory-for-nilpotent-Linfty-algebras,Henriques08Integrating-Lsb-infty-algebras}, defined by pointed simplicial
manifolds satisfying Kan conditions $Kan(n, j)$ for all $0\le j \le n$ and $Kan!(n,
j)$ for all $0\le j \le n \ge 3$. This has been proven in
 \cite{Zhu09n-groupoids-and-stacky-groupoids}.

\begin{definition}
 An abelian $\cat{C}$-group in a bicategory $\cat{C}$ with finite products
 is a group object $(G,m,u,a,l,r)$, together with an invertible 2-morphism
 $\beta\from m\Rightarrow m {\circ} T$, where
 $T\from G\times G\to G\times G$ is the canonical flip automorphism, such that
 the corresponding coherence conditions \cite[Definition
 47]{Schommer-Pries10Central-Extensions-of-Smooth-2-Groups-and-a-Finite-Dimensional-String-2-Group}
 are fulfilled.
 
 A 1-morphism of abelian 2-groups consists of a 1-morphism of the underlying
 $\cat{C}$-groups making the diagram from \cite[Definition
 48]{Schommer-Pries10Central-Extensions-of-Smooth-2-Groups-and-a-Finite-Dimensional-String-2-Group}
 commute. A 2-homomorphism of abelian 2-groups consists of an arbitrary
 2-morphism between 1-morphisms of abelian $\cat{C}$-groups.
\end{definition}

\begin{example}\label{ex:Lie-2-groups}
 \begin{itemize}
  \item [a)] If $G$ is an arbitrary Lie group, then the Lie groupoid with
        objects and morphisms equal to $G$ and structure maps equal to
        $\id_{G}$ gets a strict group object in $\sSt$ if we take the
        multiplication to be induced my the multiplication morphism on $G$ and
        the uni to be the inclusion of the unit element (and all 2-morphisms to
        be trivial). Since the inclusion $\cat{Man}\to \sSt$ preserves products
        this is just the image in $\sSt$ of the group object $G$ in
        $\cat{Man}$. We will denote this \tgp by $\disc{G}$.
  \item [b)] If $\xmodmap\from A\to B$ is a morphism of abelian Lie groups,
        then we get a Lie groupoid $(A\times B\rrarrow B)$ with 
        $s(a,b)=b$, $t(a,b)=\xmodmap(a)b$, $i_{b}=(e_{A},b)$ and
        $(a',\xmodmap(a)b){\circ} (a,b)=(a'a,b)$. This inherits the
        structure of a strict group object in $\sSt$ from the group
        multiplication on $A\times B$ and $B$ (thus $m$ is a honest morphism
        of Lie groupoids), which is abelian (where we may choose $\beta$ to be
        the identity). We will denote this abelian \tgp by
        $[A\xrightarrow{\xmodmap}B]$.
 \end{itemize}
\end{example}

\begin{definition}
 An \emph{extension} of $\disc{G}$ by $[A\xrightarrow{\xmodmap}B]$ consists of
 \begin{itemize}
  \item a Lie 2-group $\wh{G}$
  \item 1-morphisms $[A\xrightarrow{\xmodmap}B] \xrightarrow{p} \wh{G}$ and
        $\wh{G}\xrightarrow{q} G$
 \end{itemize}
 such that their composition is equal\footnote{Note that equality is the only
 sensible thing here since $\disc{G}$ is a discrete groupoid.} to the canonical
 1-morphism $[A\xrightarrow{\xmodmap}B]\to * \to \disc{G}$ and that
 $\wh{G}\to \disc{G}$ is a principal
 $[A\xrightarrow{\xmodmap}B]$-2-bundle (cf.\
 \cite{Schommer-Pries10Central-Extensions-of-Smooth-2-Groups-and-a-Finite-Dimensional-String-2-Group,
 NikolausWaldorf11Four-Equivalent-Versions-of-Non-Abelian-Gerbes,
 Wockel09Principal-2-bundles-and-their-gauge-2-groups}\footnote{Unlike in
 \cite{Wockel09Principal-2-bundles-and-their-gauge-2-groups}, we do here allow
 arbitrary morphisms of smooth stacks as local trivializations, not only those
 which are represented by smooth functors, cf.\
 \cite{Schommer-Pries10Central-Extensions-of-Smooth-2-Groups-and-a-Finite-Dimensional-String-2-Group,
 NikolausWaldorf11Four-Equivalent-Versions-of-Non-Abelian-Gerbes}.}).
 Two extensions
 $[A\xrightarrow{\xmodmap}B]\xrightarrow{p} \wh{G}\xrightarrow{q}\disc{G}$ and
 $[A\xrightarrow{\xmodmap}B]\xrightarrow{p'} \wh{G}'\xrightarrow{q'}\disc{G}$
 are \emph{equivalent} if there exist a 1-morphism $f\from\wh{G}\to \wh{G}'$
 and a 2-morphisms $\lambda\from f {\circ} p\Rightarrow p'$
 \begin{equation*}
  \vcenter{\xymatrix@=1em{
  &\wh{G}\ar[drr]|{\blabel{q}}="t"\\
  [A\xrightarrow{\xmodmap}B]&&&\disc{G}\\
  &\wh{G}'\ar[urr]|{\blabel{q'}}
  \ar "2,1"+<.5em,.75em>;   "1,2" |(.4){\blabel{p}}
  \ar "2,1"+<.5em,-.75em>;  "3,2" |(.4){\blabel{p}'}="s"
  \ar "1,2"-<.1em,1em>; "3,2"+<-.1em,1em> |{\blabel{f}}
  \ar@{=>} "1,2"-<.5em,1em>; "s"+<.125em,.75em>|{\blabel{\lambda}}
  \ar@{=} "t"; "3,2"
  }}
 \end{equation*}
 such that $q= q' \circ f$.
 In this case we also call $(f,\lambda)$ an \emph{equivalence} of central
 extensions.
\end{definition}

\begin{example}\label{ex:2-connected-cover}
 Suppose $G$ is a 1-connected Lie group. The space
 \begin{equation*}
  P_{e}G:=\{\gamma\in C([0,1],G)\mid \gamma(0)=e\}
 \end{equation*}
 of continuous pointed paths in $G$ is again a Lie group
 \cite{GlocknerNeeb14Infinite-dimensional-Lie-groups} with respect to the
 topology of uniform convergence and point-wise multiplication. Thus the
 evaluation map $\ev\from P_{e}G\to G$, $\gamma\mapsto \gamma(1)$ is a smooth
 group homomorphism and has a smooth section $\sigma\from U\to P_{e}G$ on some
 identity neighborhood $U\se G$. We extend this to a (in general
 non-continuous) section $\sigma\from G\to P_{e}G$. It follows from the
 existence of a smooth local section that $\ev$ is a submersion \cite[Appendix
 B]{NikolausSachseWockel13A-smooth-model-for-the-string-group}. The kernel
 $\ker(\ev)$ is the pointed (continuous) loop group $\Omega G$, which has the
 universal covering
 \begin{equation*}
  \pi_{2}(G)\to \wt{\Omega G}\to \Omega G.
 \end{equation*}
 Since continuous group automorphisms of $\Omega G$ lift in a unique way to
 group automorphisms of $\wt{\Omega G}$, we have that $P_{e}G$ acts by a lift
 of the conjugation action (from the right) on $\wt{\Omega G}$, which is smooth
 since $\pi_{2}(G)$ is discrete. Thus this action, along with the canonical map
 $\tau\from\wt{\Omega G}\to P_{e} G$ is a smooth crossed module and thus
 determines a \tgp $\Pi_{2}(G)$ \cite[Example
 4.3]{NikolausSachseWockel13A-smooth-model-for-the-string-group}, \cite[Remark
 2.4]{Wockel09Principal-2-bundles-and-their-gauge-2-groups},
 \cite{Porst08Strict-2-Groups-are-Crossed-Modules,Forrester-Barker02Group-Objects-and-Internal-Categories,BrownSpencer76G-groupoids-crossed-modules-and-the-fundamental-groupoid-of-a-topological-group}.
 It is of a quite simple nature, since the multiplication and inversion
 morphisms are represented by smooth functors on the action groupoid $\Gamma$
 of the action of $\wt{\Omega G}$ on $P_{e}G$ induced by $\tau$.
 
 Now $\Pi_{2}(G)$ comes along with a homomorphism of \tgps into $G$, induced by $\ev$.
 Moreover, $[\pi_{2}(G)\to 0]$ embeds canonically into $\Pi_{2}(G)$ if we
 consider $\pi_{2}(G)$ as a subgroup of $\wt{\Omega G}$. We thus obtain a
 sequence of \tgps
 \begin{equation*}
  [\pi_{2}(G)\to 0]\to \Pi_{2}(G)\to \disc{G}.
 \end{equation*}
 
 That this is in particular an extension we have to check that
 $\Pi_{2}(G)\to \disc{G}$ is a principal 2-bundle. For this it suffices to
 observe that over $U_{g}:=g\cdot U$ we have the smooth section
 $\sigma_{g}( x)=\sigma(g)\cdot \sigma(g^{-1}\cdot x) $ of $\ev$ and that this
 induces a smooth functor
 \begin{equation*}
  [\pi_{2}(G)\rightarrow 0]\times\disc{U_{g}}\to \left.\Pi_{2}(G)\right|_{U_{g}},\quad
  (a,\id_{x})\mapsto (a,\sigma_{g}(x)).
 \end{equation*}
 The latter can easily be shown to be a weak equivalence.
\end{example}

\begin{Lemma}
 \cite[Lemma
 82]{Schommer-Pries10Central-Extensions-of-Smooth-2-Groups-and-a-Finite-Dimensional-String-2-Group}
 Let $[A\xrightarrow{\xmodmap}B]\to \wh{G}\to\disc{G}$ be an extension. Then
 there exists a 1-morphism $\disc{G}\to Aut([A\xrightarrow{\xmodmap}B])$ of
 2-groups\footnote{2-groups are understood as group objects in the 2-category
 of categories.}, unique up to unique 2-morphism.
\end{Lemma}

\begin{definition}
 An extension $[A\xrightarrow{\xmodmap}B]\to \wh{G}\to\disc{G}$ is
 \emph{central} if the 1-morphism from the preceding Lemma is 2-equivalent to
 the trivial one.
\end{definition}

\begin{remark}\label{rk:inverse}
 If $A$ is a Lie group and $G$ a Lie group, then we define a Lie monoid
 extension of $G$ by $A$ to be a principal $A$-bundle $\wh{G}\to G$ which is a
 Lie monoid such that $\wh{G}\to G$ is a homomorphism of monoids, as well as
 the inclusion $A\to \wh{G}$, $a\mapsto e_{\wh{G}}.a$.%
 
 We will now see that $\wh{G}\to G$ is already an extension of Lie
 \emph{groups}, i.e., that
 $pr_1 \times m \from\wh{G}\times\wh{G}\to\wh{G}\times \wh{G}$ is a
 diffeomorphism. In fact, consider the factorization of
 $\pr\times m_{\wh{G}}=p \circ \beta$ through the canonical maps to and from
 the pull-back:
 \begin{equation*}
  \vcenter{\xymatrix{
  A\times A 
  \ar[rr]^{\pr\times m_{A}} \ar[d] &&
  A \times A
  \ar[d]^{\alpha} \ar@{=}[rr] &&
  A\times A
  \ar[d]\\
  \wh{G}\times\wh{G}
  \ar[rr]^(.4){\beta} \ar[d] &&
  (\pr\times m_{G})^{*}(G\times G)
  \ar[d]^{q}\ar[rr]^(.6){p} &&
  \wh{G}\times\wh{G}
  \ar[d]\\
  G \times G
  \ar@{=}[rr] &&
  G\times G 
  \ar[rr]^{\pr\times m_{G}} &&
  G \times G}},
 \end{equation*}
 where $p$ and $q$ are the canonical maps and $\alpha$ and $\beta$ are induced
 maps into the pull-back. Since $\pr\times m_{G}$ is a diffeomorphism and the
 pull-back is functorial it follows that $p$ is a diffeomorphism. Since
 $\pr\times m_{A}$ is an isomorphism of Lie groups we also have that $\beta$ is
 invertible.
\end{remark}

The remainder of this section will be devoted to the proof of the following
Theorem.

\begin{Theorem}\label{thm:from_Cech_Cohomology_to_Extensions}
 Given a morphism $A \xrightarrow{\xmodmap} B$ of abelian Lie groups and an
 arbitrary Lie group $G$ we have a well-defined map
 \begin{equation*}
  \check{H}^{3}(\bB G_{\bullet},A\xrightarrow{\xmodmap}B)\to
  \op{Ext}(\ul{G},[A\xrightarrow{\xmodmap}B]),
 \end{equation*}
 where $\op{Ext}(\ul{G},[A\xrightarrow{\xmodmap}B])$ denotes the equivalence classes
 of central extensions of $\disc{G}$ by $[A\xrightarrow{\xmodmap}B]$.
\end{Theorem}

\begin{tabsection}
 The proof of this theorem will be finished in the end of Section
 \ref{sect:cohomologous_cocycles}. We first warm up with the following
 construction which also gives a geometric interpretation of the \v{C}ech
 cohomology of a constant simplicial manifold. Note that we will throughout
 assume that the occurring \v{C}ech cocycles are normalized, which is justified
 by Corollary \ref{cor:normalization}.
\end{tabsection}

\subsubsection{The principal bundle structure} \label{sec:principal-bundle}

\begin{tabsection}
 For a morphism $A \xrightarrow{\xmodmap} B$ of abelian Lie groups and a
 manifold $M$, viewed as a constant simplicial manifold, let $(\eta, \gamma)$
 be a 2-cocycle representing an element in the hypercohomology
 $\check{H}^2(M , A \xrightarrow{\xmodmap} B)$ (cf.\ Remark
 \ref{rem:explicit_form_of_differentiable_cocycles}), that is, there is an open
 cover $(U_i)_{i\in I}$ on $M$ with smooth maps
 \begin{equation*}
  \eta: U_{[2]}\to A, \quad \gamma:U_{[1]} \to B,
 \end{equation*}
 satisfying $D_3(\eta, \gamma)=0$, i.e.,
 $\xmodmap {\circ} \eta=\check{\delta} (\gamma)$ and
 $\check{\delta} (\eta) =0$. Here, to simplify the notation, we refer for
 $q\geq 0$ to $U_{[q]}$ as the disjoint union
 \begin{equation} \label{eq:lower-index}
  U_{[q]}:=\coprod_{i_{0},...,i_{q}\in I}U_{i_{0}}\cap \dots\cap U_{i_{q}}
 \end{equation}
 of $(q+1)$-fold intersections. The cocycle $(\gamma,\eta)$ defines a principal
 2-bundle $\cP_{(\gamma,\eta)}\to M$ (cf.\
 \cite{NikolausWaldorf11Four-Equivalent-Versions-of-Non-Abelian-Gerbes,
 Wockel09Principal-2-bundles-and-their-gauge-2-groups}). It is presented by a
 Lie groupoid denoted by
 $U_{[1]} \times_{(\eta, \gamma)} B \times A \rrarrow  U_{[0]} \times
 B $. \cccomment{there was a typo on lower index.}  The
 structure maps are given by
 \[ \bt(x_{ij}, b, a)=(x_i, b), \quad \bs(x_{ij}, b, a)=
  (x_j, b + \xmodmap(a) +\gamma_{ij}(x) ), \quad \forall (x_{ij}, b, a)
  \in U_{ij} \times B \times A,
 \]
 and
 \begin{equation}\label{eqn:groupoid_multiplication}
  (x_{ij}, b, a) \cdot ( x_{jk}, b', a') = (x_{ik}, b, a +
  a' \highlight{-} \eta_{ijk}(x)),
 \end{equation}
 \[ e(x_i, b) = (x_{ii}, b, 0),
 \]
 \[ (x_{ij}, b, a)^{-1} = (x_{ji}, b+\highlight{\xmodmap(a) } + \gamma_{ij}(x),
  -a).
 \]
 Then $D_3(\eta, \gamma)=0$ implies that \eqref{eqn:groupoid_multiplication} is
 compatible with the source and target maps and is associative.
 
 For the description of bundle morphisms in local coordinates it is often
 necessary to change the open covers describing a principal 2-bundle. Unlike in
 the case of principal bundles, different choice of $\eta, \gamma$ representing
 the same class in $H^2(M , A \xrightarrow{\xmodmap} B)$ do not give isomorphic
 principal 2-bundles, but only essentially equivalent groupoids with isomorphic
 principal bundle structure:
\end{tabsection}

\begin{Theorem}(\cite[Theorem
 2.22]{Wockel09Principal-2-bundles-and-their-gauge-2-groups}) Given a strict
 2-group $A\xrightarrow{\xmodmap} B$, the equivalence classes of
 principal $[A\to B]$-2-bundles $\cP \to M$ are classified by
 $\check{H}^2(M, A\xrightarrow{\xmodmap} B)$.
\end{Theorem}

\subsubsection{The \tgp structure}

\begin{tabsection}
 Now starting from a cocycle
 $\phi \in \check{Z}^3_U (\bB G_{\bullet}, A\xrightarrow{\xmodmap}B )$ of a
 given simplicial covering $(U^{(n)}_i)_{i\in I^{(n)}}$, we set off to
 construct the groupoid which provides the base space for $\wh{G}$ in the
 central extension $[A\xrightarrow{\xmodmap}B]\to \wh{G}\to \disc{G}$.
 
 First of all, we need to construct a $A \xrightarrow{\xmodmap} B$ principal
 2-bundle over $G$. There is a homomorphism
 \begin{equation}
  \check{Z}_{\cU}^3(\bB G_{\bullet}, A \xrightarrow{\xmodmap} B) \xrightarrow{\tau}  \check{Z}_{\cU}^2(G, A \xrightarrow{\xmodmap} B), \quad  \tau(\phi^{1,1,0}, \phi^{1,2,-1}, \phi^{2,1,-1}, \phi^{2,0,0}, \phi^{3,0,-1}) = ( \phi^{1,1,0}, \phi^{1,2,-1})
 \end{equation}
 inducing on the level of cohomology classes the edge homomorphism
 \begin{equation*}
  \check{H}^{3}(\bB G_{\bullet},A\xrightarrow{\xmodmap}B)\to \check{H}^{2}(G,A\xrightarrow{\xmodmap}B)
 \end{equation*}
 (note that $\check{H}^{2}_{\cU}(G,A\xrightarrow{\xmodmap}B)$ is by definition
 the cohomology of the double complex $\check{C}^{1,q,r}$ for $p=1$ constant
 from \eqref{eqn:triple-complex}).
 
 The 2-cocycle $(\phi^{1,1,0}, \phi^{1,2,-1})$ gives us the desired principal
 2-bundle via the construction in Section \ref{sec:principal-bundle}. We call
 this Lie groupoid $\Gamma[\phi]$ and it will serve as the underlying Lie
 groupoid of our \tgp. Notice that $(U^{(1)}_i)_{i\in I^{(1)}}$ is a covering
 of $G$, thus by the above discussion we have
 \begin{equation*}
 \Gamma[\phi]_0 = U^{(1)}_{[0]} \times B, \quad
  \Gamma[\phi]_1=U^{(1)}_{[1]} \times B \times A.
 \end{equation*}
 Here the lower indices denote disjoint union of intersections (see
 \eqref{eq:lower-index}). For this section we will switch back to the notation
 for cocycles that we used in Section
 \ref{sect:differentiable_hypercohomology}. Then in this notation the groupoid
 multiplication on $\Gamma[\phi]$ is given by
 \[(u_0, u_1, b, a) \cdot (u_1, u_2, b', a') = 
  (u_0, u_2, b, a+a'\highlight{-} \phi^{1,2,-1} (u_0, u_1, u_2) ),
 \]
 where we have also identified the intersection $U_{i_{0}}\cap U_{i_{1}}$ with
 the pull-back $U_{i_{0}}\times_{G} U_{i_{1}}$.
 
 We recall that $\bB G$, the nerve of $(G\rrarrow *)$, is a simplicial manifold as
 established in Example \ref{ex:nerve}. We use
 $d_2, d_0 : \bB G_2=G\times G \to \bB G_1 =G$ to pullback the 2-cocycle
 $(\phi^{(1,1,0)}, \phi^{(1,2,-1)})$ on $G$. Then we obtain two such 2-cocycles
 $d^*_2(\phi^{(1,1,0)}, \phi^{(1,2,-1)})$ and
 $d^*_0(\phi^{(1,1,0)}, \phi^{(1,2,-1)})$ on $G\times G$. The above construction gives us a $(A\xrightarrow{\xmodmap}B)^{\times 2}$ principal
 2-bundle on $G\times G $, whose underlining Lie groupoid, denoted by
 $\Gamma^2[\phi]$, is given by
 \[
  U_{[1]}^{(2)} \times B^{\times 2} \times A^{\times 2} \rrarrow
  {U}_{[0]}^{(2)}\times B^{\times 2}.
 \]
 Moreover, there is a Lie groupoid morphism
 $(d_0, d_2): \Gamma^2[\phi] \to \Gamma[\phi]\times \Gamma[\phi]$ defined on
 $ U_{[1]}^{(2)} \times B^{\times 2} \times A^{\times 2} \to U_{[1]}^{(1)} \times B \times A \times U_{[1]}^{(1)} \times B \times A $
 by
 \begin{equation}\label{eq:Gamma2}
  (v_0, v_1, b_0 , b_1, a_0, a_1) \mapsto (d_0(v_0), d_0(v_1), b_0, a_0) \times (d_2(v_0) , d_2(v_1), b_1 , a_1  ).
 \end{equation}
 We will throughout define a Lie groupoid morphism on the space of arrows since
 this determines the morphism on the objects uniquely. Notice that
 $  d_{0}^{*}\Gamma[\phi]\times d_{2}^{*}\Gamma[\phi] $ is given by
 \begin{equation*}
  {U'}_{[1]}^{(2)}\times A^{\times 2}\times B^{\times 2}\rrarrow
  {U'}_{[0]}^{(2)}\times B^{\times 2},
 \end{equation*}
 where
 $({U'}_{i_{0},i_{1}}^{(2)}:=U_{i_{0}}\times U_{i_{1}})_{(i_{0},i_{1})\in I^{(1)}\times I^{(1)}}$
 is the product open covering of $G\times G$. Note that this is exactly the
 same as the coarsest common refinement of the pull-back covers
 $d_{0}^{*}(U_{i}^{(1)})_{i\in I^{(1)}}$ and
 $d_{2}^{*}(U_{i}^{(1)})_{i\in I^{(1)}}$, thus it contains $U^{(2)}$ as a
 subcovering. By \cite[Corollary
 2.23]{Wockel09Principal-2-bundles-and-their-gauge-2-groups}, the groupoid
 morphism \eqref{eq:Gamma2} as the composition of the inclusion
 $\Gamma^2[\phi]\to d_{0}^{*}\Gamma[\phi]\times d_{2}^{*}\Gamma[\phi]$ and the
 isomorphism
 $  d_{0}^{*}\Gamma[\phi]\times d_{2}^{*}\Gamma[\phi] \cong  \Gamma[\phi]^{\times 2}$
 is a weak equivalence. Thus $\Gamma^2[\phi] $ and $ \Gamma[\phi]^{ \times 2}$
 are Morita equivalent. Similarly, we use
 $d_0d_0, d_2d_0, d_2d_2: \bB G_3 \to \bB G_1$ to pullback the 2-cocycle
 $(\phi^{(1,1,0)}, \phi^{(1,2,-1)})$. This gives us a Lie groupoid
 $\Gamma^3[\phi]$ which is a $(A\xrightarrow{\xmodmap}B)^{\times 3}$ principal
 2-bundle on $G\times G\times G $. Moreover the map $(d_2 d_2, d_2 d_0, d_0 d_0)$
 gives arise to a Morita equivalence between $\Gamma^3[\phi]$ the product of
 three copies of $\Gamma[\phi]$.

 We now define the multiplication morphism $m: \Gamma^2[\phi] \to \Gamma[\phi]$
 by
 \begin{equation}\label{eqn:multiplication-morphism}
  (v_0, v_1, b_0 , b_1, a_0, a_1) \mapsto (d_1(v_0), d_1(v_1), b_0 +b_1 +
  \phi^{2,0,0} (v_0), a_0+a_1 + \phi^{2,1,-1}(v_0, v_1) )
 \end{equation}
 Since the morphism $\Gamma^{2}[\phi]\to \Gamma[\phi]\times \Gamma[\phi] $ is a
 weak equivalence, the span
 \begin{equation*}
  \Gamma[\phi]\times \Gamma[\phi]\leftarrow \Gamma^{2}[\phi]\xrightarrow{m} \Gamma[\phi]
 \end{equation*}
 defines a generalized morphism
 $\tilde{m}\from \Gamma[\phi]\times \Gamma[\phi]\to \Gamma[\phi]$. That
 \eqref{eqn:multiplication-morphism} is indeed a Lie groupoid morphism follows
 from \eqref{eq:201-210-111} and \eqref{eq:210-120} as follows. Keep in mind
 that in $\Gamma^{2}[\phi]$ we have
 \begin{equation*}
  \begin{split}
  &\bt(v_0, v_1, b_0, b_1, a_0, a_1) = (v_0, b_0, b_1), \\
  &\bs(v_0, v_1, b_0, b_1, a_0, a_1) = (v_1, b_0 +\xmodmap(a_0)+\phi^{1,1,0}(\highlight{d_0}(v_0), \highlight{d_0}(v_1)), b_1 + \mu(a_1)+\phi^{1,1,0}(\highlight{d_2}(v_0), \highlight{d_2}(v_1))), \\
  &(v_0, v_1, b_0, b_1, a_0, a_1) \cdot (v_1, v_2, b'_0, b'_1, a'_0,
  a'_1) \\ &= \big(v_0, v_2, b_0, b_1,   a_0+a'_0\highlight{-}\phi^{1,2,-1} (d_2(v_0), d_2(v_1), d_2(v_2)), a_1+ a'_1\highlight{-}  \phi^{1,2,-1}(d_0(v_0), d_0(v_1),\highlight{d_0(v_{2})})\big).
  \end{split}
 \end{equation*}
 Then $m {\circ}\bt=\bt {\circ} m$ holds by definition and we have
 \begin{equation*}
  \begin{split}
  & m \circ \bs(v_0, v_1, b_0, b_1, a_0, a_1) \\  = & \big(d_1(v_1), b_0+\mu(a_0) + \phi^{1,1,0}(d_2(v_0), d_2(v_1))+b_1 +\mu(a_1)+\phi^{1,1,0}(d_0(v_0), d_0(v_1)) + \phi^{2,0,0}(v_1)\big)
  \end{split}
 \end{equation*}
 \[\bs\circ m(v_0, v_1, b_0, b_1, a_0, a_1) =\big(d_1(v_1), b_0+b_1+\phi^{2,0,0}(v_0)+\mu(a_0+a_1)+d_{0,0,1}(\phi^{2,1,-1}(v_0, v_1))
  + \phi^{1,1,0}(d_1(v_0), d_1(v_1))\big).
 \]
 Thus $m \circ \bs=\bs\circ m$ is equivalent to \eqref{eq:201-210-111}.
 Similarly, we have on one hand
 \begin{equation*}
  \begin{split}
  &m((v_0, v_1, b_0, b_1, a_0, a_1) \cdot (v_1, v_2, b'_0, b'_1, a'_0, a'_1))\\
  =&m(v_0, v_2, b_0, b_1, a_0+a'_0\highlight{-}\phi^{1,2,-1}(d_2(v_0), d_2(v_1), d_2(v_2)), a_1+a'_1\highlight{-}\phi^{1,2,-1}(d_0(v_0), d_0(v_1), d_0(v_2))) \\
  =&(d_1(v_0), d_1(v_2), b_0+b_1+\phi^{2,0,0}(v_0), \\
  & a_0+a'_0\highlight{-}\phi^{1,2,-1}(d_2(v_0), d_2(v_1), d_2(v_2)) + a_1 +a'_1\highlight{-}\phi^{1,2,-1}(d_0(v_0), d_0(v_1), d_0(v_2)) +\phi^{2,1,-1}(v_0, v_2) ),
  \end{split}
 \end{equation*}
 and on the other hand
 \begin{equation*}
  \begin{split}
  &m(v_0, v_1, b_0, b_1, a_0, a_1) \cdot m(v_1, v_2, b'_0, b'_1, a'_0, a'_1) \\
  =&(d_1(v_0), d_1(v_1), b_0+b_1+\phi^{2,0,0}(v_0), \\
  & a_0+a_1+\phi^{2,1,-1}(v_0, v_1))\cdot (d_1(v_1), d_1(v_2), b'_0+b'_1+\phi^{2,0,0}(v_1), a'_0+a'_1+\phi^{2,1,-1}(v_1, v_2)) \\
  =&(d_1(v_0), d_1(v_2), b_0+b_1 +\phi^{2,0,0}(v_0), \\
  & a_0+a_1+\phi^{2,1,-1}(v_0, v_1)+a'_0+a'_1+\phi^{2,1,-1}(v_1, v_2) \highlight{-} \phi^{1,2,-1}(d_1(v_0), d_1(v_1), d_1(v_2))).
  \end{split}
 \end{equation*}
 Thus $m$ and $\cdot$ commute iff \eqref{eq:210-120} holds.
 
 Like the multiplication morphism, the associator will not be defined on
 $\Gamma[\phi]^{\times 3}$, but on the Morita equivalent Lie groupoid
 $\Gamma^{3}[\phi]$. There are two Lie groupoid morphisms
 $f_1 ,f_2: \Gamma^3[\phi] \to \Gamma[\phi]$ defined on
 $U^{(3)}_{[1]}  \times B^{\times 3} \times A^{\times 3} \to (U^{(1)}_{[1]} \times B  \times A)^{\times 3}$
 by
 \begin{equation} \label{eq:f1f2}
  \begin{split}
  f_1: (w_0, w_1, b_0, b_1,b_2,a_0, a_1, a_2) \mapsto &(d_1 (\highlight{d_2}(w_0)), d_1 (\highlight{d_2}(w_1)), b_0+b_1+b_2 + {{d_2}}^* \phi^{2,0,0}(w_0) + {{d_0}}^* \phi^{2,0,0} (w_0),  \\ & a_0 +a_1 +a_2 + {d_2}^* \phi^{2,1,-1} (w_0, w_1) + {d_0}^* \phi^{2,1,-1} (w_0, w_1) ) \\
  f_2:(w_0, w_1, b_0, b_1,b_2,a_0, a_1, a_2) \mapsto &(d_1 ({d_1}(w_0)), d_1 ({d_1}(w_1)), b_0+b_1+b_2 + {{d_1}}^* \phi^{2,0,0}(w_0) + {{d_3}}^* \phi^{2,0,0} (w_0),  \\ & a_0 +a_1 +a_2 + {d_1}^* \phi^{2,1,-1} (w_0, w_1) + {d_3}^* \phi^{2,1,-1} (w_0, w_1) ).
  \end{split}
 \end{equation}
 \emptycomment{ where $d'_{i}\from G^{3}\to G^{2}$ denotes for $i=0,...,3$ the
 morphisms from Example \ref{ex:nerve}.} \cccomment{as the relation of $U'$ and
 $U$ is clarified before, we avoid using the symbol $d'$.}As before, $f_1$ and
 $f_2$ are Lie groupoid morphisms by \eqref{eq:201-210-111} and
 \eqref{eq:210-120} and since $d_{1}{\circ} d_{1}=d_{1}{\circ} d_{2}$. Note
 that $f_{1}$ is the same as $m {\circ} (\id \times m)$, restricted from
 $(\id\times m)^{*}\Gamma^{2}[\phi]$ to $\Gamma^{3}[\phi]$ and that $f_{2}$ is
 the same as $m {\circ} (m\times \id)$, restricted from
 $(m\times \id)^{*}\Gamma^{2}[\phi]$ to $\Gamma^{3}[\phi]$.
 
 There is a smooth natural transformation
 $\alpha\from f_1 \highlight{\Leftarrow} f_2$, which is a map
 $\alpha: \Gamma^3[\phi]_0 \to \Gamma[\phi]_1 = U_{[1]}^{(2)} \times B \times A$
 defined by
 \[
  (w_0, b_0, b_1, b_2) \mapsto (d_1
  d_2(w_0), d_1 d_2(w_0), b_0 +b_1+b_2+{d_2}^* \phi^{2,0,0}(w_0) + {d_0}^*
  \phi^{2,0,0} (w_0), \phi^{3,0,-1} (w_0) )
 \]
 \cccomment{here there is a sign problem. $-\phi^{3, 0, 1}\to \phi^{3,0,1}$.}
 To verify this, we only need to show that
 \begin{equation} \label{eq:alpha}
  f_1(\gamma) \cdot \alpha(\bs(\gamma))=\alpha(\bt(\gamma)) \cdot
  f_2(\gamma),
 \end{equation}
 for $\gamma\in \Gamma^3[\phi]_1$ (source-target matching is equivalent to
 \eqref{eq:201-210-111} and \eqref{eq:300-201} and $d_1d_1 = d_1 d_2$). Take
 $\gamma=(w_0, w_1, b_0, b_1, b_2, a_0, a_1, a_2)$, then
 \begin{equation*}
  \begin{split}
  \bs(\gamma)= &(w_1, 
  b_0+\mu(a_0)+\phi^{1,1,0}(d_0(d_0(w_0)),d_0(d_0(w_1))),
  b_1+\mu(a_1)+\phi^{1,1,0}(d_2(d_0(w_0)),d_2(d_0(w_1)),\\& 
  b_2+\mu(a_2)+\phi^{1,1,0}(d_2(d_2(w_0)),d_2(d_2(w_1))) 
  ),
  \end{split}
 \end{equation*}
 and $\bt(\gamma)=(w_0, b_0, b_1, b_2)$. Then
 \begin{equation*}
  \begin{split}
  f_1(\gamma) \cdot \alpha(\bs(\gamma))=&(d_1 \highlight{d_{2}}(w_0), d_1 \highlight{d_{2}}(w_1),
  b_0+b_1+b_2+{d_2}^* \phi^{2,0,0}(w_0)+{d_0}^*\phi^{2,0,0}(w_0),  \\ & a_0+
  a_1+ a_2 + {d_2}^* \phi^{2,1,-1}(w_0, w_1) + {d_0}^*\phi^{2,1,-1}(w_0, w_1
  ) + \phi^{3,0,-1}(\highlight{w_{1}})),
  \end{split}
 \end{equation*}
 and
 \begin{equation*}
  \begin{split}
  \alpha(\bt(\gamma))\cdot f_2(\gamma)= &(d_1(\highlight{d_2}(w_0)), d_1 (d_1(w_1)),
  b_0+b_1+b_2+{d_2}^*\phi^{2,0,0}(w_0)+{d_0}^*\phi^{2,0,0}(w_0), 
  \phi^{3,0,-1}(w_0) \\ + & a_0+a_1+a_2+{d_1}^*\phi^{2,1,-1}(w_0,
  w_1)+{d_3}^*\phi^{2,1,-1}(w_0, w_1)) \end{split}
 \end{equation*}
 Thus \eqref{eq:alpha} is equivalent to \eqref{eq:300-210}. In the end, the
 same argument as before shows that
 \begin{eqnarray}
  d_{1,0,0} \phi^{3,0,-1} =0
 \end{eqnarray}
 is equivalent to the coherence condition that the associator $\alpha$ has to
 satisfy (see also \cite[Theorem
 99]{Schommer-Pries10Central-Extensions-of-Smooth-2-Groups-and-a-Finite-Dimensional-String-2-Group}).
 
 For the unit we choose some $i\in I^{(1)}$ such that the identity $e$ of $G$
 is in $U_{i}^{(1)}$. Then the unit of $\wh{G}$ is given by the groupoid
 morphism, uniquely determined by
 \[ u: * \to \Gamma[\phi],\quad * \mapsto (e, 0) \in U_i\times B.
 \]
 Then the composition $m\circ (\id \times u)$ is defined as a smooth functor on
 the Lie groupoid $s_{1}^{*}\Gamma^{2}[\phi]$, where
 \begin{equation*}
  s_{1}\from G\times \nobreak*\cong G\to G\times G
 \end{equation*}
 is the embedding into the first factor. Since
 $(s_{1}^{-1}(U_{i}^{(2)}))_{i\in I^{(2)}}$ is a refinement of
 $(d_{2}(U_{i}^{(2)}))_{i\in I^{(2)}}$ and since
 $(d_{2}(U_{i}^{(2)}))_{i\in I^{(2)}}$ is a refinement of
 $(U_{i}^{(1)})_{i\in I^{(1)}}$ by assumption, the natural inclusion
 $s_{1}^{*}\Gamma^{2}[\phi]\hookrightarrow\Gamma[\phi]$ is a weak equivalence.
 It thus suffices to check that
 \begin{equation*}
  s_{1}^{*}\Gamma^{2}[\phi] \to \Gamma^{2}[\phi]\xrightarrow{m}\Gamma [\phi]
 \end{equation*}
 is also equal to the inclusion
 $s_{1}^{*}\Gamma^{2}[\phi]\hookrightarrow \Gamma[\phi]$. An arbitrary morphism
 in $s_{1}^{*}\Gamma^{2}[\phi]$ now is of the form
 $(v_{0},v_{1},b_{0},0,a_{0},0)$, where $v_{0}=s_{1}(u_{0})$ and
 $v_{1}=s_{1}(u_{1})$. Thus the simplicial identities and our normalization
 conditions imply that
 \begin{equation*}
  m(v_{0},v_{1},b_{0},0,a_{0},0)=
  (d_{1}(v_{0}),d_{1}(v_{1}),b_{0}+\phi^{2,0,1}(v_{0}),a_{0}+\phi^{2,1,0}(v_{0},v_{1}))=
  (v_{0},v_{1},b_{0},a_{0}).
 \end{equation*}
 This shows that the 2-morphism $m {\circ} (\id\times u)\Rightarrow \id$ can
 actually be taken to be represented by the identity natural transformation on
 the inclusion $s_{1}^{*}\Gamma^{2}[\phi]\hookrightarrow\Gamma[\phi]$. The same
 argument shows that $u$ is also a strict right unit.
 A categorification of the argument in Remark \ref{rk:inverse} now shows that
 \eqref{eq:inverse} holds for $\Gamma[\phi]$. Thus $\Gamma[\phi]$, together
 with $m,u$ and $\alpha$ is indeed a \tgp, denoted $\wh{G}_{\phi}$.
 
 Having fixed the choice of $i\in I^{(1)}$ with $e\in U_{i}$, there is a
 canonical morphism of Lie groupoids
 \mbox{$(A\times B\rrarrow B)\to \Gamma[\phi]$}, given by
 \begin{equation*}
  (a,b)\mapsto (e_{i},e_{i},a,b)\in U_{i}\cap U_{i}\times B\times A.
 \end{equation*}
 (where we again have identified $U_{i}\times_{G}U_{i}$ with
 $U_{i}\cap U_{i}\se G$) one easily checks that this defines a morphism of the
 associated \tgps $[A\xrightarrow{\xmodmap}B]\xrightarrow{p} \wh{G}_{\phi}$
 with trivial $F_{2}$ and $F_{0}$. Likewise, the morphism
 $\Gamma[\phi]\to \disc{G}$, given by
 \begin{equation*}
  (u_{0},u_{1},b,a)\mapsto u_{0}
 \end{equation*}
 gives a homomorphism of \tgps $\wh{G}_{\phi}\xrightarrow{q} \disc{G}$. From
 this it is obvious that the composition
 $[A\xrightarrow{\xmodmap}B]\to\wh{G}_{\phi}\to\disc{G}$ is the trivial
 morphism $[A\xrightarrow{\xmodmap}B]\to{*}\to\disc{G}$. We thus have an
 extension
 \begin{equation*}
  [A\xrightarrow{\xmodmap}B]\to \wh{G}_{\phi}\to \disc{G}.
 \end{equation*}
 That this is in fact a central extension follows from the fact that we
 considered the action of $\disc{G}$ on $[A\xrightarrow{\xmodmap}B]$ as trivial
 (cf.\ \cite[Lemma
 84]{Schommer-Pries10Central-Extensions-of-Smooth-2-Groups-and-a-Finite-Dimensional-String-2-Group}).
 We thus arrive at the following
\end{tabsection}
 
\begin{Proposition}\label{prop:central-extensions-from-differentiable-cocycles}
 For $\phi\in \check{Z}^3_U (\bB G_{\bullet}, A\xrightarrow{\xmodmap}B )$, the
 \tgp $\wh{G}_{\phi}$, together with the morphisms
 $[A\xrightarrow{\xmodmap}B]\xrightarrow{p} \wh{G}_{\phi}$, 
 $\wh{G}_{\phi}\xrightarrow{q} \disc{G}$  is a central extension of \tgps.
\end{Proposition}
 
\subsection{Cohomologous cocycles}
\label{sect:cohomologous_cocycles}

\begin{tabsection}
 Now suppose that we have two cohomologous 3-cocycles $\phi$ and $\tphi$
 satisfying Equations \eqref{101-110}, \eqref{101-200}, \eqref{110},
 \eqref{110-200} and \eqref{200} with a 2-cochain $\psi$. Since different
 covers lead to weak equivalences we may as well assume that $\phi$ and
 $\tphi$ live on the same cover $U$. In this case the 1-morphism
 $f: \wh{G}_{\phi}\to\wh{G}_{\tphi}$ is given by the smooth functor
 \[
  F\from \Gamma[\phi]\to \Gamma[\tphi],\quad  (u_0, u_1, b, a) \mapsto (u_0, u_1,
  b+\psi^{1,0,0}(u_0), a+\psi^{1,1,-1}(u_0, u_1) ),
 \]
 the smooth natural transformation
 $ F_{2}\from m_{G_{\phi}} {\circ}(F\times F)\Rightarrow F {\circ} m_{G_{\tphi}}$,
 \begin{equation*}
  (v_{0},b_{0},b_{1})\mapsto (d_{1}(v_{0}),d_{1}(v_{0}),b_{0}+b_{1}+\phi^{2,0,0}(v_{0})+\psi^{1,0,0}(d_{1}(v_{0})),\psi^{2,0,-1}(v_{0}))
 \end{equation*}
 and the smooth natural transformation $F_{0}\from F {\circ} u\Rightarrow u'$,
 $*\mapsto (e_{i},e_{\tilde{i}},0,0)$.
 
 Indeed, \eqref{101-110} implies that $F$ preserves source and target, and
 \eqref{110} implies that $F$ preserves the groupoid multiplication. Thus $F$
 is a groupoid morphism. Moreover, \eqref{101-200} and \eqref{110-200} imply
 source-target matching for $F_{2}$. That $F_{2}$ satisfies the coherence
 condition is then implied by \eqref{200}. The equivalence between the central
 extensions is then completed by the smooth natural transformation
 \begin{equation*}
  \lambda\from F {\circ} p \Rightarrow p',\quad (a,b)\mapsto (e_{i},e_{\tilde{i}},b,a).
 \end{equation*}
\end{tabsection}

\section{Geometric cocycle constructions} \label{sec:geo-cocycle}

\begin{tabsection}
 This section describes a geometric way for constructing differentiable
 cocycles on $G$ from Lie algebra cocycles on $\fg$ and is the heart of the
 paper.
\end{tabsection}

\subsection{Locally smooth cocycles} \label{sect:locally_smooth_cocycles}

\begin{tabsection}
 The continuous second Lie algebra cohomology $H_{c}^{2}(\fg,\fz)$ classifies
 topologically split central extensions
 \begin{equation*}
  \fz\to \wh{\fg}\to\fg
 \end{equation*}
 of $\fg$ by $\fz$. Given a 2-cocycle $\omega$ representing a class
 $[\omega]\in H^2(\g, \z)$, we first review the construction from
 \cite{Wockel08Categorified-central-extensions-etale-Lie-2-groups-and-Lies-Third-Theorem-for-locally-exponential-Lie-algebras}
 of a locally smooth cocycle $(\locsm{F},\locsm{\Theta})$ integrating $\omega$.	
\end{tabsection} 

\begin{definition}
 Let $G$ be an arbitrary Lie group and $A$ an abelian Lie group. Then we set
 \begin{equation*}
  C^{n}_{\op{loc}}(G,A):=\{f\from G^{n}\to A\mid f \text{ is smooth on some identity neighborhood}\}.
 \end{equation*}
 On $C^{n}_{\op{loc}}(G,A)$ we have the usual group differential
 \begin{multline*}
  \dd\from C^{n}_{\op{loc}}(G,A)\to
  C^{n+1}_{\op{loc}}(G,A),\quad \dd f(g_{0},...g_{n})=\\ g_{0}.f(g_{1},...,g_{n})
  -\sum _{i=0}^{n-1}(-1)^{i}f(g_{0},...,g_{i}g_{i+1},...,g_{n})-
  (-1)^{n}f(g_{0},...,g_{n-1}).
 \end{multline*}
 The corresponding cohomology groups will be denoted by $H^{n}_{\op{loc}}(G,A)$.
 If $A\xrightarrow{\xmodmap} B$ is a morphism of abelian Lie groups, then an
 \emph{$(A\xrightarrow{\xmodmap}B)$-valued locally smooth group cocycle} on $G$
 (shortly called locally smooth cocycle if the setting is understood) consists
 of two maps $F\in C_{\op{loc}}^{2}(G,B)$ and $\Theta \in C^{3}_{\op{loc}}(G,A)$ such that
 \begin{align*}
  \dd F =\xmodmap \circ\Theta\quad\text{ and }\quad
  \dd \Theta=0.
 \end{align*}
\end{definition}

\begin{tabsection}
 Now $\locsm{\Theta}$ is constructed as follows. For each
 $g,h\in G$, let $\aone{g}\from \Delta^{(1)}\to G$ be smooth with
 $\aone{g}(0)=e$ and $\aone{g}(1)=g$ and $\atwo{g,h}\from \Delta^{(2)}\to G$ be
 smooth with
 \begin{equation}\label{eqn:definition_of_alpha_and_beta_tilde}
  g.\aone{h}-\aone{gh}+\aone{g}=\psing\atwo{g,h},
 \end{equation}
 where $\psing$ denotes the differential of singular chains.
 These maps exist since we assume $G$ to be 1-connected. In addition, we may
 choose these maps so that
 \begin{equation*}
  G\ni g\mapsto \aone{g}\in C_{*}^{\infty}(\Delta^{(1)},G) \quad\text{ and }\quad
  G^{2}\ni (g,h)\mapsto \atwo{g,h}\in C^{\infty}_{*}(\Delta^{(2)},G)
 \end{equation*}
 \emptycomment{\cccomment{ $ C^{\infty}_{*}(\Delta^{(2)},G)$ are locally smooth
 cocycles? then maybe we can sort all notations somewhere in notation, or we
 point it out here locally.}} are smooth on an identity neighborhood $U$. In
 fact, if $\varphi\from P\to \wt{P}\se\fg$ is a chart for $G$ with
 $\varphi(e)=0$ and $\wt{P}$ convex and if we set $\wt{g}:=\varphi(g)$ and
 $\wt{g}\op{\star}\wt{h}:=\wt{gh}$, \emptycomment{\cccomment{what is $*$?
 notice that later it is used for concatenation. is here the same meaning?}}
 then we set
 \begin{align}
  \aone{g}(t)&:=\varphi^{-1}(t\cdot \wt{g})\label{eqn:definition_of_alpha}\\
  \atwo{g,h}(s,t)&:=\varphi^{-1}(t(\wt{g}\op{\star}s\wt{h})+s(\wt{g}\op{\star}(1-t)\wt{h}))\nonumber.
 \end{align}
 for $g,h\in U$ and $U\se P$ open with $U^{2}\se P$ and $\varphi(U)$ convex
 (cf.\ \cite[Lemma
 1.7]{Wockel08Categorified-central-extensions-etale-Lie-2-groups-and-Lies-Third-Theorem-for-locally-exponential-Lie-algebras}).
 Since the maps
 \begin{equation*}
  U\times [0,1]\ni(g,t) \mapsto \aone{g}(t)\in G\quad\tx{ and } \quad U\times U\times \Delta^{(2)}\ni (g,h,(s,t))\mapsto \atwo{g,h}(s,t)\in G
 \end{equation*}
 are smooth, the maps $g\mapsto \aone{g}$ and $(g,h)\mapsto \atwo{g,h}$ are
 smooth on $U$ and $U\times U$ respectively
 \cite{GlocknerNeeb14Infinite-dimensional-Lie-groups}. In addition, we fix
 some $V\se U$ open with $e\in V$, $V^{2}\se U$ and $V=V^{-1}$.
\end{tabsection}

\begin{Lemma}\label{lem:definition_of_theta_tilde}
 (\cite[Lemmas 1.5,1.6 and
 1.7]{Wockel08Categorified-central-extensions-etale-Lie-2-groups-and-Lies-Third-Theorem-for-locally-exponential-Lie-algebras})
 For $g,h,k\in G$ we have that
 \begin{equation}\label{eqn:definition_of_theta_tilde}
  \locsm{\Theta}(g,h,k):=g.\atwo{h,k}-\atwo{gh,k}+\atwo{g,hk}- \atwo{g,h}
 \end{equation}
 is a closed singular 2-chain on $G$ and thus defines an element of
 $H_{2}(G)\cong \pi_{2}(G)$. Moreover, the map
 $(g,h,k)\mapsto [\locsm{\Theta}(g,h,k)]$ is a $(\pi_{2}(G)\to 0)$-valued locally
 smooth group cocycle.
\end{Lemma}

\begin{tabsection}
 From the building blocks of $\locsm{\Theta}$, we can also construct a
 $(\pi_{2}(G)\xrightarrow{\per_{\omega}}\fz)$-valued cocycle as follows. We set
 \begin{equation}\label{eqn:definition_of_F_tilde}
  \locsm{F}(g,h):=\int_{\atwo{g,h}}\omega^{l} = \per_{\omega} (\atwo
  {g, h})  .
 \end{equation}
 Since $(g,h)\mapsto\atwo{g,h}$ is smooth on $V$ \emptycomment{\cccomment{$V$
 was not defined in the recalling paragraph. only $U$ was mentioned there. }}
 and since integration along a fixed integrand defines a smooth map
 $C^{\infty}(\Delta^{(2)},G)\to \fz$ (see \ref{sect:appendix}), it follows that
 $\locsm{F}$ is smooth on $V\times V$. From the definition, it follows directly
 that
 \begin{equation}\label{eqn:cocycle2}
  \locsm{F}(h,k)-\locsm{F}(gh,k)+\locsm{F}(g,hk)-\locsm{F}(g,h)=\int_{g.\atwo{h,k}-\atwo{gh,k}+\atwo{g,hk}- \atwo{g,h}}\omega^{l}=\per_{\omega}(\locsm{\Theta}(g,h,k))
 \end{equation}
 (where we used the fact that $\omega^{l}$ is left-invariant) and thus
 $(\locsm{F},\locsm{\Theta})$ is a
 $(\pi_{2}(G)\xrightarrow{\per_{\omega}}\fz)$-valued locally smooth cocycle on
 $G$. It has the property that it integrates $\omega$ (in an appropriate sense,
 cf.\ \cite[Theorem
 1.11]{Wockel08Categorified-central-extensions-etale-Lie-2-groups-and-Lies-Third-Theorem-for-locally-exponential-Lie-algebras})
 and it is universal with this property \cite[Corollary
 1.18]{Wockel08Categorified-central-extensions-etale-Lie-2-groups-and-Lies-Third-Theorem-for-locally-exponential-Lie-algebras}.\\

 Eventually, we now come to the point how the cocycle
 $(\locsm{F},\locsm{\Theta})$ is related to the integration procedure from
 \cite{Neeb02Central-extensions-of-infinite-dimensional-Lie-groups}.
\end{tabsection}

\begin{proposition}\label{prop:discrete-periods}
 If $\Per{\omega}\se\fz$ is discrete and $q\from \fz\to Z:= \fz/\Per{\omega}$
 is the quotient map, then $f:=q \circ \locsm{F}$ defines a locally smooth
 $(0\to Z)$-valued cocycle. It integrates $\omega$ in the sense that
 $D([f])=[\omega]$, where
 \begin{equation*}
  D\from  H^2_{\op{loc}}(G, \z/\per_{\omega}(\pi_2(G)))\to H_{c}^{2}(\fg,\fz)
 \end{equation*}
 is the differentiation homomorphism from \cite[Section
 4]{Neeb02Central-extensions-of-infinite-dimensional-Lie-groups}.
\end{proposition}

\begin{proof}
 Equation \eqref{eqn:cocycle2} shows that $(\dd f)(g,h,k)$ vanishes in $Z$ and
 since $\locsm{F}$ is smooth in a neighborhood of $(e,e)$, the same is true
 for $f=q \circ \locsm{F}$. Since $f$ coincides with the cocycle constructed in
 \cite[Section 6]{Neeb02Central-extensions-of-infinite-dimensional-Lie-groups}
 it integrates $\omega$ by \cite[Corollary
 6.3]{Neeb02Central-extensions-of-infinite-dimensional-Lie-groups}.
\end{proof}

\begin{tabsection}
 Since (equivalence classes of) $(0\to Z)$-valued locally smooth group cocycles
 are the same thing as central extensions of $G$ by $Z$
 \cite{Neeb02Central-extensions-of-infinite-dimensional-Lie-groups}, the
 previous proposition answers the integrability question for $\omega$ in the
 case of discrete $\Per{\omega}$.
\end{tabsection}

\subsection{Differentiable cocycles} \label{sec:differentiable_cocycles}

\begin{wrapfigure}{r}{0.4\textwidth}
	\centering
 \includegraphics[width=0.4\textwidth]{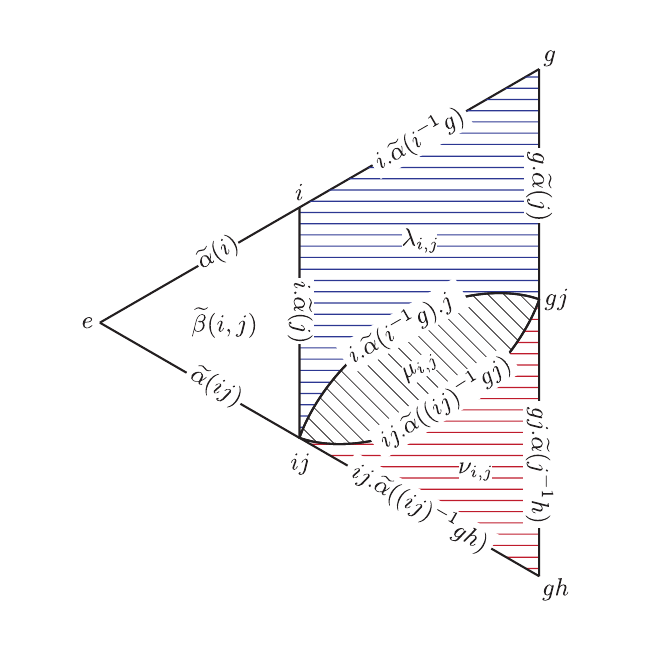}%
 \caption{Construction of $\alpha_{i,j}$ (the hatched areas depend smoothly on
 $(g,h)$ and the blank areas are constant)}%
 \label{fig:Dreieck}%
\end{wrapfigure}

 The locally smooth cocycle from the previous section lacked the global
 smoothness properties of the group structure. In the case of ordinary groups,
 the locally smooth group cocycles induced a smooth structure on the whole
 group extension, turning it into an extension of Lie group. This procedure
 made heavy used of the associativity of the group multiplication and thus does
 not seem to work for higher groups any more.
 
 In this section, we shall now enhance the construction from the previous
 section in an ad-hoc manner to a globally smooth object associated to the Lie
 algebra cocycle $\omega$, namely a differentiable cocycle with respect to an
 equivariant cover of $G$.

 We will now describe how to obtain the \v{C}ech cocycle describing the underlying 2-bundle from $(\locsm{F},\locsm{\Theta})$  (cf.\ \cite[Remark
 7.2]{Wockel08Categorified-central-extensions-etale-Lie-2-groups-and-Lies-Third-Theorem-for-locally-exponential-Lie-algebras})
 The cocycle $(\locsm{F},\locsm{\Theta})$ has the property that $\locsm{F}$ is
 smooth on $U\times U$ and $\locsm{\Theta}$ is smooth on $U\times U\times U$
 for some identity neighborhood $U\se G$. Let $V\se U$ be open such that
 $e\in V$, $V=V^{-1}$ and $V^{2}\se U$. From $V$ we obtain the open cover
 $(V_{i})_{i\in G}$ when setting
 \begin{equation}
  \label{eq:vi}
  V_{i}:=i\cdot V.
 \end{equation}
 We associate to $(\locsm{F},\locsm{\Theta})$ (see
 \eqref{eqn:definition_of_alpha_and_beta_tilde},
 \eqref{eqn:definition_of_F_tilde} and \eqref{eqn:definition_of_theta_tilde})
 the cocycle
 \begin{align}
  \gamma_{i,j}\from V_{i}\cap V_{j}\to \fz,\quad &
  g\mapsto -\locsm{F}(j,j^{-1}g)+\locsm{F}(i,i^{-1}g)-\per_{\omega}(\locsm{\Theta}(i,i^{-1}j,j^{-1}g))\label{eqn:transgression1}\\
  \eta_{i,j,l}\from V_{i}\cap V_{j}\cap V_{l}\to \pi_{2}(G),\quad &
  g\mapsto -\locsm{\Theta}(j,j^{-1}l,l^{-1}g)+\locsm{\Theta}(i,i^{-1}l,l^{-1}g)-\locsm{\Theta}(i,i^{-1}j,j^{-1}g),
  \label{eqn:transgression2}
 \end{align}
 see Section \ref{sect:interpretation_of_the_construction_i} for an
 interpretation of this assignment. 

From this definition, one immediately
 checks that $\gamma_{i,j}$ and $\eta_{i,j,l}$ satisfy \eqref{eq:120} and
 \eqref{eq:111-120} (recall that $\locsm{F}$ and $\locsm{\Theta}$ vanish
 whenever one of its arguments is $e$). That $\gamma_{i,j}$ depends smoothly on
 $g$ follows from
 \begin{equation*}
  \locsm{F}(i,i^{-1}g)-\locsm{F}(j,j^{-1}g)-\per_{\omega}(\locsm{\Theta}(i,i^{-1}j,j^{-1}g))=
  \locsm{F}(i,i^{-1}j)-i.\locsm{F}(i^{-1}j,j^{-1}g)
 \end{equation*}
 for $j^{-1}g\in V\se U$ if $g$ in $V_{j}$ and if
 $V_{i}\cap V_{j}\neq \emptyset$, then $i^{-1}j\in V^{2}\se U$. Similarly, one
 sees that $\eta_{i,j,l}$ depends smoothly on $g$ from
 \begin{equation*}
  \locsm{\Theta}(i,i^{-1}j,j^{-1}g)-
  \locsm{\Theta}(i,i^{-1}l,l^{-1}g)+
  \locsm{\Theta}(j,j^{-1}l,l^{-1}g) =
  \locsm{\Theta}(i,i^{-1}j,j^{-1}l)+
  i.\locsm{\Theta}(i^{-1}j,j^{-1}l,l^{-1}g)
 \end{equation*}

\begin{wrapfigure}{l}{0.61\textwidth}
 \centering%
 \includegraphics[width=0.59\textwidth]{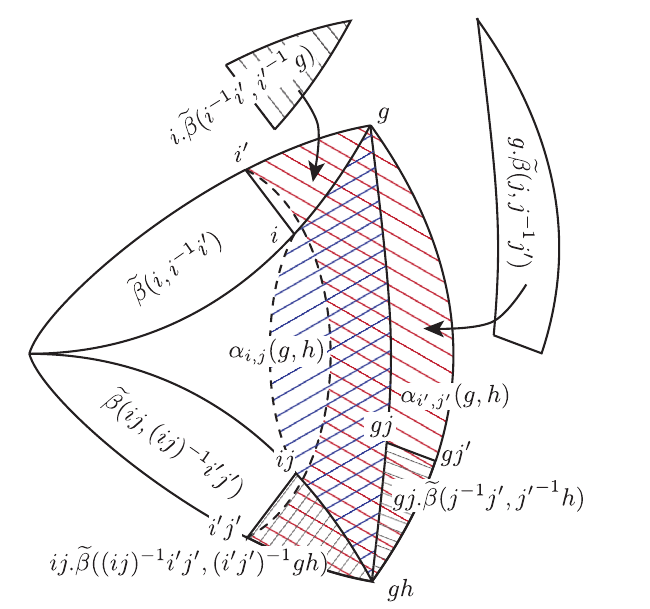}%
 \caption{Construction of $\Phi_{(i,j),(i',j')}$}%
 \label{fig:Kugel}	
\end{wrapfigure}

 The construction from Section \ref{sec:principal-bundle} now gives a Lie
groupoid $c=(V_{i},\gamma_{i,j},\eta_{i,j,l})_{i,j,k\in G}$. We call this Lie
groupoid from now on $\cG_{\omega}$. We will justify this notation later on
when showing that $\cG_{\omega}$ does up to equivalence not depend on all the
choices that we made).

 This Lie groupoid will be shown to carry the structure of a Lie-2 group that
integrates $\omega$ in the general case, regardless of the question whether
$\per_{\omega}(\pi_{2}(G))$ is discrete or not.

 In what follows, we stick to the notation introduced in Section
 \ref{sect:locally_smooth_cocycles}. What we choose in addition is for each
 $j\in G$ some open identity neighborhood $W_{j}\se V$ with the property that
 $j^{-1} W_{j} j\se V$ and that $\varphi(W_{j})\se \fg$ is convex. In the case
 that $G=C^{\infty}(M,K)$ (for $M$ a compact manifold) or in the case that
 $G=C(X,K)$ (for $X$ a compact space) and $K$ a Lie group with compact Lie
 algebra, we may without loss of generality assume that $W_{j}=V$. In fact, in these cases there
 exists convex $\op{Ad}$-invariant zero neighborhoods in $\fg$ on which the
 exponential map restricts to a diffeomorphism. This yields an equivariant chart
 $\varphi$ with respect to the conjugation action on $G$ and the adjoint action on $\fg$.
 
 Having fixed these choices we set
 \begin{equation}\label{eqn:definition_of_V_ij}
  V_{i,j}:=\{(g,h)\in G\times G: i^{-1}g\in W_{j}, j^{-1}h\in V, (ij)^{-1} gh\in V\}.
 \end{equation}
 Since $(g,h)\in V_{g,h}$, we clearly have an open cover of $G\times G$ and
 because the indexing set is $G\times G$, we have canonical maps
 $p_{1},p_{2},p_{3}\from G\times G\to G$ satisfying
 $p_{a}(V_{i,j})\se V_{p_{a}(i,j)}$ for
$a=1,2,3$.

\begin{lemma}\label{lem:defining_properties_of_alpha_ij}
 Let $C^{\infty}_{pw}(\Sigma,G)$ denote the space of piece-wise smooth maps
 from the simplicial complex $\Sigma$ in Figure \ref{fig:Triangulation} to $G$
 (cf.\ Remark \ref{rem:piece-wise-smooth-maps}). On $V_{i,j}$ there exists a
 smooth map $\alpha_{i,j}\from V_{i,j}\to C^{\infty}_{pw}(\Sigma,G)$ such that
 \begin{equation}\label{eqn:boundary_of_alpha_ij}
  \psing\alpha_{i,j}(g,h)=
  \aone{i}+i.\aone{i^{-1}g}+g.\aone{j}+gj.\aone{j^{-1}h}-
  ij.\aone{(ij)^{-1}gh}-\aone{ij}
 \end{equation}
 (see Figure \ref{fig:Dreieck}).
\end{lemma}

\begin{proof}
 We first observe that
 \begin{equation*}
  V_{i,j}\times[0,1]^{2} \ni((g,h),(s,t))\mapsto i\cdot \aone{i^{-1}g}(s)\cdot \aone{j}(1-t)\in G
 \end{equation*}
 defines a smooth map (since $i^{-1}g\in V$ if $(g,h)\in V_{i,j}$) and thus a
 smooth map $\lambda_{i,j}\from V_{i,j}\to C^{\infty}([0,1]^{2},G)$. Moreover,
 we may choose an orientation on $[0,1]^{2}$ such that
 $\left.\lambda_{i,j}(g,h)\right|_{\partial [0,1]^{2}}$ is the piece-wise
 smooth path
 \begin{equation*}
  i.\aone{i^{-1}g}*g.\aone{j}*\ol{i.\aone{i^{-1}g}.j}*\ol{i.\aone{j}}
 \end{equation*}
 (here, $*$ means concatenation of paths and $\ol{~\vphantom{\cdot}~}$ means
 orientation reversion). Choosing an appropriate triangulation of $[0,1]^{2}$
 then gives a map
 $\lambda_{i,j}\from V_{i,j}\to C^{\infty}_{pw}(\Sigma,G)$ with
 \begin{equation*}
  \psing \lambda_{i,j}(g,h)=
  i.\aone{i^{-1}g}+g.\aone{j}-{i.\aone{i^{-1}g}.j}-{i.\aone{j}}.
 \end{equation*}
 Next, we consider the map
 \begin{equation}\label{eqn1}
  V_{i,j}\times[0,1]^{2}\ni ((g,h),(s,t))\mapsto ij.\aone{(j^{-1}\cdot\takeCAREofTHIS(s)\cdot j)}(t)\in G.
 \end{equation}
 Since $\varphi(W_{j})\se\fg$ is convex and since $i^{-1}g\in W_{j}\se V$ it
 follows from the construction of $\aone{i^{-1}g}$ in
 \eqref{eqn:definition_of_alpha} that
 $j^{-1}\aone{i^{-1}g}(s)j\in j^{-1}W_{j}j\se V$ for all $s\in [0,1]$ and thus
 \eqref{eqn1} defines a smooth function. This in turn restricts to a piecewise
 smooth function $V_{i,j}\to C^{\infty}([0,1],G)$ on the boundary
 $\partial [0,1]^{2}$,
 \begin{equation*}
  i.\aone{i^{-1}g}.j*\ol{   ij.\aone{(ij)^{-1}gj}}*\ol{\vphantom{j}c_{ij}}*c_{ij}
 \end{equation*}
 (where $c_{ij}$ denotes the constant path at $ij\in G$). Choosing the same
 triangulation of $[0,1]^{2}$ as above this yields a smooth map
 $\mu_{i,j}\from V_{i,j}\to C^{\infty}_{pw}(\Sigma,G)$ with
 \begin{equation*}
  \psing \mu_{i,j}(g,h)= i.\aone{i^{-1}g}.j-ij.\aone{(ij)^{-1}gj}.
 \end{equation*}
 Since $\left.\atwo{}\right|_{V\times V}$ is smooth, it follows that
 \begin{equation*}
  \nu_{i,j}\from V_{i,j}\to C^{\infty}_{pw}(\Sigma,G), \quad (g,h)\mapsto ij.\atwo{(ij)^{-1}gj,j^{-1}h}
 \end{equation*}
 is smooth with
 \begin{equation*}
  \psing \nu_{i,j}(g,h)=ij.\aone{(ij)^{-1}gj}+gj.\aone{j^{-1}h}-ij.\aone{(ij)^{-1}gh})
 \end{equation*}
 All together, we have that
 $\aone{i,j}:=\atwo{i,j}+\lambda_{i,j}+\mu_{i,j}+\nu_{i,j}$ has the desired
 properties (where we interpret $\atwo{i,j}\in C^{\infty}_{pw}(\Sigma,G)$
 as a constant map).
\end{proof}

\begin{lemma}
 If $\alpha_{i,j}\from V_{i,j}\to C^{\infty}_{pw}(\Sigma,G)$ is smooth and
 satisfies \eqref{eqn:boundary_of_alpha_ij}, then
 \begin{align}\label{eqn:definition_of_Phi_iji'j'}
  \Phi_{(i,j),(i',j')}\from V_{i,j}\cap V_{i',j'}\to \langle C^{\infty}_{pw}(\Sigma,G)\rangle_{\Z},
  \quad(g,h)\mapsto &
  \alpha_{i,j}(g,h)-\alpha_{i',j'}(g,h)\\
  -&\atwo{i,i^{-1}i'}+i.\atwo{i^{-1}i',i'^{-1}g}\notag\\
  -&g.\atwo{j,j^{-1}j'}+gj.\atwo{j^{-1}j',j'^{-1}h}\notag\\
  +&\atwo{ij,(ij)^{-1}i'j'}-ij.\atwo{(ij)^{-1}i'j',(i'j')^{-1}gh}\notag
 \end{align}
 (see Figure \ref{fig:Kugel}) actually takes values in the singular 2-cycles
 $Z_{2}(G)$ and determines a smooth (aka locally constant) map to
 $H_{2}(G)\cong \pi_{2}(G)$. We denote this map
 $V_{i,j}\cap  V_{i',j'}\to \pi_{2}(G)$ also by $\Phi_{(i,j),(i',j')}$ for
 simplicity.
\end{lemma}

\begin{proof}
 The smoothness of $\Phi_{(i,j),(i',j')}$ follows from the smoothness of
 $\alpha_{i,j}$, $\left.\beta\right|_{U\times U}$ and from the fact that
 $i^{-1}i'\in V^{2}\se U$ if $V_{i,j}\cap V_{i',j'}\neq \emptyset$.
 Now \eqref{eqn:definition_of_alpha_and_beta_tilde} and \eqref{eqn:boundary_of_alpha_ij} imply
 \begin{align*}
  &  \psing \Phi_{(i,j),(i',j')}(g,h) \\
  =& \aone{i}+i.\aone{i^{-1}g}+g.\aone{j}+
  gj.\aone{j^{-1}h}-ij.\aone{(ij)^{-1}gh}-\aone{ij}\\
  -&\aone{i'}-i'.\aone{i'^{-1}g}-g.\aone{j'}
  -gj'.\aone{j'^{-1}h}+i'j'.\aone{(i'j')^{-1}gh}+\aone{i'j'}\\
  -&(\aone{i}+i.\aone{i^{-1}i'}-\aone{i'})
  +i.(\aone{i^{-1}i'}+i^{-1}i'.\aone{i'^{-1}g}-\aone{i^{-1}g})\\
  -&.g(\aone{j}+j.\aone{j^{-1}j'}-\aone{j'})
  +gj.(\aone{j^{-1}j'}+j^{-1}j'.\aone{j'^{-1}h}-\aone{j^{-1}h})\\
  +&(\aone{ij}+ij.\aone{(ij)^{-1}i'j'}-\aone{i'j'})\\
  -&ij.(\aone{(ij)^{-1}i'j'}+(ij)^{-1}i'j'.\aone{i'j'^{-1}gh}-
  \aone{(ij)^{-1}gh})=0,
 \end{align*}
 and the claim follows.
\end{proof}

\begin{tabsection}
 The maps $\alpha_{i,j}\from V_{i,j}\to C^{\infty}_{pw}(\Sigma,G)$
 (composed with the integration map $C^{\infty}_{pw}(\Sigma,G)\to \fz$) and
 $\Phi_{(i,j),(i',j')}\from V_{i,j}\cap V_{i',j'}\to \pi_{2}(G)$ will yield
 multiplication morphisms on the groupoid $\cG_{\omega}$. What we need in
 addition in order to turn $\cG_{\omega}$ into a \tgp is a 2-morphism yielding
 the associator. This will be furnished by the next construction. For this, we
 note that we have an open cover
 \begin{equation}\label{eqn:definition_of_V_ijl}
  V_{i,j,l}:=\{(g,h,k)\in G^{\times 3}: (g,h)\in V_{i,j},(g,hk)\in V_{i,jl}, (gh,k)\in V_{ij,l}, (h,k)\in V_{j,l}\}.
 \end{equation}
 of $G\times G\times G$ and the canonical maps
 $p_{a}\from G^{\times   3}\to G^{\times 2}$ (for $a=1,2,3,4$) satisfy
 $p_{a}(V_{i,j,l})\se V_{p(i,j,l)}$.
\end{tabsection}

\begin{lemma}
 If $\alpha_{i,j}\from V_{i,j}\to C^{\infty}_{pw}(\Sigma,G)$ is smooth and satisfies
 \eqref{eqn:boundary_of_alpha_ij} then for each $(i,j,l)\in G^{3}$, the map
 \begin{equation}\label{eqn:definition_of_Theta_ijl}
  \Theta_{i,j,l}\from V_{i,j,l}\to \langle C^{\infty}_{pw}(\Sigma,G)\rangle_{\Z},\quad (g,h,k)\mapsto
  -g.\alpha_{j,l}(h,k)+\alpha_{ij,l}(gh,k)-\alpha_{i,jl}(g,hk)+\alpha_{i,j}(g,h)
 \end{equation}
 takes values in the singular 2-chains $Z_{2}(G)$ and determines a smooth (aka
 constant) map to $H_{2}(G)\cong \pi_{2}(G)$.
\end{lemma}

\begin{proof}
 From \eqref{eqn:boundary_of_alpha_ij} we get
 \begin{align*}
  \psing \Theta_{i,j,l}(g,h,k)=
  &g.(\aone{j}+j.\aone{j^{-1}h}+h.\aone{l}+hl.\aone{l^{-1}k}-jl.\aone{(jl)^{-1}hk}-\aone{jl})\\
  -&\aone{ij}+ij.\aone{(ij)^{-1}gh}+gh.\aone{l}+ghl.\aone{l^{-1}k}-ijl.\aone{(ijl)^{-1}ghk}-\aone{ijl}\\
  &(\aone{i}+i.\aone{i^{-1}g}+g.\aone{jl}+gjl.\aone{(jl)^{-1}hk}-ijl.\aone{(ijl)^{-1}ghk}-\aone{ijl})\\
  -&\aone{i}+i.\aone{i^{-1}g}+g.\aone{j}+gj.\aone{j^{-1}h}-ij.\aone{(ij)^{-1}gh}-\aone{ij}
  =0.
 \end{align*}
 Then the claim follows from and the fact that $\alpha_{i,j}$ is smooth on
 $V_{i,j}$.
\end{proof}

We now set
\begin{equation}\label{eqn:definition_of_F_ij}
 F_{i,j}:=\per_{\omega}{\circ}\alpha_{i,j} (:= \int_{\alpha_{i,j}}
 \omega^l ),
\end{equation}
which is a smooth map from $V_{i,j}$ to $\fz$ by the assumption on
$\alpha_{i,j}$.

\begin{lemma}
 \label{rem:subsuming_the_differentiable_cocycle} There is a simplicial cover
 of $\bB G_\bullet$ with $\cU^{(1)} = \{ V_i\}$, $\cU^{(2)}= \{ V_{i,j}\} $ and
 $\cU^{(3)}=\{ V_{i, j, l} \}$ with $V_i$, $V_{i,j}$ and $V_{i,j,l}$ defined in
 \eqref{eq:vi}, \eqref{eqn:definition_of_V_ij} and
 \eqref{eqn:definition_of_V_ijl} respectively. Moreover
 \begin{align*}
  (\gamma_{i,j},\eta_{i,j,l},F_{i,j}, -\Phi_{(i,j),(i',j')},\Theta_{i,j,k})
 \end{align*}
 as defined in \eqref{eqn:transgression1}, \eqref{eqn:transgression2},
 \eqref{eqn:definition_of_F_ij}, \eqref{eqn:definition_of_Phi_iji'j'} and
 \eqref{eqn:definition_of_Theta_ijl} constitutes a differentiable
 $\pi_{2}(G)\xrightarrow{\per_{\omega}}\fz$-valued cocycle on $\bB G_{\bullet}$
 with respect to this cover.
\end{lemma}

\begin{proof}
 A simplicial cover of $\bB G_{\bullet}$ is induced as follows. As indexing
 simplicial set we take $I^{\bullet}$ with $I^{(n)}=G^{n}$ with the standard
 simplicial maps from Example \ref{ex:nerve}. Then $\cU^{(1)}$, $\cU^{(2)}$ and
 $\cU^{(3)}$ are covers of $\bB G_{1}$, $\bB G_{2}$ and $\bB G_{3}$ respectively, which
 are by construction compatible with all simplicial maps. We now define the
 cover of $\bB G_{n}$ inductively from the one of $\bB G_{n-1}$ by setting
 \begin{equation*}
  U^{(n)}_{g_{1},...,g_{n}}=\bigcap_{k=0}^{n}(d_{k}^{n})^{-1}(U^{(n-1)}_{d_{k}^{n}(g_{1},...,g_{n})}).
 \end{equation*}
 Then
 $d_{k}^{n}(U^{(n)}_{g_{1},...,g_{n}})\se U^{(n-1)}_{d_{k}^{n}(g_{1},...,g_{n})}$
 holds by definition and since $d_{k}^{n} \circ s_{k}^{n}=\id$ we have
 $s_{k}^{n}(U^{(n)}_{g_{1},...,g_{n}})\se U^{n+1}_{s_{k}^{n}(g_{1},...,g_{n})}$.

 Plugging in the definitions we obtain immediately
 \begin{equation*}
  \gamma_{j,l}(x)-\gamma_{i,l}(x)+\gamma_{i,j}(x)+
  \per_{\omega}(\eta_{i,j,l}(x))=0
  \quad\text{for}\quad  x\in V_{i}\cap V_{j}\cap V_{l}.
 \end{equation*}
 Since $\alpha_{i,j}\from V_{i,j}\to C^{\infty}_{pw}(\Sigma,G)$ satisfies
 \eqref{eqn:boundary_of_alpha_ij}, it follow from this and
 \eqref{eqn:transgression1} that
 \begin{equation*}
  \per_{\omega}(\Phi_{(i,j),(i',j')}(g,h))=
  F_{i,j}(g,h)-F_{i',j'}(g,h)
  -\gamma_{i,i'}(g)-\gamma_{j,j'}(h)+\gamma_{ij,i'j'}(gh)
 \end{equation*}
 for $ (g,h)\in V_{i,j}\cap V_{i',j'}$ (see
 \eqref{eqn:definition_of_Phi_iji'j'} and Figure \ref{fig:Kugel} for the
 definition of $\Phi_{(i,j),(i',j')}$). Finally, the maps
 $\Theta_{i,j,l}\from V_{i,j,l}\to \pi_{2}(G)$ satisfy
 \begin{equation*}
  \per_{\omega}(\Theta_{i,j,l}(g,h,k))+
  F_{j,l}(h,k)-F_{ij,l}(gh,k)+F_{i,jl}(g,hk)-F_{i,j}(g,h)=0
 \end{equation*}
 by their very definition \eqref{eqn:definition_of_Theta_ijl} and
 $F_{ij}=\per_{\omega}{\circ} \alpha_{ij}$.
 
 In order to obtain a differentiable cocycle, we have to check some further
 properties. In fact, we obtain
 \begin{equation*}
  \eta_{j,k,m}(g)-\eta_{i,l,m}(g)+\eta_{i,j,m}(g)-\eta_{i,j,l}(g)=0
 \end{equation*}
 for $g\in V_{i}\cap V_{j}\cap V_{l}\cap V_{m}$ by plugging in the definition
 of $\eta_{i,j,l}(g)$ from \eqref{eqn:transgression2}. To check the
 compatibility of $\Phi_{(i,j),(i',j')}$ with $\eta_{i,j,l}$ we observe that
 \begin{multline*}
  \Phi_{(i',j'),(i'',j'')}(g,h)
  -\Phi_{(i,j),(i'',j'')}(g,h)+\Phi_{(i,j)(i',j')}(g,h)=\\
  \locsm{\Theta}(i,i^{-1}i',i'^{-1}i'')
  +\locsm{\Theta}(i^{-1}i',i'^{-1}i'',i''^{-1}g)\\
  +\locsm{\Theta}(j,j^{-1}j',j'^{-1}j'')
  +\locsm{\Theta}(j^{-1}j',j'^{-1}j'',j''^{-1}g)\\
  -\locsm{\Theta}(ij,(ij)^{-1}i'j',(i'j')^{-1}i''j'')
  -\locsm{\Theta}((ij)^{-1}i'j',(i'j')^{-1}i''j'',(i''j'')^{-1}g)=\\
  -\eta_{i,i',i''}(g)-\eta_{j,j',j''}(h)+\eta_{ij,i'j',i''j''}(gh),
 \end{multline*}
 where the first equality follows from the definition of $\Phi_{(i,j),(i',j')}$
 in \eqref{eqn:definition_of_Phi_iji'j'} and of $\locsm{\Theta}$ in
 \eqref{eqn:definition_of_theta_tilde} (note that $G$ acts trivially on
 $\pi_{2}(G)$) and the second equality follows from \eqref{eqn:transgression2}
 and the cocycle identity for $\locsm{\Theta}$. The compatibility of
 $\Theta_{i,j,l}$ and $\Phi_{(i,j),(i',j')}$ in turn reads
 \begin{multline*}
  \Theta_{i',j',l'}(g,h,k)-\Theta_{i,j,l}(g,h,k)=\\
  g.\Phi_{(j,l),(j',l')}(h,k)- \Phi_{(ij,l),(i'j',l')}(gh,k)
  + \Phi_{(i,jl)(i',j'l')}(g,hk)- \Phi_{(i,j),(i,j')}(g,h),
 \end{multline*}
 which follows by plugging the definitions of $\Theta_{i,j,l}$ from
 \eqref{eqn:definition_of_Theta_ijl} and of $\Phi_{(i,j),(i',j')}$ from
 \eqref{eqn:definition_of_Phi_iji'j'}. Finally, we have to check that
 $\Theta_{i,j,l}$ is closed with respect to $\dsimp$
 \begin{equation*}
  g.\Theta_{j,l,m}(h,k,o)
  -\Theta_{ij,l,m}(gh,k,o)
  +\Theta_{i,jl,m}(g,hk,o)
  -\Theta_{i,j,lm}(g,h,ko)
  +\Theta_{i,j,l}(g,h,k)=0,
 \end{equation*}
 which follows from the simple fact that $\Theta_{i,j,k}$ is already the image
 of $\alpha_{i,j}$ under $\dsimp$.
\end{proof}

\subsection{Dependency on choices}
In this section we shall briefly discuss the dependence of the construction
from the previous section on the various choices that we made. If we first fix the
Lie algebra 2-cocycle $\omega$, then we are left with
\begin{enumerate}
       \renewcommand{\labelenumi}{(1.\alph{enumi})}
 \item the choice of $\aone{} \from G \to C^{\infty}(\Delta^{(1)},G)$
 \item the choice of $\atwo{}\from G\times G\to C^{\infty}( \Delta^{(2)},G)$
 \item the choice of the identity neighborhoods $U,V$
 \item the choice of the open cover $(V_{i,j})_{(i,j)\in G\times G}$ of
       $G\times G$
 \item the choice of
       $\alpha_{i,j}\from V_{i,j}\to C^{\infty}_{pw}(\Sigma,G)$.
\end{enumerate}
Those choices were made in a way such that they satisfy
\begin{enumerate}\renewcommand{\labelenumi}{(2.\alph{enumi})}
 \item $\aone{}$ is smooth on $U$ and $\atwo{}$ is smooth on $U\times U$
 \item $\psing\atwo{g,h}=\aone{g}+g.\aone{h}-\aone{gh}$
 \item $V=V^{-1}$ and $V^{2}\se U$
 \item $\pr_{2}(V_{i,j})\se j\cdot V$, $V_{i,j}\cdot V_{i,j}\se (ij)\cdot V$,
       $\pr_{1}(V_{i,j})\se i\cdot V$ and $(i,j)\in V_{i,j}$
 \item $\alpha_{i,j}\from V_{i,j}\to C^{\infty}_{pw}(\Sigma,G)$ is smooth
       and satisfies \eqref{eqn:boundary_of_alpha_ij}, i.e.
       \begin{equation*}
        \psing\alpha_{i,j}(g,h)=
        \aone{i}+i.\aone{i^{-1}g}+g.\aone{j}+gj.\aone{j^{-1}h}-
        ij.\aone{(ij)^{-1}gh}-\aone{ij}.
       \end{equation*}
\end{enumerate}
Moreover, we constructed $\aone{}$ and $\atwo{}$ on an identity neighborhood with the aid of a chart $\varphi$ and
$V_{i,j}$ with the aid of open identity neighborhoods $W_{j}\se G$ for each $j\in G$.

\begin{remark}
 All the remaining data of the differentiable cocycle
 $(\gamma_{i,j},\eta_{i,j,l},F_{i,j},-\Phi_{(i,j),(i',j')},\Theta_{i,j,l})$,
 including the cover of $\bB G_{\bullet}$ were constructed from these choices
 (cf.\ \eqref{eqn:transgression1}, \eqref{eqn:transgression2},
 \eqref{eqn:definition_of_F_ij}, \eqref{eqn:definition_of_Phi_iji'j'},
 \eqref{eqn:definition_of_Theta_ijl} and \eqref{eqn:definition_of_V_ijl}). If
 we have another collection
 \begin{equation*}
  \aone{}',\atwo{}',U',V',V'_{i,j}\text{ and }\alpha_{i,j}'
 \end{equation*}
 of the data (1.a)-(1.e) satisfying the conditions (2.a)-(2.e), then we obtain
 another differentiable cocycle
 $(\gamma_{i,j}',\eta_{i,j,l}',F_{i,j}',\Phi_{(i,j),(i',j')}',\Theta_{i,j,l}')$
 by the aforementioned construction. We will argue now that these two
 differentiable cocycles differ by a coboundary. In particular, the choices of
 the chart $\varphi\from P\to \varphi(P)\se\fg$ and of the $W_{j}$ are also
 inessential. The coboundary can be constructed from the given data as follows.
 
 Now $\psing (\aone{g}-\aoneprime{g})=0$ (interpreting $\aone{g}-\aoneprime{g}$
 as a singular 1-chain) and since $G$ is assumed to be simply connected, there
 exists for each $g\in G$ a map $\Aone{g}\in C(\Delta^{(2)},G)$ such that
 \begin{equation}\label{eqn:definition_of_A_tilde}
  \psing\Aone{g}=\aone{g}-\aoneprime{g}.
 \end{equation}
 Moreover, we may assume without loss of generality that
 $\Aone{g}\in C^{\infty}(\Delta^{(2)},G)$ and that $g\mapsto \Aone{g}$ is
 smooth on some identity neighborhood $U''\se U\cap U'$. Let $V''\se U''$ be
 open with $e\in V''={V''}^{-1}\se V\cap V'$. With this, we set
 \begin{equation*}
  \xi_{i}\from V''_{i}\to \fz,\quad
  \quad g\mapsto \int_{\Aone{i}}\omega^{l}+\int_{i.\Aone{i^{-1}g}}\omega^{l}\\
 \end{equation*}
 with $V''_{i}:=i\cdot V''$ for $i\in G$,
 \begin{align*}
  \rho_{i,j} \from V''_{i}\cap V''_{j}\to \pi_{2}(G),
  \quad g\mapsto & \atwo{i,i^{-1}g}-i.\atwo{i^{-1}j,j^{-1}g}-
  (\atwoprime{i,i^{-1}g}-i.\atwoprime{i^{-1}j,j^{-1}g})\\
  & -(\Aone{i}+i.\Aone{i^{-1}g})
  +\Aone{j}+j.\Aone{j^{-1}g}
 \end{align*}
 and
 \begin{align*}
  \sigma_{i,j}\from V_{i,j}''\to \pi_{2}(G),
  \quad (g,h)\mapsto&
  \alpha_{i,j}(g,h)-\alpha'_{i,j}(g,h)-(\Aone{i}+i.\Aone{i^{-1}})\\
  &-(\Aone{j}+j.\Aone{j^{-1}g})+\Aone{ij}+ij.\Aone{(ij)^{-1}gh}
 \end{align*}
 with
 \begin{equation*}
  V''_{i,j}:=\{(g,h)\in V_{i,j} \cap V'_{i,j}\mid g\in i\cdot V'',h\in j\cdot V'',gh\in ij\cdot V''\}
 \end{equation*}
 for $(i,j)\in G\times G$. One readily checks with
 \eqref{eqn:definition_of_alpha_and_beta_tilde},
 \eqref{eqn:definition_of_A_tilde} and \eqref{eqn:boundary_of_alpha_ij} that
 $\rho_{i,j}(g)$ and $\sigma_{i,j}(g,h)$ are in fact closed singular 2-chains
 on $G$ and thus define elements of $\pi_{2}(G)$. Moreover, it follows from the
 smoothness assumptions on $\atwo{}$, $\atwoprime{}$, $\Aone{}$, $\alpha_{i,j}$
 and $\alpha'_{i,j}$ that $\xi_{i}$, $\rho_{i,j}$ and $\sigma_{i,j}$ define
 smooth maps. Now, a lengthy but straight-forward calculation shows that
 \begin{align*} 
  \check{\delta}(\xi)+\per_{\omega}\circ\rho&=\gamma-\gamma'\notag\\
  \dsimp(\xi)+\per_{\omega}\circ\sigma&=F-F'\notag\\
  \check{\delta}(\rho)&=\eta-\eta'\notag\\
  \dsimp(\rho)+\check{\delta}(\sigma)&=-(\Phi-\Phi')\\
  \dsimp(\sigma)&=\Theta-\Theta'\notag
 \end{align*}
 holds on the refinement
 \begin{equation*}
  \cdots \rrrrarrow (V_{i,j,l}'')_{(i,j,l)\in G\times G\times G} \rrrarrow (V_{i,j}'')_{(i,j)\in G\times G}\rrarrow (V''_{i})_{i\in G},
 \end{equation*}
 where $V''_{i,j,l}$ is constructed from $V_{i,j}''$ as in
 \eqref{eqn:definition_of_V_ijl}.
\end{remark}

\begin{tabsection}
 Now let us fix all the data in (1.a)-(1.e), and take $\omega$ and $\omega'$
 representing the same class in $H^2(\g, \z)$, i.e.,
 $\omega(x,y)-\omega'(x,y)=b([x,y])$ for $b\from \fg\to \fz$ continuous and
 linear. This then results in two different differentiable cocycles
 $\phi=(\gamma_{i,j},\eta_{i,j,l},F_{i,j},\Phi_{(i,j),(i',j')},\Theta_{i,j,l})$
 and
 $\phi'=(\gamma'_{i,j},\eta'_{i,j,l},F'_{i,j},\Phi'_{(i,j),(i',j')},\Theta'_{i,j,l})$.
 To see that they are equivalent we define
 \begin{equation*}
  \xi_{i}\from i\cdot V\to \fz,\quad g\mapsto\int_{\aone{i}}b^{l}+\int_{i.\aone{i^{-1}g}}b^{l},
 \end{equation*}
 where $b^{l}$ is the left-invariant 1-from on $G$ with $b^{l}(e)=b$. Moreover,
 we set
 \begin{equation*}
  \rho_{i,j}\from i\cdot V\cap j\cdot V\to \pi_{2}(G),\quad g\mapsto -\locsm{\Theta}(i,i^{-1}j,j^{-1}g)
 \end{equation*}
 and $\sigma=0$. Then a straight-forward computation shows that $\phi$ and
 $\phi'$ differ by the coboundary $D_3(\rho, \xi, \sigma)$. In summary, we thus
 have the following
\end{tabsection}

\begin{Proposition}
 The class in $\check{H}^{3}(\bB G_{\bullet},A\xrightarrow{\xmodmap}B)$ of the
 differentiable cocycle
 \begin{equation*}
  (\gamma_{i,j},\eta_{i,j,l},F_{i,j},\Phi_{(i,j),(i',j')},\Theta_{i,j,l})
 \end{equation*}
 in $\check{H}^{3}(\bB G_{\bullet},A\xrightarrow{\xmodmap}B)$ from Lemma
 \ref{rem:subsuming_the_differentiable_cocycle} does not depend on the choices
 that we made throughout Section \ref{sec:differentiable_cocycles}.
\end{Proposition}

\subsection{A bundle-theoretic interpretation of the main construction}
\label{sect:interpretation_of_the_construction_i}

If $\locsm{f}\from G\times G\to Z$ is a locally smooth group cocycle, then we obtain a central
extension of $G$ by $Z$ as follows. We endow the set $Z\times G$ with the group
structure $\wh{\mu}_{\locsm{f}}((x,g),(y,h))=(x+y+\locsm{f}(g,h),gh)$ and
denote the resulting group by $Z\times_{\locsm{f}}G$ or shortly
$\wh{G}_{\locsm{f}}$. This turns
\begin{equation}\label{eqn:cent_ext1}
 Z\to \wh{G}_{\locsm{f}}\to G
\end{equation}
with the canonical maps into a central extension of groups. Let $\locsm{f}$ be
smooth on $U\times U$ for $U\se G$ an open identity neighborhood and $V\se U$
be open such that $e\in V$, $V=V^{-1}$ and $V\cdot V\se U$. Since $V$ is open
in $G$, $Z\times V$ generates $\wh{G}_{\locsm{f}}$ and since $Z\times U$
carries a natural manifold structure, Theorem
\ref{thm:globalisation-of-smooth-structures-on-groups} yields a Lie group
structure on $\wh{G}_{\locsm{f}}$. Clearly, \eqref{eqn:cent_ext1} is then an
exact sequence of Lie groups an since we have the smooth section
$U\ni x\mapsto (0,x)\in Z\times U$ it is a locally trivial principal bundle.

\begin{lemma}\label{lem:transgression_for_honest_cocycles}\cite[Proposition
 2.3]{Wockel09Non-integral-central-extensions-of-loop-groups} For $i\in G$ we
 set $V_{i}:=i\cdot V$. Then the assignment
 \begin{equation*}
  (\tau \locsm{f})_{i,j}\from V_{i}\cap V_{j}\to Z,\quad g\mapsto \locsm{f}(i,i^{-1}g)-\locsm{f}(j,j^{-1}g)
 \end{equation*}
 defines a smooth \v{C}ech cocycle on the open cover $(V_{i})_{i\in G}$ of $G$.
 The locally trivial principal bundle \eqref{eqn:cent_ext1}  has this cocycle
 as classifying cocycle.
\end{lemma}

It will be important for an understanding of the construction in this paper to
have a coordinate representation of the multiplication map
$\wh{\mu}_{\locsm{f}}\from \wh{G}_{\locsm{f}}\times \wh{G}_{\locsm{f}} \to \wh{G}_{\locsm{f}}$
in terms of these bundle coordinates.

\begin{remark}\label{rem:a_bundle_theoretic_interpretation}
 The multiplication map makes the diagram
 \begin{equation*}
  \xymatrix{
  Z\times Z\ar[d] \ar[r]^(.55){+}&Z\ar[d]\\
  \wh{G}_{\locsm{f}}\times\wh{G}_{\locsm{f}}\ar[d]\ar[r]^(.55){\wh{\mu}_{\locsm{f}}}&\wh{G}_{\locsm{f}}\ar[d]\\
  G\times G\ar[r]^(.55){\mu}&G}
 \end{equation*}
 commute, so we obtain a description of $\wh{\mu}_{\locsm{f}}$ on the open
 cover ${W}_{i,j}:=\{(g,h)\in V_{i}\times V_{j}:gh\in V_{ij}\}$. Since the
 local trivializations of the bundle $\wh{G}_{\wt{f}}$ are given by
 \begin{equation*}
  Z\times V_{i} \to \wh{G}_{\locsm{f}},\quad
  (x,g)\mapsto (x+\locsm{f}(i,i^{-1}g),g)
 \end{equation*}
 (cf.\ \cite[Proposition
 2.3]{Wockel09Non-integral-central-extensions-of-loop-groups}), one checks
 directly that the map
 \begin{equation}\label{eqn:coordinate_representation_of_the_multiplication}
  {f}_{i,j}\from {W}_{i,j}\to Z,\quad (g,h)\mapsto f_{i,j}(g,h):=\locsm{f}(i,i^{-1}g)+\locsm{f}(j,j^{-1}h)+\locsm{f}(g,h)-\locsm{f}(ij,(ij)^{-1}gh)
 \end{equation}
 is the coordinate representation of (the $Z$-component of) the multiplication
 map $\wh{\mu}_{\locsm{f}}$. We obtain in this way a morphism
 \begin{equation*}
  \op{Ext}(G,Z)\cong H^{2}_{\op{loc}}(G,Z)\to \check{H}^{2}(\bB G_{\bullet},0\to Z).
 \end{equation*}
 This is an isomorphism, since each class $\phi=[(\gamma_{i,j},F_{i})]$ in
 $\check{H}^{2}(\bB G_{\bullet},0\to Z)$ determines a bundle via its underlying
 \v{C}ech cocycle $\gamma_{i,j}$ and a group structure thereon by the smooth
 maps $F_{i}$. These assignments are clearly inverse to each other.
\end{remark}

\begin{tabsection}
 From this description it does not follow that $f_{i,j}$ actually is smooth, we
 only know it because we can put in a whole lot of bundle theory (yielding this
 expression of the coordinate representation) and group theory (yielding the
 smoothness of the group multiplication in Theorem
 \ref{thm:globalisation-of-smooth-structures-on-groups}).
 
 The crucial point of the construction in Section
 \ref{sec:differentiable_cocycles} was that there is an alternative expression
 for $f_{i,j}$ in the case that $f=q \circ \wt{F}$ is the locally smooth cocycle
 constructed in the case of discrete periods from a Lie algebra cocycle (cf.\
 Proposition \ref{prop:discrete-periods}). Indeed, the expression
 $\int_{\ol{\alpha}_{i,j}}\omega^{l}$ with
 \begin{equation}
  \ol{\alpha}_{i,j}(g,h)=	\wt{\beta}(i,i^{-1}g)+g.\wt{\beta}(j,j^{-1}h)+\wt{\beta}(g,h)-
  \wt{\beta}(ij,(ij)^{-1}gh)
 \end{equation}
 coincides on $V_{i,j}$ with $\int_{\alpha_{i,j}}$ since it follows immediately
 from \eqref{eqn:definition_of_alpha_and_beta_tilde} and
 \eqref{eqn:boundary_of_alpha_ij} that the difference
 $\alpha_{i,j}(g,h)-\ol{\alpha}_{i,j}(g,h)$ is closed and thus
 \begin{equation*}
  \int_{\alpha_{i,j}(g,h)-\ol{\alpha}_{i,j}(g,h)}\omega^{l}\in \Per{\omega}.
 \end{equation*}
 While the smoothness of $f_{i,j}=\int_{\ol{\alpha}_{i,j}}\omega^{l}$ is not
 immediate from its construction, the smoothness of
 $\int_{{\alpha}_{i,j}}\omega^{l}$ is so. This made the construction of the
 differentiable cocycle in Section \ref{sec:differentiable_cocycles} work.
\end{tabsection}

\section{Lie's Third Theorem} \label{sec:lie-third}

\subsection{Deriving Lie algebras from \'etale Lie 2-groups}

We now explain how to associate a Lie algebra to an \'etale \tgp. Recall from Definition \ref{def:Lie_2-group} that a Lie 2-group is a (weak) group object in the bicategory of smooth stacks (cf.\ also Appendix \ref{sec:grothendieck-pretop}).

\begin{definition}
 A Lie groupoid is \'etale if all its structure maps are local diffeomorphisms.
 We call a \tgp for which the underlying Lie groupoid is \'etale an
 \emph{\'etale \tgp.} We denote the full subcategory of \'etale \tgps in
 $\cat{\tgps}$ by $\cat{\tgps_{\acute{e}t}}$.
\end{definition}

The next short lemma is the key fact about \'etale Lie groupoids that will make the
construction in the sequel work.

\begin{lemma}\label{lem:etalness_forces_natural_transformation=id}
 Suppose that $\varphi\from \cG\to\cH$ is a smooth functor between the Lie
 groupoids $\cG$ and $\cH$ and that $\alpha\from G_{0}\to H_{1}$ is a smooth
 natural transformation with source $\varphi$. If $\cH$ is \'etale and
 $\alpha(x)=\bid(\varphi_{0}(x))$ for some $x\in G_{0}$, then
 $\alpha=\bid \circ \varphi_{0}$ on a whole neighborhood of $x$.
\end{lemma}

\begin{proof}
 Let $V\se H_{0}$ be open with $\varphi_{0}(x)\in V$ such that
 $\left.\bid\right|_{V}$ is a diffeomorphism. Since
 $\bs \circ \bid = \id_{G_{0}}$ we have that $\bs$ is the (two-sided) inverse
 of $\bid$ on $\bid(V)$, and thus it is in particular a right-sided inverse.
 For $y\in \varphi_{0}^{-1}(V)$ it thus follows from
 $\bs(\alpha(y))=\varphi_{0}(y)$ that
 $\alpha(y)= \bid(\bs(\alpha(y)))=\bid(\varphi_{0}(y))$.
\end{proof}

\begin{tabsection}
 The construction of the Lie algebra associated to the \'etale Lie 2-group
 $\cG$ is along the lines of \cite[Section
 5]{TsengZhu06Integrating-Lie-algebroids-via-stacks}. The multiplication
 morphism $m\from\cG\times\cG\to\cG$ (not to be confused with the composition
 in $\cG$) is a generalized morphism (see Appendix
 \ref{sec:grothendieck-pretop}) is represented by another Lie groupoid $\cH$,
 a weak equivalence $\cG\times \cG\xleftarrow{\Phi} \cH$ and a smooth
 functor $\cH\xrightarrow{m} \cG$. Thus there exists an open neighborhood $U$
 of $e\in G_{0}$ (where $e=u_{0}(*)$) and a smooth section
 \begin{equation*}
  \sigma\from U\times U\to {H_{0}}_{\;\Phi_{0}\!\!} \times_{s_{\cG}\times s_{\cG}}(G_{1}\times G_{1})
 \end{equation*}
 of $(t_{\cG}\times t_{\cG}) {\circ} \pr_{2}$. If $\sigma_{1}$ denotes its
 first component, then $m_{0}\from U\times U\to G_{0}$,
 $(x,y)\mapsto m_{0}(\sigma_{1}(x,y))$ is smooth and represents the restriction
 of the multiplication in $\cG$, restricted to the full sub groupoid
 $\cU:=(s_{\cG}^{-1}(U)\cap t_{\cG}^{-1}(U) \rrarrow U)$. Note that
 $m_{1}\from G_{1}\times_{U}G_{1}\to G_{1}$ is uniquely determined by its
 property to be a smooth functor and the universal property of the pull-back in
 the definition of weak equivalence.
 
 Now consider the 2-morphism
 $a\from m \circ (m \times \id)\Rightarrow m {\circ} (\id\times m)$, which is
 also represented by another Lie groupoid $\cK$, a weak equivalence
 $\cG\times \cG\times \cG\xleftarrow{\Psi} \cK$ and a smooth natural
 transformation $a\from K_{0}\to G_{1}$ between the induced smooth functors
 $m {\circ} (m \times \id)\from \cK\to\cG$ and
 $ m {\circ} (\id\times m)\from \cK\to\cG$. Inside $U$ we take an open
 neighborhood $V$ of $e$ such that $m_{0}(V\times V)\se U$ and that there
 exists a section
 \begin{equation*}
  \sigma'\from V\times V\times V\to {K_{0}}_{_{\;\Psi_{0}\!\!}}
  \times_{s_{\cG}\times s_{\cG}\times s_{\cG}}(G_{1}\times G_{1}\times G_{1})
 \end{equation*}
 of $(t_{\cG}\times t_{\cG}\times t_{\cG}) {\circ} \pr_{2}$. Then
 $\alpha\from V\times V\times V\to G_{1}$,
 $(x,y,z)\mapsto a(\sigma_{1}(x,y,z))$ defines a smooth natural transformation
 between $m {\circ} (m \times \id)$ and $ m {\circ} (\id\times m)$.
 
 Now in general $m(e,e)\neq e$, but we can re-define $m$ to achieve equality
 here. To this end consider the 2-morphism
 $\ell\from m  \circ (u\times\id)\Rightarrow \id$, which is represented by
 another Lie groupoid $\cL$, a weak equivalence $\cG\xleftarrow{\Xi} \cL$
 and a smooth natural transformation $\ell\from L_{0}\to G_{1}$ between the
 induced smooth functors ${m \circ (u\times\id)}\from \cL\to\cG$ and
 ${\id}\from \cL\to\cG$. After possibly shrinking $U$ we may assume that there
 exists a section
 $\sigma''\from U\to {L_{0}}_{_{\;\Xi_{0}\!\!}}\times_{s_{\cG}}G_{1}$ of
 $t_{\cG}\circ\pr_{2}$. Then $\lambda\from U\to G_{1}$,
 $x\mapsto \ell(\sigma''(x))$ defines a smooth natural transformation
 between $m \circ(u \times\id)$ and $\id$. With the same procedure we derive
 $\rho$ from $r\from m  \circ (\id\times u)\Rightarrow \id$. Now there exists a
 neighborhood $U'$ of $m(e,e)$ and a section $\ol{\sigma}\from U'\to G_{1}$ of
 $s_{\cG}$ with $\ol{\sigma}(m(e,e))=\lambda(e)$. We may assume that
 $m(U\times U)\se U'$ and thus define a smooth natural transformation
 $U\times U\to G_{1}$ with source $m$ by $(x,y)\mapsto \ol{\sigma}(m(x,y))$. We
 now re-define $m$ as the target of this natural transformation. Since
 $t_{\cG}(\ol{\sigma}(m(e,e)))=e$ we thus have that $m(e,e)=e$ holds for the
 re-defined $m$. If we also use this natural transformation to re-define the
 other structure morphisms of $\cG$, then this endows $\cU$ with the structure
 of a ``local \tgp'', where ``local'' means that all morphisms and 2-morphisms
 defining the group structure are only defined on the full subgroupoid of some
 neighborhood of $e$.
 
 Since $m(e,e)=e$ holds we have that $\lambda(e)=\rho(e)$ is the identity in
 $e$ and then also $\alpha(e,e,e)$ due to the coherence of $\ell,r$ and
 $\alpha$. Now Lemma \ref{lem:etalness_forces_natural_transformation=id}
 implies that $\alpha(x,y,z)$ is the identity of $m(x,m(y,z))=m(m(x,y),z)$ on
 some neighborhood of $(e,e,e)$, which we may still assume to be $V$. Thus
 $(U,V,m,e)$ is a local Lie group in the sense of \cite[Definition
 II.10]{Neeb06Towards-a-Lie-theory-of-locally-convex-groups} (the requirement
 on the existence of inverses follows from requiring $\pr_{1}\times m$ to be
 invertible with a similar argument as above).
 
 We now have to take care about the choices that we made above. Different
 choices will lead to a priori different local Lie groups $(U,V,m,e)$ and
 $(U',V',m',e')$, and we now argue that they actually agree. We first observe
 that we can achieve $e=e'$ with the same method as above when ensuring
 $m(e,e)=e$. If we construct $m'$ with the aid of a different weak
 equivalence $\cG\times\cG\leftarrow \cH'$, then the functors $m$ and $m'$ are
 smoothly equivalent when restricted to the full subgroupoid
 \begin{equation*}
  \cU \cap \cU':=\big(t^{-1}_{\cG}(U\cap U')\cap s^{-1}_{\cG}(U\cap U')
  \rrarrow U\cap U'\big).
 \end{equation*}
 Since $m(e,e)=m'(e,e)=e$ it follows from the \'etaleness of $\cG$ that this
 smooth natural transformation is actually the identity on some neighborhood of
 $e$. We thus see that $m'=m$ on some neighborhood of $(e,e)$. Thus the germ of
 the local group is uniquely determined by $\cG$, which in turn determines
 uniquely a Lie algebra $L(\cG)$.
 
 Now the same argumentation may also be applied to morphisms to show that the
 assignment $\cG\mapsto L(\cG)$ actually defines a functor
 \begin{equation}\label{eqn:Lie_functor}
  L\from \cat{\tgps_{\acute{e}t}}\to\cat{Lie Algebras},
 \end{equation}
 called the \emph{Lie functor}. It obviously has the property that if we
 pre-compose it with the fully faithful embedding
 $\cat{Lie Groups}\to \cat{\tgps_{\acute{e}t}}$, given by $G\mapsto \disc{G}$
 (see Example \ref{ex:Lie-2-groups}), then it coincides with the ordinary Lie
 functor $L\from \cat{Lie Groups}\to\cat{Lie Algebras}$.
 
 We now observe that the Lie functor is compatible with extensions.
\end{tabsection}

\begin{proposition}
 If $A\xrightarrow{\xmodmap}B$ has discrete $A$, then any central extension
 $[A\xrightarrow{\xmodmap}B]\to \wh{G}\to \disc{G}$ is equivalent to one with
 \'etale $\wh{G}$.
\end{proposition}
 
\begin{proof}
 The Lie groupoid underlying the Lie 2-groups
 $A\times B\rrarrow B$ is \'etale, and so is the total space of the principal 2-bundle
 $\cP_{(\gamma,\eta)}\to \disc{G}$, constructed as in Section
 \ref{sec:principal-bundle} from a cocycle
 $(\gamma,\eta)\in \check{Z}^{3}_{\cU}(G,A\xrightarrow{\xmodmap}B)$.
 The morphism $\cP_{(\gamma,\eta)}\to \wh{G}$ is a weak equivalence
 \cite[Proposition 2.19]{Wockel09Principal-2-bundles-and-their-gauge-2-groups}
 and thus induces on $\cP_{(\gamma,\eta)}$ the structure of a \tgp such that
 \begin{equation*}
  \vcenter{\xymatrix@=1em{
  &\wh{G}\ar[drr]_{~}="t"\\
  [A\xrightarrow{\xmodmap}B]&&&\disc{G}\\
  &\cP_{(\gamma,\eta)}\ar[urr]
  \ar "2,1"+<.5em,.75em>;   "1,2" 
  \ar "2,1"+<.5em,-.75em>;  "3,2" ^{~}="s"
  \ar "1,2"-<.1em,1em>; "3,2"+<-.1em,1em> 
  \ar@{=} "1,2"-<.75em,1.5em>; "s"+<.125em,.75em>
  \ar@{=} "t"; "3,2"+<.75em,1.5em>
  }}.
 \end{equation*}
 is an equivalence of central extensions.
\end{proof}
 
 \begin{corollary}\label{cor:differentiation}
  If $A\xrightarrow{\xmodmap}B$ has discrete $A$, then the functor
  \eqref{eqn:Lie_functor} induces a morphism
  \begin{equation*}
   D\from \op{Ext}(\ul{G},[A\to B])\to\op{Ext}(\fg,L(B)),
  \end{equation*}
  given by passing from a central extension $[A\to B]\to \wh{G}\to \ul{G}$ to an equivalent one with \'etale $\wh{G}$ and then applying $L$.
 \end{corollary}
 
\begin{proof}
 We first show that if $\wh{G}$ is \'etale, then the sequence
 $L(B)\to L(\wh{G})\to \fg$ is in fact a central extension of topological Lie
 algebras. Since the action of $\disc{G}$ on $[A\xrightarrow{\xmodmap}B]$ is
 trivial for any central extension
 $[A\xrightarrow{\xmodmap}B]\to \wh{G}\to \disc{G}$ (cf.\ \cite[Lemma
 84]{Schommer-Pries10Central-Extensions-of-Smooth-2-Groups-and-a-Finite-Dimensional-String-2-Group}),
 we only have to verify that the morphism $L(\wh{G})\to L(G)$ of topological
 Lie algebras has a continuous linear section. This in turn follows from the
 local triviality of the principal 2-bundle $\wh{G}\to \disc{G}$, since it
 implies the existence of a smooth local section of the smooth functor
 $(P_{1}\rrarrow P_{0})\to \disc{G}$ for any Lie groupoid $P_{1}\rrarrow P_{0}$
 representing $\wh{G}$.
 
 The claim now follows if we can show that for each equivalence of central
 extensions
 \begin{equation*}
  \vcenter{\xymatrix@=1em{
  &\wh{G}\ar[drr]|{\blabel{q}}="t"\\
  [A\xrightarrow{\xmodmap}B]&&&\disc{G}\\
  &\wh{G}'\ar[urr]|{\blabel{q'}}
  \ar "2,1"+<.5em,.75em>;   "1,2" |(.4){\blabel{p}}
  \ar "2,1"+<.5em,-.75em>;  "3,2" |(.4){\blabel{p}'}="s"
  \ar "1,2"-<.1em,1em>; "3,2"+<-.1em,1em> |{\blabel{f}}
  \ar@{=>} "1,2"-<.5em,1em>; "s"+<.125em,.75em>|{\blabel{\lambda}}
  \ar@{=} "t"; "3,2"
  }}.
 \end{equation*}
 with \'etale $\wh{G}$ and $\wh{G}'$ the resulting central extensions of Lie
 algebras $L(B)\to L(\wh{G})\to \fg$ and $L(B)\to L(\wh{G}')\to \fg$ are
 equivalent. But this readily follows from the fact that $f$ is in particular a
 morphism of \'etale \tgps and $L(f)$ makes the diagram
 \begin{equation*}
  \vcenter{\xymatrix@=1em{
  && L(\wh{G})\ar[dd]^{L(f)}\ar[drr]^{L(q)}\\
  L(B)\ar[urr]^{L(p)}\ar[drr]_{L(p')}&&&& \fg\\
  &&L(\wh{G}')\ar[urr]_{L(q')}
  }}
 \end{equation*}
 commute. 
\end{proof}

\begin{remark}\label{rem:differentiation_for_Cech_cocycles}
 Note that there is also
 a morphism
 \begin{equation*}
  D\from \check{H}^{2}(\bB G_{\bullet},\pi_{2}(G)\xrightarrow{\per_{\omega}}\fz)\to  H_{c}^{2}(\fg,\fz)
 \end{equation*}
 given by composing the morphism
 \begin{equation*}
  \check{H}^{2}(\bB G_{\bullet},\pi_{2}(G)\xrightarrow{\per_{\omega}}\fz)\to \op{Ext}(\ul{G},[\pi_{2}(G)\xrightarrow{\per_{\omega}}\fz])
 \end{equation*}
 from Theorem \ref{thm:from_Cech_Cohomology_to_Extensions} with
 \begin{equation*}
  D\from  \op{Ext}(\ul{G},[\pi_{2}(G)\xrightarrow{\per_{\omega}}\fz])\to \op{Ext}(\fg,\fz)\cong H_{c}^{2}(\fg,\fz)
 \end{equation*}
 from Corollary \ref{cor:differentiation}. This is clearly given on the cocycle
 level by
 \begin{equation*}
  DF_{i}(x,y):=d^{2}F_{i}(1,1)(x,y)-d^{2}F_{i}(1,1)(y,x)
 \end{equation*}
 for some $F_{i}\from U_{i}^{(2)}\to \fz$ with
 $(1,1)\in U_{i}^{(2)}\se G\times G$ (cf.\ \cite[Lemma
 4.6]{Neeb02Central-extensions-of-infinite-dimensional-Lie-groups}).
\end{remark}

 \subsection{Lie's Third Theorem for locally exponential Lie algebras}
 
\begin{tabsection}
 Central extensions of the Lie algebra $\g$ by the abelian Lie algebra $\z$ are
 classified by $H^2_{c}(\g, \z)$. We now use this fact and the established
 integration procedure to give a criterion for a Lie algebra to come from an
 \'etale \tgp. For this we first show how $[\omega]\in H_{c}^{2}(\fg,\fz)$ and
 the differentiable cocycle
 \begin{equation*}
  \phi:=  (\gamma_{i,j},\eta_{i,j,l},F_{i,j},\Phi_{(i,j),(i',j')},\Theta_{i,j,l})
 \end{equation*}
 in $\check{H}^{3}(\bB G_{\bullet},\pi_{2}(G)\xrightarrow{\per_{\omega}}\fz)$ from
 Lemma \ref{rem:subsuming_the_differentiable_cocycle} are related.
\end{tabsection}
 
 \begin{lemma}
  If we apply the Lie functor to the central extension
  $[\pi_{2}(G)\xrightarrow{\per_{\omega}}\fz]\to G_{\phi}\to\disc{G}$, then
  this results in a central extension isomorphic to $\fz \oplus_{\omega}\fg$.
 \end{lemma}
 
\begin{proof}
 By construction we have to check that the local Lie group associated to
 $G_{\phi}$ has the Lie algebra $\fz\oplus_{\omega}\fg$. But by Remark
 \ref{rem:differentiation_for_Cech_cocycles} this is exactly the same argument
 as \cite[Lemma
 1.9]{Wockel08Categorified-central-extensions-etale-Lie-2-groups-and-Lies-Third-Theorem-for-locally-exponential-Lie-algebras}
 (cf.\ also \cite[Lemma
 4.6]{Neeb02Central-extensions-of-infinite-dimensional-Lie-groups}).
\end{proof}

 This immediately implies the following

\begin{theorem}
 If $G$ is a simply connected Lie group with Lie algebra $\fg$, $\fz$ is a
 Mackey-complete locally convex space and $\omega\from \fg\times\fg\to\fz$ is
 a continuous Lie algebra cocycle. Then the differentiation homomorphism
 \begin{equation*}
  D\from  \op{Ext}(\ul{G},[\pi_{2}(G)\xrightarrow{\per_{\omega}}\fz])\to \op{Ext}(\fg,\fz)\cong H_{c}^{2}(\fg,\fz)
 \end{equation*}
 has $[\omega]$ in its image. Here $\ul{G}$ is the Lie group $G$ viewed as a
 Lie 2-group (see Example \ref{ex:Lie-2-groups}).
\end{theorem}

\begin{remark}\label{rem:relation_to_Neeb02}
 We now discuss the relation of the previous theorem to the results from
 \cite{Neeb02Central-extensions-of-infinite-dimensional-Lie-groups}. To this
 end, consider a morphism $\xmod$ of abelian Lie groups such that
 $\ker(\xmodmap)\leq \xmoda$ is a closed Lie subgroup and
 $\im{\xmodmap}\leq \xmodb$ is \emph{discrete}. Then $\pi:=\ker(\xmodmap)$ and
 $Z:=\xmodb/\im({\xmodmap})$ carry natural Lie group structures and
 $[\pi\to 0]\times[0\to Z]$ is equivalent to $[A\to B]$. Moreover, we have an
 induced sequence
 \begin{equation}\label{eqn:relation_to_Neeb02_2}
  \check{H}^{2}(\bB G_{\bullet},\pi\to 0)\xrightarrow{\chi_{*}} 
  \check{H}^{2}(\bB G_{\bullet},\xmod)\xrightarrow{\zeta_{*}} 
  \check{H}^{2}(\bB G_{\bullet},0\to Z)\to 0
 \end{equation}
 of abelian groups, given by composing cocycles with the canonical morphisms
 $\chi\from (\pi\to 0)\to (\xmod)$ and $\zeta\from(\xmod)\to (0\to Z)$ of chain
 complexes.
 
 We claim that \eqref{eqn:relation_to_Neeb02_2} is exact. It is clear that the
 sequence is of order two. If
 $\phi=(\gamma_{i,j},\eta_{i,j,k},F_{i},\Phi_{i,j},\Theta_{i})$ represents a
 class in $\ker(\zeta_{*})$, then there exist $ \xi _{i}\from U_{i}^{(1)}\to B$
 such that $\gamma_{i,j}+\check{\delta}(\xi)_{i,j}$ and
 $F_{i}-\dsimp (\xi)_{i}$ take values in $\im(\xmodmap)$\footnote{Strictly
 speaking, we would have to choose a refinement of the cover
 $(U_{i}^{(1)})_{i\in I_{1}}$ of $G$ and $\xi_{i}$ might only exist on this
 refinement. The same applies to each coboundary and each lift throughout the
 entire construction. Since the construction terminates after finitely many
 steps we may in the end choose a simplicial common refinement of all covers
 (cf.\ Lemma \ref{lem:refinements_to_simplicial_covers}) that are involved and
 thus may work throughout with one fixed cover.}. We thus may lift (after
 possibly refining the cover) $\gamma_{i,j}+\check{\delta}(\xi)_{i,j}$ to an
 $A$-valued cochain $ \rho _{i,j}\from U_{i}^{(1)}\cap U_{j}^{(1)}\to A$ and
 $F_{i}-\dsimp (\xi)_{i}$ to an $A$-valued cochain
 $\sigma_{i}\from U_{i}^{(2)}\to A$. Then
 ${\eta}':= \eta+\check{\delta}(\rho)$,
 $ \Phi':=\Phi-\dsimp (\rho)-\check{\delta}(\sigma)$ and
 $\Theta':=\Theta-\dsimp(\sigma)$ takes values in $\pi$:
 \begin{align*}
  \xmodmap \circ {\eta}'&=(\xmodmap \circ \eta)+(\xmodmap \circ \check{\delta}(\rho))=(\xmodmap \circ \eta)+
  \check{\delta}(\gamma)=0\\
  \xmodmap \circ \Phi'&=(\xmodmap \circ \Phi)-\dsimp(\gamma+\check{\delta}(\xi))-\check{\delta}(F-\dsimp(\xi))=(\xmodmap \circ \Phi)-\dsimp(\gamma)-\check{\delta}(F)=0\\
  \xmodmap \circ \Theta'&=\xmodmap \circ \Theta - \dsimp(F-\dsimp(\xi))=\xmodmap \circ \Theta - \dsimp(F)=0.
 \end{align*}
 Moreover,
 \begin{equation*}
  (0,\eta',0,\Phi',\Theta')=\phi+ D_{3}(\xi,\rho,\sigma)
 \end{equation*}
 is equivalent to $\phi$ and clearly contained in the image of $\chi_{*}$. This
 implies that \eqref{eqn:relation_to_Neeb02_2} is exact in
 $\check{H}^{2}(\bB G_{\bullet},\xmod)$. Finally, the
 surjectivity of $\zeta_{*}$ follows from a lifting argument similar to the
 previous one (i.e., lift $Z$-valued cocycles to $B$-valued cochains and
 compensate the failure of these lifts to being cocycles by $A$-valued
 cocycles).
 
 We now apply \eqref{eqn:relation_to_Neeb02_2} to
 $\pi_{2}(G)\xrightarrow{\per_{\omega}}\fz$ in the case that
 $\per_{\omega}(\pi_{2}(G))\leq \fz$ is discrete. The class of the integrating
 cocycle
 \begin{equation*}
  \phi:=  (\gamma_{i,j},\eta_{i,j,l},F_{i,j},\Phi_{(i,j),(i',j')},\Theta_{i,j,l})
 \end{equation*}
 in $\check{H}^{3}(\bB G_{\bullet},\pi_{2}(G)\xrightarrow{\per_{\omega}}\fz)$
 from Lemma \ref{rem:subsuming_the_differentiable_cocycle} then gets mapped
 under $\zeta_{*}$ to a class in
 \begin{equation*}
  \check{H}^{2}(\bB G_{\bullet},0\to Z)\cong \op{Ext}(G,Z)
 \end{equation*}
 (see Remark \ref{rem:a_bundle_theoretic_interpretation}). This is precisely
 the class of the integrating cocycle from
 \cite{Neeb02Central-extensions-of-infinite-dimensional-Lie-groups}.
\end{remark}
 
\begin{remark}\label{rem:injectivity_of_D}
 Since we assumed that $G$ is simply connected we have that the map
 \begin{equation*}
  D\from \op{Ext}(\ul{G},\ul{Z})\cong \check{H}^{2}(\bB G_{\bullet},0\to Z)\to H_{c}^{2}(\fg,\fz)\cong  \op{Ext}(\fg,\fz)
 \end{equation*}
 is injective for each $Z=\fz/\Gamma$ with $\Gamma\leq \fz$ discrete (see
 \cite[Theorem
 7.12]{Neeb02Central-extensions-of-infinite-dimensional-Lie-groups}). More
 precisely, the kernel of $D$ coincides with the image of
 \begin{equation*}
  \Hom(\pi_{1}(G),Z)\to \op{Ext}(\ul{G},\ul{Z}), \quad \varphi\mapsto (Z\times\wt{G})/\pi_{1}(G)
 \end{equation*}
 (where $\wt{G}\to G$ is the simply connected cover of $G$ and $\pi_{1}(G)$
 acts on $Z$ via $\varphi$). However, the map
 \begin{equation*}
  D\from \check{H}^{3}(\bB G_{\bullet},\pi_{2}(G)\xrightarrow{\per_{\omega}}\fz)\to  H_{c}^{2}(\fg,\fz)
 \end{equation*}
 is in general not injective. For instance, if $\fz=0$, then the differentiable
 cocycle
 \begin{equation*}
  \phi:=  (0,\eta_{i,j,l},0,\Phi_{(i,j),(i',j')},\Theta_{i,j,l})
 \end{equation*}
 in $\check{H}^{2}(\bB G_{\bullet},\pi_{2}(G)\xrightarrow{0}0)$ from Lemma
 \ref{rem:subsuming_the_differentiable_cocycle} is in general non-trivial,
 since it describes the 2-connected cover of $G$ (see Section
 \ref{sec:an_interpretation_in_terms_of_2_connected_covers}).
 
 In general, put $\pi:=\ker(\per_{\omega})$. Then the canonical morphism
 $\chi \from(\pi\to 0)\to(\pi_{2}(G)\xrightarrow{\per_{\omega}}\fz)$ of chain
 complexes gives rise to a morphism
 \begin{equation}\label{eqn:injectivity_of_D_1}
  \chi_{*}\from  \check{H}^{2}(\bB G_{\bullet},\pi\to 0)\to
  \check{H}^{2}(\bB G_{\bullet},\pi_{2}(G)\xrightarrow{\per_{\omega}}\fz),
 \end{equation}
 and we claim that the kernel of $D$ coincides with the subgroup
 $\chi_{*}(\check{H}^{2}(\bB G_{\bullet},\pi\to 0))$. This is in general
 non-trivial, for instance if $\pi_{2}(G)$ is finitely generated. Then $\pi$ is
 a direct summand in $\pi_{2}(G)$ since $\fz$ is torsion free and thus each
 generator $a$ is either contained in $\ker(\pi)$ or $\per_{\omega}$ is
 injective on $\langle a\rangle$. If $\pi$ is a direct summand in $\pi_{2}(G)$,
 then $\chi_{*}$ is clearly injective.
 
 In order to verify the claim, observe that if $D(\phi)=0$, then we can assume
 that the $\fz$-valued components of $\phi$ are locally constant. Consequently,
 $\phi$ can also be considered as a cocycle representing an element in
 $\check{H}^{2}(\bB G_{\bullet},\pi_{2}(G)\xrightarrow{\per_{\omega}}\fz^{\delta})$,
 where $\fz^{\delta}$ denotes the abelian Lie group $\fz$ with the discrete
 topology. From Remark \ref{rem:relation_to_Neeb02} we have the short exact
 sequence
 \begin{equation*}
  \chi_{*}( \check{H}^{2}(\bB G_{\bullet},\pi\to 0))\hookrightarrow
  \check{H}^{2}(\bB G_{\bullet},\pi_{2}(G)\xrightarrow{\per_{\omega}}\fz^{\delta})\xrightarrow{\zeta_{*}} 
  \check{H}^{2}(\bB G_{\bullet},0\to Z^{\delta})
 \end{equation*}
 with $Z^{\delta}:=\fz^{\delta}/\per_{\omega}(\pi_{2}(G))$. From covering
 theory it follows that
 \begin{equation*}
  \check{H}^{2}(\bB G_{\bullet},0\to {Z}^{\delta})\cong\op{Ext}(G,Z^{\delta})\cong \Hom(\pi_{1}(G),Z),
 \end{equation*}
 and this vanishes for $G$ being simply connected. This shows the claim.
\end{remark}

\begin{tabsection}
 This now readily implies our generalization of Lie's Third Theorem.
\end{tabsection}
 
 \begin{theorem}
  If $\fg$ is a locally convex locally exponential Lie algebra such that
  $\fz:=\fz(\fg)\se\fg$ is a complemented and Mackey-complete subspace, then there exists an
  \'etale \tgp $G$ with $L(G)\cong \fg$.
 \end{theorem}
 
 \begin{proof}
  Since $\fg$ is locally exponential, $\fg_{\op{ad}}:=\fg/\fz$ has
  $G_{\op{ad}}\se\Aut(\fg)$ as Lie group with $L(G_{\op{ad}})=\fg_{\op{ad}}$
  \cite[Theorem IV.3.8]{Neeb06Towards-a-Lie-theory-of-locally-convex-groups}.
  Since $\fz$ is complemented, we have that
  $\fg\cong \fz\oplus_{\omega_{\op{ad}}}\fg_{\op{ad}}$ for the continuous
  cocycle $\omega_{\op{ad}}\from \fg_{\op{ad}}\times\fg_{\op{ad}}\to\fz$,
  determined by $(x,y)\mapsto [x,y]_{1}$, where $[x,y]_{1}$ is the
  $\fz$-component with respect so some chosen topological isomorphism
  $\fg\cong \fz\times\fg_{\op{ad}}$. Thus we may apply the preceding lemma to
  the central extension
  \begin{equation*}
   \fz(\fg)\to\fg\to\fg_{\op{ad}},
  \end{equation*}
  which shows the claim.
 \end{proof}
 
\section{An interpretation in terms of 2-connected covers}
\label{sec:an_interpretation_in_terms_of_2_connected_covers}

\begin{tabsection}
 The integration of the Lie algebra cocycle $\omega$ to a locally smooth group
 cocycle is obstructed by $\pi_{2}(G)$ (or more precisely by
 $\per_{\omega}(\pi_{2}(G))$). In any case, the obstruction vanishes if
 $\pi_{2}(G)$ does (which is in particular the case for finite-dimensional
 Lie groups). For non-simply connected Lie groups one can always pass to the
 simply connected cover to resolve obstructions coming from non-trivial
 fundamental groups, but $\pi_{2}(G)$ cannot be ruled out in a similar fashion
 for the following reasons.
 \begin{itemize}
  \item The 2-connected ``cover'' ${G}^{\sharp}\to G$ exists as a topological
        group, but it might not have a Lie group structure.
  \item Even if $G^{\sharp}$ has a Lie group structure and $G^{\sharp}\to G$ is
        a submersion (for instance for certain loop groups or
        $U(\ell^{2}(\mathbb{C}))\to PU(\ell^{2}(\mathbb{C}))$), then
        $G^{\sharp}$ will in general not have a Lie algebra isomorphic to the
        one of $G$, since the kernel of $G^{\sharp}\to G$ cannot be discrete.
 \end{itemize}
 However, the central extension $\Pi_{2}(G)\to\disc{G}$ from Example
 \ref{ex:2-connected-cover} takes over the r\^ole of the 2-connected cover of
 $G$. This is made precise by the following theorem. Note that the underlying
 Lie groupoid of $\Pi_{2}(G)$ determines a simplicial manifold
 $\bB \Pi_{2}(G)_{\bullet}$. If we now assume that $G$ is metrizable, then each
 $\bB \Pi_{2}(G)_{n}$ is so and thus $\bB \Pi_{2}(G)_{\bullet}$ is in particular a
 simplicial group object in the category of compactly generated Hausdorff spaces. On
 these objects the (ordinary) geometric realization
 $|\Pi_{2}(G)|:=|\bB \Pi_{2}(G)_{\bullet}|$ is particularly well-behaved. Under
 these requirements we can now show that the canonical map
 $|\Pi_{2}(G)|\xrightarrow{q} G$, induced by $\ev\from P_{e}G\to G$, is a
 2-connected cover, i.e., is a fibration, $\pi_{2}(|\Pi_{2}(G)|)$ vanishes and
 $\pi_{i}(q)$ is an isomorphism for $i\neq 2$
\end{tabsection}
 
 \begin{theorem}
  If $G$ is a metrizable 1-connected Lie group, then
  $|\Pi_{2}(G)|\xrightarrow{q} G$ is a 2-connected cover of $G$. Moreover, the
  multiplication functor of $\Pi_{2}(G)$ induces on $|\Pi_{2}(G)|$ the
  structure of a group object in the category of compactly generated Hausdorff
  spaces and $q$ is a morphism thereof.
 \end{theorem}
 
 \begin{proof}
  (cf.\ \cite[Theorem
  28]{BaezCransStevensonSchreiber07From-loop-groups-to-2-groups}) From
  \cite[Proposition
  4.9]{NikolausSachseWockel13A-smooth-model-for-the-string-group} it follows
  that $|\Pi_{2}(G)|\xrightarrow{q} G$ is a fibration and that the fibre is a
  $K(\pi_{2}(G),1)$. Note that the latter agrees with the geometric realization
  of
  \begin{equation*}
   \cK:=\ker(\Pi_{2}(G)\to G)=\left(\wt{\Omega G}\rtimes \Omega G\rrarrow \Omega G\right).
  \end{equation*}
  Then the same argument as in \cite[Theorem
  28]{BaezCransStevensonSchreiber07From-loop-groups-to-2-groups} now shows that
  the claim follows if we can show that for the canonical morphism
  $\disc{\Omega G}\to \cK$ the induced morphism
  $\pi_{1}(|\disc{\Omega G}|)\to \pi_{1}(|\cK|)$ is an isomorphism. If we
  consider the diagram
  \begin{equation*}
   \vcenter{\xymatrix{
   \disc{\pi_{2}(G)}
   \ar[r]
   \ar[d]&
   \disc{\wt{\Omega G}}
   \ar[r]
   \ar[d]&
   \disc{\Omega G}
   \ar[d]\\
   \disc{\pi_{2}(G)}
   \ar[r]&
   \cE(\wt{\Omega G})
   \ar[r]&
   \cK
   }}
  \end{equation*}
  of \tgps (where $\cE(\wt{\Omega G})$ is the pair groupoid of
  $\wt{\Omega G}$), then the top row is exact by construction and the bottom
  row is exact by \cite[Proposition
  4.9]{NikolausSachseWockel13A-smooth-model-for-the-string-group}. Since
  $|\cE(\wt{\Omega G})|$ is contractible by \cite[\S
  3]{Segal68Classifying-spaces-and-spectral-sequences} and \cite[Lemma
  4.4]{NikolausSachseWockel13A-smooth-model-for-the-string-group} and
  $\pi_{1}(\disc{\Omega G})\to \pi_{0}(\disc{G})$ is an isomorphism by
  construction, the claim follows from the commutativity of
  \begin{equation*}
   \vcenter{\xymatrix{
   \pi_{1}(|\disc{\Omega G}|)
   \ar[r]^{\cong}
   \ar[d]&
   \pi_{0}(|\disc{\pi_{2}(G)}|)
   \ar[d]^{\cong}\\
   \pi_{1}(|\cK|)
   \ar[r]^{\cong}&
   \pi_{0}(|\disc{\pi_{2}(G)}|)
   }}.
  \end{equation*}
~
 \end{proof}
 
\begin{tabsection}
 Thus $\Pi_{2}(G)\to\disc{G}$ may be interpreted as a 2-connected cover.
 Moreover, there is a smooth version $\Pi_{2}^{\infty}(G)$ of Example
 \ref{ex:2-connected-cover}, where one replaces $P_{e}G$ with
 \begin{equation}
  P_{e}^{\infty} G:=\{\gamma\in C^{\infty}([0,1],G)\mid \gamma(0)=e\}
 \end{equation}
 and $\Omega G$ with
 \begin{equation}
  \ker(\ev)\cap P_{e}^{\infty}G:= \Omega^{\infty} G:=\{\gamma\in C^{\infty}([0,1],G)\mid \gamma(0)=e=\gamma(1)\}.
 \end{equation}
 Since the inclusions $P_{e}^{\infty}G \hookrightarrow P_{e}G $ and
 $\Omega ^{\infty}G\to \Omega G$ are homotopy equivalences it follows from
 \cite[Proposition
 4.5]{NikolausSachseWockel13A-smooth-model-for-the-string-group} that
 $|\Pi_{2}^{\infty}(G)|\to G$ is also a 2-connected cover.
 
 If we apply the construction from Section
 \ref{sec:differentiable_cocycles} to $\fz=0$
 and thus the zero cocycle, then we obtain a differentiable cocycle $\phi_{0}$
 and by Proposition \ref{prop:central-extensions-from-differentiable-cocycles}
 a central extension
 \begin{equation*}
  [\pi_{2}(G)\to 0]\to \wh{G}_{\phi_{0}}\to \disc{G}.
 \end{equation*}
 As above, the geometric realization $|\wh{G}_{\phi_{0}}|\to G$ is a
 2-connected cover of $G$, but now only as an $A_{\infty}$-space since
 $|\wh{G}_{\phi_{0}}|$ is in general not a group (cf.\ \cite[\S
 3.7]{Schommer-Pries10Central-Extensions-of-Smooth-2-Groups-and-a-Finite-Dimensional-String-2-Group})\footnote{This
 can be made precise by constructing an explicit morphism
 $\wh{G}_{\phi_{0}}\to \Pi_{2}^{\infty}(G)$ with the aid of the maps $\aone{}$
 and $\atwo{}$. However, while the construction in Section
 \ref{sec:differentiable_cocycles} relies on concatenations of paths and
 triangles, the construction of $\Pi_{2}^{\infty}(G)$ relies on point-wise
 multiplication. This is the reason why the morphism
 $\wh{G}_{\phi_{0}}\to \Pi_{2}^{\infty}(G)$ would respect the composition (and
 also the group structure) only up to homotopy. Explicit formulae for this map
 would exceed its use.}. Moreover, we have a canonical morphism
 $\wh{G}_{\phi}\to \wh{G}_{\phi_{0}}$, induced by forgetting the
 $\fz$-component in each term. This then gives rise to a central extension
 \begin{equation*}
  \disc{\fz}\to \wh{G}_{\phi}\to \wh{G}_{\phi_{0}}
 \end{equation*}
 (in the more general setting from \cite[\S
 3.6]{Schommer-Pries10Central-Extensions-of-Smooth-2-Groups-and-a-Finite-Dimensional-String-2-Group}),
 which can be seen as a central extension of the 2-connected cover
 $\wh{G}_{\phi_{0}}$ of $G$. In this light the construction from Section
 \ref{sec:differentiable_cocycles} looks as if we have first passed to the
 2-connected cover $\wh{G}_{\phi_{0}}$, where we then have solved the integration
 problem which is trivial due to the 2-connectedness of $\wh{G}_{\phi_{0}}$.
\end{tabsection}
 
 \appendix
 
 \section{Appendix: Differential calculus on locally convex spaces}
 \label{sect:appendix}
 
\begin{tabsection}
 We provide some background material on locally convex Lie groups and their Lie
 algebras in this appendix. See also
 \cite{Wockel13Infinite-dimensional-and-higher-structures-in-differential-geometry} for more details on this.
\end{tabsection}
 
\begin{definition}
 \label{def:diffcalcOnLocallyConvexSpaces} Let \mbox{ $X$} and \mbox{ $Y$} be
 locally convex spaces and \mbox{ $U\se X$} be open. Then \mbox{
 $f\from U\to Y$} is \emph{differentiable} or \emph{\mbox{ $C^{1}$}} if it is
 continuous, for each \mbox{ $(x,v)\in U\times X$} the differential
 \[
  df (x).v:=\lim_{h\to 0}\frac{f (x+hv)-f (x)}{h}
 \]
 exists and if the map \mbox{ $df\from U\times X\to Y$} is continuous. If
 \mbox{ $n>1$} we inductively define \mbox{ $f$} to be \emph{\mbox{ $C^{n}$}}
 if it is \mbox{ $C^{1}$} and \mbox{ $df$} is \mbox{ $C^{n-1}$} and to be
 \mbox{ $C^{\infty}$} or \emph{smooth} if it is \mbox{ $C^{n}$} for all
 $n\in\N$. We denote the corresponding spaces of maps by \mbox{ $C^{n}(U,Y)$}
 and \mbox{ $C^{\infty}(U,Y)$}.
 
 A (locally convex) manifold is a topological Hausdorff space (without any
 further topological requirements) that is locally homeomorphic to open subsets
 of a locally convex space with smooth chart changes. A (locally convex)
 \textit{Lie group} is a group which is a smooth Hausdorff manifold modeled on
 a locally convex space such that the group operations are smooth. A locally
 convex Lie algebra is a Lie algebra, whose underlying vector space is locally
 convex and whose Lie bracket is continuous.
\end{definition}
 
 The proof of the following theorem is standard, see for instance \cite[Proposition
 III.1.9.18]{Bourbaki98Lie-groups-and-Lie-algebras.-Chapters-1--3}.
 
 \begin{theorem}\label{thm:globalisation-of-smooth-structures-on-groups}
  Let $G$ be a group $U\se G$ be a subset containing $e$ and let $U$ be
  endowed with a manifold structure. Moreover, assume that there exists an open
  neighborhood $V\se U$ of $e$ such that
  \begin{itemize}
   \item [i)] $V^{-1}=V$ and $V\cdot V\se U$,
   \item [ii)] $V\times V\ni (g,h)\mapsto gh \in U$ is smooth,
   \item [iii)] $V\ni g\mapsto g^{-1}\in V$ is smooth and
   \item [iv)] $V$ generates $G$ as a group.
  \end{itemize}
  Then there exists a manifold structure on $G$ such that $V$ is open in $G$
  and such that group multiplication and inversion is smooth. Moreover, for
  each other choice of $V$, satisfying the above conditions, the resulting
  smooth structures on $G$ coincide.
 \end{theorem}

 \begin{definition}\label{def:exponentialfunction}
  Let \mbox{\mbox{ $G$}} be a locally convex Lie group. The group \mbox{\mbox{
  $G$}} is said to have an \emph{exponential function} if for each \mbox{\mbox{
  $x \in \fg$}} the initial value problem
  \[
   \gamma (0)=e,\quad \gamma(t)^{-1}\cdot \gamma'(t)=x
  \]
  has a solution \mbox{\mbox{ $\gamma_{x}\in C^{\infty} (\R,G)$}} and the
  function
  \[
   \exp_{G}:\fg\to G,\;\;x \mapsto \gamma_x (1)
  \]
  is smooth. Furthermore, if there exists a zero neighborhood \mbox{\mbox{
  $W\se \fg$}} such that \mbox{\mbox{ $\left.\exp_{G}\right|_{W}$}} is a
  diffeomorphism onto some open identity neighborhood of \mbox{\mbox{ $G$}},
  then \mbox{\mbox{ $G$}} is said to be \emph{locally exponential}.
 \end{definition}
 
 The Fundamental Theorem of
  Calculus for locally convex spaces (cf.\ \cite[Theorem
  1.5]{Glockner02Infinite-dimensional-Lie-groups-without-completeness-restrictions})
  yields that a locally convex Lie group \mbox{\mbox{ $G$}} can have at most
  one exponential function (cf.\ \cite[Lemma
  II.3.5]{Neeb06Towards-a-Lie-theory-of-locally-convex-groups}).
  
  Typical examples of locally exponential Lie groups are Banach-Lie groups (by
  the existence of solutions of differential equations and the inverse mapping
  theorem, cf.\ \cite{Lang99Fundamentals-of-differential-geometry}) and groups
  of smooth and continuous mappings from compact manifolds into locally
  exponential groups (\cite[\S
  3.2]{Glockner02Lie-group-structures-on-quotient-groups-and-universal-complexifications-for-infinite-dimensional-Lie-groups},
  \cite{Wockel06Smooth-extensions-and-spaces-of-smooth-and-holomorphic-mappings}).
  However, diffeomorphism groups of compact manifolds are never locally
  exponential (cf.\ \cite[Example
  II.5.13]{Neeb06Towards-a-Lie-theory-of-locally-convex-groups}) and direct
  limit Lie groups not always (cf.\ \cite[Remark
  4.7]{Glockner05Fundamentals-of-direct-limit-Lie-theory}). For a detailed
  treatment of locally exponential Lie groups and their structure theory we
  refer to \cite[Section
  IV]{Neeb06Towards-a-Lie-theory-of-locally-convex-groups}.
 
 \begin{definition}\label{def:loc-exp}
  A locally convex Lie algebra $\fg$ is said to be locally exponential if there
  exists a circular convex open zero neighborhood $U\se \fg$ and an open
  subset $D\se U\times U$ on which there exists a smooth map
  \[
   m_{U}\from D\to U,\quad (x,y)\mapsto x*y
  \]
  such that $(D,U,m_{U},0)$ is a local Lie group and such that the following
  holds.
  \begin{itemize}
   \item [i)] For $x\in U$ and $|t|, |s|, |t+s| \leq 1$, we have
         $(tx, sx) \in D$ with $tx*sx=(t+s)x$.
   \item [ii)] The second order term in the Taylor expansion of $m_{U}$ in $0$
         is $b(x,y)=\frac{1}{2}[x,y]$.
  \end{itemize}
 \end{definition}
 
  As above, one
  has that all Banach-Lie algebras are locally exponential, as well as all Lie
  algebras of locally exponential groups (cf.\ \cite[Example
  IV.2.4]{Neeb06Towards-a-Lie-theory-of-locally-convex-groups}).
 
 \begin{Theorem}
  (\cite[Theorem
  IV.3.8]{Neeb06Towards-a-Lie-theory-of-locally-convex-groups})\label{thm:gAdIsEnlargible}
  Let $\fg$ be a locally exponential Lie algebra. Then the adjoint group
  $G_{\ad}\leq \Aut(\fg)$ carries the structure of a locally exponential Lie
  group whose Lie algebra is $\fg_{\ad}:=\fg /Z(\fg)$.
 \end{Theorem}
 
\begin{wrapfigure}{r}{0.4\textwidth}
	\centering
 \includegraphics[width=0.4\textwidth]{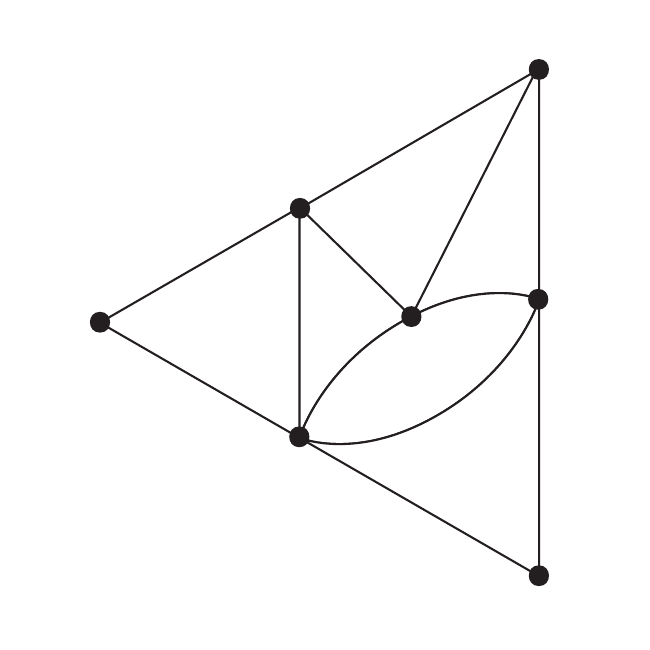}%
 \caption{The triangulation of $\Sigma$} \label{fig:Triangulation}
\end{wrapfigure}

 Recall that a locally convex space $X$ is said to be \emph{convenient} or
\emph{Mackey-complete} if each element of $C^{\infty}([0,1],X)$ has an integral
in $X$ (cf.\ \cite[\S
I.2]{KrieglMichor97The-Convenient-Setting-of-Global-Analysis}). In particular,
complete spaces are convenient \cite[Theorem
I.2.14]{KrieglMichor97The-Convenient-Setting-of-Global-Analysis}. In what
follows we silently assume that $\fz$ is a Fr\'echet space. Then all claims
follow from the results of
\cite{Wockel06Smooth-extensions-and-spaces-of-smooth-and-holomorphic-mappings}.
All what we claim also stays valid if $\fz$ is only assumed to be convenient,
the more involved arguments for this case can be found in
\cite{KrieglMichor97The-Convenient-Setting-of-Global-Analysis}.
 
 We treat $\Delta^{(n)}\se \R^{n}$ as a manifold with corners as in
\cite{Wockel06Smooth-extensions-and-spaces-of-smooth-and-holomorphic-mappings}.
Thus a map $f\from \Delta^{(n)}\to G$ is called \emph{ $C^{1}$} if it is
differentiable in the interior $\op{int}(\Delta^{(n)})$ and in each local chart
of $G$ the differentials $(x,v)\mapsto df(x).v$ extend continuously to the
boundary $\partial(\Delta^{(n)})$. It is called \emph{ $C^{k}$} if it is
$C^{1}$ and the differential is $C^{k-1}$ and smooth if it is $C^{k}$ for each
$k\in \N$. From this one defines a smooth singular $n$-chain to be a formal sum
$f_{1}+...+f_{n}$ of smooth maps $f_{i}\from \Delta^{(n)}\to G$ and denotes by
$\mathscr{C}_{n}(G)$ the abelian group of smooth $n$-chains. Since the coface
maps $\Delta^{(n-1)}\to \Delta^{n}$ are smooth, $\psing$ restricts to a
boundary operator $\psing\from \mathscr{C}_{n}(G)\to \mathscr{C}_{n-1}(G)$.
 
 If $f\from [0,1]^{n}\to \fz$ is smooth, then
 \begin{equation*}
  \hat{f}_{n}\from [0,1]\to C^{\infty}([0,1]^{n-1},\fz),\quad
  s\mapsto ((x_{1},...,x_{n-1})\mapsto f(x_{1},...,x_{n-1},s))
 \end{equation*}
 is smooth. Since $C^{\infty}([0,1]^{n-1},\fz)$ is again Fr\'echet
 (respectively convenient), it follows from \cite[Theorem
 I.2.14]{KrieglMichor97The-Convenient-Setting-of-Global-Analysis} that
 $\int_{0}^{1}f\mathop{dx_{n}}:=\int_{0}^{1}\hat{f}_{n}$ exists in $\fz$. We
 thus may define the iterated integral
 \begin{equation*}
  \int_{[0,1]^{n}}f:=\int_{0}^{1}...\left(\int_{0}^{1}f \mathop{dx_{n}}\right)...\mathop{dx_{1}}
 \end{equation*}
 If $\omega$ is a smooth $n$-form on $[0,1]^{n}$, then we set as usual
 \begin{equation*}
  \int_{[0,1]^{n}} \omega:=\int _{[0,1]^{n}}\omega( \partial_{x_{n}},\ldots,\partial_{x_{1}})
 \end{equation*}
 
 For $q_{n}\from [0,1]^{n}\to \Delta^{(n)}$ a smooth map that restricts to a
 diffeomorphism on the interior (e.g., take
 \begin{equation*}
  (x_{1},...,x_{n})\mapsto ((1-x_{2})\cdots(1-x_{n})x_{1},...,(1-x_{n})x_{n-1},x_{n})
 \end{equation*}
 as such a map) and $\omega\in \Omega^{n}(\Delta^{(n)},\fz)$, we then define
 \begin{equation*}
  \int_{\Delta^{(n)}}\omega:=\int_{[0,1]^{n}}\mathop{q_{n}^{*}}\omega.
 \end{equation*}
 For $\sigma\in C^{\infty}(\Delta^{(n)},G)$ and $\omega\in \Omega^{n}(G,\fz)$
 we define $\int_{\sigma}\omega:=\int_{\Delta^{(n)}}\sigma^{*}\omega$ and,
 moreover, for each $\sigma\in \mathscr{C}_{n}(G)$ by additive extension.

 \begin{proposition}\label{prop:smoothness_of_integration}
  The map
  \begin{equation}\label{eqn:integration_is_smooth}
   C^{\infty}(\Delta^{(n)},G) \to \fz,\quad \sigma\mapsto \int_{\sigma}\omega
  \end{equation}
  defined in the previous remark is smooth, where we endow
  $C^{\infty}(\Delta^{(n)},G)$ with the smooth structure from
  \cite{Wockel06Smooth-extensions-and-spaces-of-smooth-and-holomorphic-mappings}.
 \end{proposition}
 
 \begin{proof}
  Since an atlas on $C^{\infty}(\Delta^{(n)},G)$ is given by
  $\gamma\mapsto \varphi \circ (\gamma^{-1}\cdot \gamma')$ for
  $\varphi\from U\to \wt{U}\se \fg$ a chart of $G$ and
  $\gamma '\in \gamma\cdot C^{\infty}(\Delta^{(n)},U)$, it clearly suffices to
  show that the map
  \begin{equation*}
   C^{\infty}(\Delta^{(n)},\fg) \to \fz,\quad \sigma\mapsto \int_{\sigma}\wt{\omega}=\int ...\int\left(\mathop{(\sigma \circ q_{n})^{*}}\wt{\omega}\right)( \partial_{x_{n}},\ldots,\partial_{x_{i}})\mathop{dx_{n}}...\mathop{dx_{1}}
  \end{equation*}
  is smooth for each $\wt{\omega}\in \Omega^{n}(\Delta^{(n)},\fg)$. Since
  $C^{\infty}(\Delta^{(n)},\fg)\ni \sigma \mapsto d \sigma\in C^{\infty}(T \Delta^{(n)},\fg)$
  is linear and continuous (by the definition of the topology on
  $C^{\infty}(\Delta^{(n)},\fg)$), and since evaluation is smooth,
  $d \sigma(Tq_{n}(\partial_{x_{i}}))\in C^{\infty}([0,1]^{n},\fg)$ depends
  smoothly on $\sigma$. By the definition of a smooth $n$-form (cf.\
  \cite[Definition I.4.1]{Neeb06Towards-a-Lie-theory-of-locally-convex-groups}),
  $\omega$ is an element of $C^{\infty}(\fg\times\fg^{n},\fz)$ and since push
  forward is a smooth map on mapping spaces by \cite[Corollary
  29]{Wockel06Smooth-extensions-and-spaces-of-smooth-and-holomorphic-mappings},
  \begin{equation*}
   \wt{\omega}_{q_{n}(\cdot)}(d \sigma(Tq_{n}(\partial_{x_{n}})),...,d \sigma(Tq_{n}(\partial_{x_{n}}))))\in C^{\infty}([0,1]^{n},\fz)
  \end{equation*}
  depends smoothly on $\sigma$.
  
  Now the integration map, sending a smooth curve to its integral, is
  continuous and linear, so in particular smooth. This implies the smoothness
  of \eqref{eqn:integration_is_smooth}.
 \end{proof}
 
 \begin{remark}\label{rem:piece-wise-smooth-maps}
  Suppose $\Sigma=\cup_{i\in I}\tau_{i}$ is a simplicial complex, where
  $\tau_{i}$ are homeomorphic images of $\Delta^{(n_{i})}$ for some
  $n_{i}\in\N_{0}$. We call $\Sigma$ a \emph{$p$-complex}, if $n_{i}=p$ for all
  $i\in I$ and \emph{finite} if $I$ is finite. Unless mentioned otherwise,
  $\Sigma$ will always refer to the simplicial 2-complex from Figure
  \ref{fig:Triangulation}.

  A map $f\from \Sigma\to G$ is called piece-wise smooth if it is continuous
  and $\left.f\right|_{\tau_{i}}$ is smooth for all $i\in I$. The space
  $C^{\infty}_{pw}(\Sigma,G)$ of piece-wise smooth maps is then also a Lie
  group with Lie algebra $C^{\infty}_{pw}(\Sigma,\fg)$, the construction from
  \cite[\S
  4]{Wockel06Smooth-extensions-and-spaces-of-smooth-and-holomorphic-mappings}
  or \cite[\S
  3.2]{Glockner02Lie-group-structures-on-quotient-groups-and-universal-complexifications-for-infinite-dimensional-Lie-groups}
  carries over to this slightly more general situation. The integral of a
  $\fz$-valued $p$-form $\omega$ over a finite $p$-simplex is defined by
  additive extension. Since the restriction of piece-wise smooth maps to the
  simplices is smooth it follows from Proposition
  \ref{prop:smoothness_of_integration} that
  \begin{equation*}
   C^{\infty}_{pw}(\Sigma,G)\to \fz,\quad \sigma\mapsto \sum_{i\in I} \int_{\left.\sigma\right|_{\tau_{i}}}\omega
  \end{equation*}
  is also smooth.
 \end{remark}
 
 \section{Appendix: Lie groupoids and the submersion Grothendieck pretopology on locally
 convex manifolds} \label{sec:grothendieck-pretop} 

\begin{tabsection}
 Our Lie group(oid)s are based on locally convex manifolds. More precisely they
 are group(oid) objects in the category of locally convex manifolds endowed
 with a Grothendieck pretopology that we will explain in this appendix. See
 also
 \cite{Wockel13Infinite-dimensional-and-higher-structures-in-differential-geometry}
 for more details on this.
 
 \begin{definition}
  A \emph {singleton Grothendieck pretopology} $\cT$
  on a category $\cC$ is a collection of morphisms, called covers, subject the
  following three axioms:
  \begin{enumerate}
   \item Isomorphisms are covers.
   \item The composition of two covers is a cover.
   \item If $U\to X$ is a cover and $Y\to X$ is a morphism, then the pull-back
         $Y\times_X  U$ is representable, and the natural morphism
         $Y\times_X U \to Y$ is a cover.
  \end{enumerate}
 \end{definition}
 
 A smooth map $f : M \to N$ of manifolds is a {\em submersion} if for each
 $x\in M$ there exist charts around $x$ and $f(x)$ such that the coordinate
 representation of $f$ is a projection. In this case the pull-back of an
 arbitrary smooth map of a submersion is a manifold and the canonical map
 $Z \times_N M \to  Z$ is again a submersion. Thus the category of such
 manifolds $\cC_0$ has a Grothendieck pretopology $\cT_0$ whose covers are
 surjective submersions \cite[Appendix
 B]{NikolausSachseWockel13A-smooth-model-for-the-string-group}.
 
 A groupoid object in $(\cC, \cT)$ is a groupoid $G_1\rrarrow G_0$ such that
 both $G_1, G_0 \in \cC$, all the structure maps are morphisms in $\cC$, and
 the source and target maps $\bt, \bs$ are covers in $\cT$. When $G_0$ is the
 terminal object $*$ in $\cC$, a groupoid object is a group object. The theory
 of groupoids object, stacks, weak equivalence, and generalized morphisms are
 well known in various categories (see for example \cite{BehrendXu11Differentiable-stacks-and-gerbes,Metzler03Topological-and-Smooth-Stacks,MoerdijkMrcun03Introduction-to-foliations-and-Lie-groupoids,Noohi08Notes-on-2-groupoids-2-groups-and-crossed-modules} and references therein). Such a theory for higher groupoids for general
 $(\cC, \cT)$ is partially developed in
 \cite{Zhu09n-groupoids-and-stacky-groupoids}. Our Lie group(oid) is a
 group(oid) object in $(\cC_{0},\cT_{0})$.
 
 We briefly recall here for 1-groupoids to make our paper self-contained the
 following
 \begin{definition}\label{defequivalence}
  \emptycomment{ A morphism $f:Z\to X$ of groupoid objects in $(\cC, \cT)$ is a
  \emph {hypercover}\footnote{This concept can be extended to higher groupoid
  objects, see \cite{Zhu09n-groupoids-and-stacky-groupoids}. It is motivated
  from the definition of a trivial fibration of Quillen.} if
  \begin{itemize}
   \item the map $Z_0\to X_0$ is a cover in $\cT$;
   \item the natural map from $Z_1$ to the pull-back
         $X_1\times_{X_0     \times X_0} Z_0 \times Z_0$ is an isomorphism.
  \end{itemize}}
  A morphism $f:Z\to X$ of groupoid objects in $(\cC, \cT)$ is a \emph {weak
  equivalence} (see \cite{MoerdijkMrcun03Introduction-to-foliations-and-Lie-groupoids} if
  \begin{itemize}
   \item the map $Z_0\times_{ X_0} X_1 \to X_0$ is a cover in $\cT$;
   \item the natural map from $Z_1$ to the pull-back
         $Z_0\times_{ X_0} X_1 \times_{X_0} Z_0(=X_1\times_{X_0     \times X_0} Z_0 \times Z_0)$
         is an isomorphism.
  \end{itemize}
 \end{definition}
 
 \begin{definition}\label{defi:gen-morp}
  A \emph {generalized morphism} between two groupoid objects $X$ and $Y$ in
  $(\cC, \cT)$ consists of a \emph {span} of morphisms
  $X\stackrel{\sim}{\leftarrow}Z\to Y$, where the map $Z\stackrel{\sim}{\to} X$
  is a weak equivalence.
 \end{definition}
 \emptycomment{
 \begin{remark}
  One may also use weak equivalence instead of a hypercover to form such a span
  to define a generalized morphism. However, there is no essential difference
  between the two methods, that is, the 2-category formed using these two sorts
  of generalized morphisms are equivalent.
 \end{remark}
 } Weak equivalences are in general not invertible as morphisms of groupoid
 objects. One has to take the symmetric closure to come to an equivalence
 relation (due to the axioms of the pretopology we only have to take spans of
 weak equivalences instead of arbitrary zig-zags to achieve this).
 
 \begin{definition}\label{defi:m-equi-2gpd}
  Two groupoid objects $X$ and $Y$ in $(\cC, \cT)$ are \emph {Morita
  equivalent} if there is another groupoid $Z$ in $(\cC, \cT)$ and weak
  equivalences
  $ X\stackrel{\sim}{\leftarrow} Z \stackrel{\sim}{\rightarrow} Y$.
 \end{definition}
 
 A {\em morphism between two generalized morphisms}
 $X\stackrel{\sim}{\leftarrow} Z \to Y$ and
 $X\stackrel{\sim}{\leftarrow} Z' \to Y$ is given by a third zig-zag
 $X\stackrel{\sim}{\leftarrow} Z'' \to Y$ such that $Z''$ also provides a weak
 equivalence $Z\stackrel{\sim}{\leftarrow} Z'' \stackrel{\sim}{\to} Z'$ and all
 triangle diagrams commute up to 2-morphisms of groupoid objects:
 \[
  \xymatrix{ & Z \ar[dl]_\sim \ar[dr] & \\
  X & Z'' \ar[l]_\sim \ar[r] \ar[d] \ar[u] & Y \\ 
  & Z' \ar[ul]^\sim \ar[ur] &}
 \]
 
 The bicategory of Lie groupoids we use in this article has Lie groupoids in
 $(\cC_0, \cT_0)$ as objects, generalized morphisms as 1-morphisms, and
 morphisms between generalized morphisms as 2-morphisms. Morita equivalent Lie
 groupoids one-to-one corresponds presentable stacks in various categories with
 Grothendieck pretopologies (see for example
 \cite{BehrendXu11Differentiable-stacks-and-gerbes,Metzler03Topological-and-Smooth-Stacks,
 Noohi08Notes-on-2-groupoids-2-groups-and-crossed-modules,Pronk96Etendues-and-stacks-as-bicategories-of-fractions}).
 The same technique will apply here for our case of $(\cC_0, \cT_0)$. Thus our
 Lie 2-group is in some sense also a stacky Lie group in the sense of
 \cite{Blohmann08Stacky-Lie-groups}.
\end{tabsection}

\bigskip
\footnotesize
\noindent\textit{Acknowledgments.}
 The authors would like to thank the referees of this paper for helpful
 suggestions that lead to several improvements.

\bibliographystyle{new}

\begin{thebibliography}{FHRZ08}
\providecommand{\url}[1]{\texttt{#1}}
\providecommand{\urlprefix}{URL }
\expandafter\ifx\csname urlstyle\endcsname\relax
  \providecommand{\doi}[1]{doi:\discretionary{}{}{}#1}\else
  \providecommand{\doi}{doi:\discretionary{}{}{}\begingroup
  \urlstyle{rm}\Url}\fi
\providecommand{\eprint}[2][]{\url{#2}}

\bibitem[Bar06]{Bartels06Higher-gauge-theory-I:-2-Bundles}
Bartels, T.
\newblock \emph{{H}igher gauge theory {I}: 2-{B}undles} 2006.
\newblock \href{http://arxiv.org/abs/math/0410328}{\texttt{arXiv:math/0410328}}

\bibitem[BCSS07]{BaezCransStevensonSchreiber07From-loop-groups-to-2-groups}
Baez, J.~C., Crans, A.~S., Stevenson, D. and Schreiber, U.
\newblock \emph{From loop groups to 2-groups}.
\newblock Homology, Homotopy Appl. \textbf{9} (2007)(2):101--135.
\newblock \href{http://arxiv.org/abs/math/0504123}{\texttt{arXiv:math/0504123}}

\bibitem[BL04]{BaezLauda04Higher-dimensional-algebra.-V.-2-groups}
Baez, J.~C. and Lauda, A.~D.
\newblock \emph{Higher-dimensional algebra. {V}. 2-groups}.
\newblock Theory Appl. Categ. \textbf{12} (2004):423--491 (electronic).
\newblock \href{http://arxiv.org/abs/math.QA/0307200}{\texttt{arXiv:math.QA/0307200}}

\bibitem[Blo08]{Blohmann08Stacky-Lie-groups}
Blohmann, C.
\newblock \emph{Stacky {L}ie groups}.
\newblock Int. Math. Res. Not. IMRN  (2008):Art. ID rnn 082, 51

\bibitem[Bou98]{Bourbaki98Lie-groups-and-Lie-algebras.-Chapters-1--3}
Bourbaki, N.
\newblock \emph{Lie groups and {L}ie algebras. {C}hapters 1--3}.
\newblock Elements of Mathematics (Berlin) (Springer-Verlag, Berlin, 1998).
\newblock Translated from the French, Reprint of the 1989 English translation

\bibitem[Bre90]{Breen90Bitorseurs-et-cohomologie-non-abelienne}
Breen, L.
\newblock \emph{Bitorseurs et cohomologie non ab{\'e}lienne}.
\newblock In \emph{The {G}rothendieck {F}estschrift, {V}ol.\ {I}}, \emph{Progr.
  Math.}, vol.~86, pp. 401--476 (Birkh{\"a}user Boston, Boston, MA, 1990)

\bibitem[Bre94]{Breen94On-the-classification-of-2-gerbes-and-2-stacks}
Breen, L.
\newblock \emph{On the classification of {$2$}-gerbes and {$2$}-stacks}.
\newblock Ast\'erisque \textbf{225} (1994):160

\bibitem[Bry93]{Brylinski93Loop-spaces-characteristic-classes-and-geometric-quantization}
Brylinski, J.-L.
\newblock \emph{Loop spaces, characteristic classes and geometric
  quantization}, \emph{Progress in Mathematics}, vol. 107 (Birkh{\"a}user
  Boston Inc., Boston, MA, 1993)

\bibitem[BS76]{BrownSpencer76G-groupoids-crossed-modules-and-the-fundamental-groupoid-of-a-topological-group}
Brown, R. and Spencer, C.~B.
\newblock \emph{{$G$}-groupoids, crossed modules and the fundamental groupoid
  of a topological group}.
\newblock Nederl. Akad. Wetensch. Proc. Ser. A {\bf 79}=Indag. Math.
  \textbf{38} (1976)(4):296--302

\bibitem[BS07]{BaezSchreiber07Higher-gauge-theory}
Baez, J.~C. and Schreiber, U.
\newblock \emph{Higher gauge theory}.
\newblock In \emph{Categories in algebra, geometry and mathematical physics},
  \emph{Contemp. Math.}, vol. 431, pp. 7--30 (Amer. Math. Soc., Providence, RI,
  2007).
\newblock (math/0504123), \href{http://arxiv.org/abs/math/0511710}{\texttt{arXiv:math/0511710}}

\bibitem[BX11]{BehrendXu11Differentiable-stacks-and-gerbes}
Behrend, K. and Xu, P.
\newblock \emph{Differentiable stacks and gerbes}.
\newblock J. Symplectic Geom. \textbf{9} (2011)(3):285--341.
\newblock
  \urlprefix\url{http://projecteuclid.org/getRecord?id=euclid.jsg/1310388899}.
\newblock \href{http://arxiv.org/abs/math/0605694}{\texttt{arXiv:math/0605694}}

\bibitem[CF03]{CrainicFernandes03Integrability-of-Lie-brackets}
Crainic, M. and Fernandes, R.~L.
\newblock \emph{Integrability of {L}ie brackets}.
\newblock Ann. of Math. (2) \textbf{157} (2003)(2):575--620.
\newblock \href{http://dx.doi.org/10.4007/annals.2003.157.575}{\texttt{doi:10.4007/annals.2003.157.575}}

\bibitem[Con03]{Conrad03Cohomological-Descent}
Conrad, B.
\newblock \emph{Cohomological descent}.
\newblock preprint 2003.
\newblock \urlprefix\url{http://math.stanford.edu/\~conrad/}

\bibitem[Ded60]{Dedecker60Sur-la-cohomologie-non-abelienne.-I}
Dedecker, P.
\newblock \emph{Sur la cohomologie non ab{\'e}lienne. {I}}.
\newblock Canad. J. Math. \textbf{12} (1960):231--251

\bibitem[Del74]{Deligne74Theorie-de-Hodge.-III}
Deligne, P.
\newblock \emph{Th{\'e}orie de {H}odge. {III}}.
\newblock Inst. Hautes {\'E}tudes Sci. Publ. Math. \textbf{44} (1974):5--77

\bibitem[DL66]{DouadyLazard66Espaces-fibres-en-algebres-de-Lie-et-en-groupes}
Douady, A. and Lazard, M.
\newblock \emph{Espaces fibr\'es en alg\`ebres de {L}ie et en groupes}.
\newblock Invent. Math. \textbf{1} (1966):133--151

\bibitem[Dus79]{Duskin79Higher-dimensional-torsors-and-the-cohomology-of-topoi:-the-abelian-theory}
Duskin, J.
\newblock \emph{Higher-dimensional torsors and the cohomology of topoi: the
  abelian theory}.
\newblock In \emph{Applications of sheaves ({P}roc. {R}es. {S}ympos. {A}ppl.
  {S}heaf {T}heory to {L}ogic, {A}lgebra and {A}nal., {U}niv. {D}urham,
  {D}urham, 1977)}, \emph{Lecture Notes in Math.}, vol. 753, pp. 255--279
  (Springer, Berlin, 1979)

\bibitem[EK64]{EstKorthagen64Non-enlargible-Lie-algebras}
van Est, W.~T. and Korthagen, T.~J.
\newblock \emph{Non-enlargible {L}ie algebras}.
\newblock Nederl. Akad. Wetensch. Proc. Ser. A 67=Indag. Math. \textbf{26}
  (1964):15--31

\bibitem[Est58]{Est58A-group-theoretic-interpretation-of-area-in-the-elementary-geometries.}
van Est, W.~T.
\newblock \emph{A group theoretic interpretation of area in the elementary
  geometries.}
\newblock Simon Stevin \textbf{32} (1958):29--38

\bibitem[Est62a]{Est62Local-and-global-groups.-I}
van Est, W.~T.
\newblock \emph{Local and global groups. {I}}.
\newblock Nederl. Akad. Wetensch. Proc. Ser. A 65 = Indag. Math. \textbf{24}
  (1962):391--408

\bibitem[Est62b]{Est62Local-and-global-groups.-II}
van Est, W.~T.
\newblock \emph{Local and global groups. {II}}.
\newblock Nederl. Akad. Wetensch. Proc. Ser. A 65 = Indag. Math. \textbf{24}
  (1962):409--425

\bibitem[FB02]{Forrester-Barker02Group-Objects-and-Internal-Categories}
Forrester-Barker, M.
\newblock \emph{{G}roup {O}bjects and {I}nternal {C}ategories} 2002.
\newblock \href{http://arxiv.org/abs/math/0212065}{\texttt{arXiv:math/0212065}}

\bibitem[FHRZ08]{FelderHenriquesRossiZhu08A-gerbe-for-the-elliptic-gamma-function}
Felder, G., Henriques, A., Rossi, C.~A. and Zhu, C.
\newblock \emph{A gerbe for the elliptic gamma function}.
\newblock Duke Math. J. \textbf{141} (2008)(1):1--74.
\newblock \href{http://dx.doi.org/10.1215/S0012-7094-08-14111-0}{\texttt{doi:10.1215/S0012-7094-08-14111-0}}

\bibitem[Fri82]{Friedlander82Etale-homotopy-of-simplicial-schemes}
Friedlander, E.~M.
\newblock \emph{\'{E}tale homotopy of simplicial schemes}, \emph{Annals of
  Mathematics Studies}, vol. 104 (Princeton University Press, Princeton, N.J.,
  1982)

\bibitem[Get09]{Getzler09Lie-theory-for-nilpotent-Linfty-algebras}
Getzler, E.
\newblock \emph{Lie theory for nilpotent {$L_\infty$}-algebras}.
\newblock Ann. of Math. (2) \textbf{170} (2009)(1):271--301.
\newblock \href{http://dx.doi.org/10.4007/annals.2009.170.271}{\texttt{doi:10.4007/annals.2009.170.271}}

\bibitem[Gir71]{Giraud71Cohomologie-non-abelienne}
Giraud, J.
\newblock \emph{Cohomologie non ab\'elienne} (Springer-Verlag, Berlin, 1971).
\newblock Die Grundlehren der mathematischen Wissenschaften, Band 179

\bibitem[GJ99]{GoerssJardine99Simplicial-homotopy-theory}
Goerss, P.~G. and Jardine, J.~F.
\newblock \emph{Simplicial homotopy theory}, \emph{Progress in Mathematics},
  vol. 174 (Birkh{\"a}user Verlag, Basel, 1999)

\bibitem[Gle82]{Glenn82Realization-of-cohomology-classes-in-arbitrary-exact-categories}
Glenn, P.~G.
\newblock \emph{Realization of cohomology classes in arbitrary exact
  categories}.
\newblock J. Pure Appl. Algebra \textbf{25} (1982)(1):33--105.
\newblock \href{http://dx.doi.org/10.1016/0022-4049(82)90094-9}{\texttt{doi:10.1016/0022-4049(82)90094-9}}

\bibitem[Gl{\"o}02a]{Glockner02Infinite-dimensional-Lie-groups-without-completeness-restrictions}
Gl{\"o}ckner, H.
\newblock \emph{Infinite-dimensional {L}ie groups without completeness
  restrictions}.
\newblock In \emph{Geometry and Analysis on Finite- and Infinite-Dimensional
  {L}ie Groups (B\polhk edlewo, 2000)}, \emph{Banach Center Publ.}, vol.~55,
  pp. 43--59 (Polish Acad. Sci., Warsaw, 2002)

\bibitem[Gl{\"o}02b]{Glockner02Lie-group-structures-on-quotient-groups-and-universal-complexifications-for-infinite-dimensional-Lie-groups}
Gl{\"o}ckner, H.
\newblock \emph{Lie group structures on quotient groups and universal
  complexifications for infinite-dimensional {L}ie groups}.
\newblock J. Funct. Anal. \textbf{194} (2002)(2):347--409.
\newblock \href{http://dx.doi.org/10.1006/jfan.2002.3942}{\texttt{doi:10.1006/jfan.2002.3942}}

\bibitem[Gl{\"o}05]{Glockner05Fundamentals-of-direct-limit-Lie-theory}
Gl{\"o}ckner, H.
\newblock \emph{Fundamentals of direct limit {L}ie theory}.
\newblock Compos. Math. \textbf{141} (2005)(6):1551--1577

\bibitem[GN14]{GlocknerNeeb14Infinite-dimensional-Lie-groups}
Gl{\"o}ckner, H. and Neeb, K.-H.
\newblock \emph{Infinite-dimensional {L}ie groups} (Springer-Verlag, 2014)

\bibitem[Gom05]{Gomi05Equivariant-smooth-Deligne-cohomology}
Gomi, K.
\newblock \emph{Equivariant smooth {D}eligne cohomology}.
\newblock Osaka J. Math. \textbf{42} (2005)(2):309--337.
\newblock
  \urlprefix\url{http://projecteuclid.org/getRecord?id=euclid.ojm/1153494380}

\bibitem[Hen08]{Henriques08Integrating-Lsb-infty-algebras}
Henriques, A.
\newblock \emph{Integrating {$L\sb \infty$}-algebras}.
\newblock Compos. Math. \textbf{144} (2008)(4):1017--1045.
\newblock \href{http://arxiv.org/abs/math/0603563}{\texttt{arXiv:math/0603563}}

\bibitem[KM97]{KrieglMichor97The-Convenient-Setting-of-Global-Analysis}
Kriegl, A. and Michor, P.~W.
\newblock \emph{The {C}onvenient {S}etting of {G}lobal {A}nalysis},
  \emph{Mathematical Surveys and Monographs}, vol.~53 (American Mathematical
  Society, Providence, RI, 1997)

\bibitem[Lan99]{Lang99Fundamentals-of-differential-geometry}
Lang, S.
\newblock \emph{Fundamentals of differential geometry}, \emph{Graduate Texts in
  Mathematics}, vol. 191 (Springer-Verlag, New York, 1999)

\bibitem[LBv11]{Li-BlandSevera11Integration-of-Exact-Courant-Algebroids}
Li-Bland, D. and \v{S}evera, P.
\newblock \emph{{I}ntegration of {E}xact {C}ourant {A}lgebroids} 2011.
\newblock \href{http://arxiv.org/abs/1101.3996}{\texttt{arXiv:1101.3996}}

\bibitem[Met03]{Metzler03Topological-and-Smooth-Stacks}
Metzler, D.
\newblock \emph{{T}opological and {S}mooth {S}tacks} 2003.
\newblock \href{http://arxiv.org/abs/math/0306176}{\texttt{arXiv:math/0306176}}

\bibitem[MM03]{MoerdijkMrcun03Introduction-to-foliations-and-Lie-groupoids}
Moerdijk, I. and Mr{\v{c}}un, J.
\newblock \emph{Introduction to foliations and {L}ie groupoids},
  \emph{Cambridge Studies in Advanced Mathematics}, vol.~91 (Cambridge
  University Press, Cambridge, 2003)

\bibitem[MT10]{MehtaTang10From-double-Lie-groupoids-to-local-Lie-2-groupoids}
Mehta, R.~A. and Tang, X.
\newblock \emph{{F}rom double {L}ie groupoids to local {L}ie 2-groupoids}.
\newblock Bulletin of the Brazilian Mathematical Society 42 (2011), no. 4,
  651-681 2010.
\newblock \href{http://dx.doi.org/10.1007/s00574-011-0033-4}{\texttt{doi:10.1007/s00574-011-0033-4}}.
\newblock \href{http://arxiv.org/abs/1012.4103}{\texttt{arXiv:1012.4103}}

\bibitem[Nee02]{Neeb02Central-extensions-of-infinite-dimensional-Lie-groups}
Neeb, K.-H.
\newblock \emph{Central extensions of infinite-dimensional {L}ie groups}.
\newblock Ann. Inst. Fourier (Grenoble) \textbf{52} (2002)(5):1365--1442.
\newblock \urlprefix\url{http://annalif.ujf-grenoble.fr/Vol52/E5253/E5253.html}

\bibitem[Nee06]{Neeb06Towards-a-Lie-theory-of-locally-convex-groups}
Neeb, K.-H.
\newblock \emph{Towards a {L}ie theory of locally convex groups}.
\newblock Jpn. J. Math. \textbf{1} (2006)(2):291--468

\bibitem[Noo08]{Noohi08Notes-on-2-groupoids-2-groups-and-crossed-modules}
Noohi, B.
\newblock \emph{{N}otes on 2-groupoids, 2-groups and crossed-modules}.
\newblock Homotopy, Homology, and Applications, 9, (2007), no. 1, 75--106 2008.
\newblock \href{http://arxiv.org/abs/math/0512106}{\texttt{arXiv:math/0512106}}

\bibitem[NSW13]{NikolausSachseWockel13A-smooth-model-for-the-string-group}
Nikolaus, T., Sachse, C. and Wockel, C.
\newblock \emph{A smooth model for the string group}.
\newblock Int. Math. Res. Not. IMRN  (2013)(16):3678--3721.
\newblock \href{http://dx.doi.org/10.1093/imrn/rns154}{\texttt{doi:10.1093/imrn/rns154}}.
\newblock \href{http://arxiv.org/abs/1104.4288}{\texttt{arXiv:1104.4288}}

\bibitem[NW11]{NikolausWaldorf11Four-Equivalent-Versions-of-Non-Abelian-Gerbes}
Nikolaus, T. and Waldorf, K.
\newblock \emph{{F}our {E}quivalent {V}ersions of {N}on-{A}belian {G}erbes}
  2011.
\newblock \href{http://arxiv.org/abs/1103.4815}{\texttt{arXiv:1103.4815}}

\bibitem[Por08]{Porst08Strict-2-Groups-are-Crossed-Modules}
Porst, S.~S.
\newblock \emph{{S}trict 2-{G}roups are {C}rossed {M}odules} 2008.
\newblock \href{http://arxiv.org/abs/0812.1464}{\texttt{arXiv:0812.1464}}

\bibitem[Pra68]{Pradines68Troisieme-theoreme-de-Lie-les-groupoi-des-differentiables}
Pradines, J.
\newblock \emph{Troisi{\`e}me th{\'e}or{\`e}me de {L}ie les groupo\"\i des
  diff{\'e}rentiables}.
\newblock C. R. Acad. Sci. Paris S{\'e}r. A-B \textbf{267} (1968):A21--A23

\bibitem[Pro96]{Pronk96Etendues-and-stacks-as-bicategories-of-fractions}
Pronk, D.~A.
\newblock \emph{Etendues and stacks as bicategories of fractions}.
\newblock Compositio Math. \textbf{102} (1996)(3):243--303.
\newblock \urlprefix\url{http://www.numdam.org/item?id=CM_1996__102_3_243_0}

\bibitem[Seg68]{Segal68Classifying-spaces-and-spectral-sequences}
Segal, G.
\newblock \emph{Classifying spaces and spectral sequences}.
\newblock Inst. Hautes {\'E}tudes Sci. Publ. Math. \textbf{34} (1968):105--112

\bibitem[Smi51a]{Smith51The-complex-of-a-group-relative-to-a-set-of-generators.-I}
Smith, P.~A.
\newblock \emph{The complex of a group relative to a set of generators. {I}}.
\newblock Ann. of Math. (2) \textbf{54} (1951):371--402

\bibitem[Smi51b]{Smith51The-complex-of-a-group-relative-to-a-set-of-generators.-II}
Smith, P.~A.
\newblock \emph{The complex of a group relative to a set of generators. {II}}.
\newblock Ann. of Math. (2) \textbf{54} (1951):403--424

\bibitem[SP11]{Schommer-Pries10Central-Extensions-of-Smooth-2-Groups-and-a-Finite-Dimensional-String-2-Group}
Schommer-Pries, C.
\newblock \emph{{C}entral extensions of smooth 2-groups and a
  finite-dimensional string 2-group}.
\newblock Geom. Topol. \textbf{15} (2011):609---676.
\newblock \href{http://dx.doi.org/10.2140/gt.2011.15.609}{\texttt{doi:10.2140/gt.2011.15.609}}.
\newblock \href{http://arxiv.org/abs/0911.2483}{\texttt{arXiv:0911.2483}}

\bibitem[SZ11]{ShengZhu11Higher-Extensions-of-Lie-Algebroids-and-Application-to-Courant-Algebroids}
Sheng, Y. and Zhu, C.
\newblock \emph{{H}igher {E}xtensions of {L}ie {A}lgebroids and {A}pplication
  to {C}ourant {A}lgebroids} 2011.
\newblock \href{http://arxiv.org/abs/1103.5920}{\texttt{arXiv:1103.5920}}

\bibitem[TZ06]{TsengZhu06Integrating-Lie-algebroids-via-stacks}
Tseng, H.-H. and Zhu, C.
\newblock \emph{Integrating {L}ie algebroids via stacks}.
\newblock Compos. Math. \textbf{142} (2006)(1):251--270

\bibitem[Wei94]{Weibel94An-introduction-to-homological-algebra}
Weibel, C.~A.
\newblock \emph{An introduction to homological algebra}, \emph{Cambridge
  Studies in Advanced Mathematics}, vol.~38 (Cambridge University Press,
  Cambridge, 1994)

\bibitem[Woc06]{Wockel06Smooth-extensions-and-spaces-of-smooth-and-holomorphic-mappings}
Wockel, C.
\newblock \emph{Smooth extensions and spaces of smooth and holomorphic
  mappings}.
\newblock J. Geom. Symmetry Phys. \textbf{5} (2006):118--126.
\newblock \href{http://arxiv.org/abs/math/0511064}{\texttt{arXiv:math/0511064}}

\bibitem[Woc10]{Wockel09Non-integral-central-extensions-of-loop-groups}
Wockel, C.
\newblock \emph{{N}on-integral central extensions of loop groups}.
\newblock Contemp. Math. \textbf{519} (2010):203--214.
\newblock \href{http://arxiv.org/abs/0910.1937}{\texttt{arXiv:0910.1937}}

\bibitem[Woc11a]{Wockel08Categorified-central-extensions-etale-Lie-2-groups-and-Lies-Third-Theorem-for-locally-exponential-Lie-algebras}
Wockel, C.
\newblock \emph{{C}ategorified central extensions, {\'e}tale {L}ie 2-groups and
  {L}ie's {T}hird {T}heorem for locally exponential {L}ie algebras}.
\newblock Adv. Math. \textbf{228} (2011):2218--2257.
\newblock \href{http://dx.doi.org/10.1016/j.aim.2011.07.003}{\texttt{doi:10.1016/j.aim.2011.07.003}}.
\newblock \href{http://arxiv.org/abs/0812.1673}{\texttt{arXiv:0812.1673}}

\bibitem[Woc11b]{Wockel09Principal-2-bundles-and-their-gauge-2-groups}
Wockel, C.
\newblock \emph{{P}rincipal 2-bundles and their gauge 2-groups}.
\newblock Forum Math. \textbf{23} (2011):565--610.
\newblock \href{http://dx.doi.org/10.1515/FORM.2011.020}{\texttt{doi:10.1515/FORM.2011.020}}.
\newblock \href{http://arxiv.org/abs/0803.3692}{\texttt{arXiv:0803.3692}}

\bibitem[Woc13]{Wockel13Infinite-dimensional-and-higher-structures-in-differential-geometry}
Wockel, C.
\newblock \emph{Infinite-dimensional and higher structures in differential
  geometry}.
\newblock Lecture notes for a course given at the University of Hamburg 2013.
\newblock
  \urlprefix\url{http://www.math.uni-hamburg.de/home/wockel/teaching/higher_structures.html}

\bibitem[WW13]{WagemannWockel13A-Cocycle-Model-for-Topological-and-Lie-Group-Cohomology}
Wagemann, F. and Wockel, C.
\newblock \emph{{A} {C}ocycle {M}odel for {T}opological and {L}ie {G}roup
  {C}ohomology}.
\newblock Trans. Amer. Math. Soc. (to appear, 32. pp)  (2013).
\newblock \href{http://arxiv.org/abs/1110.3304}{\texttt{arXiv:1110.3304}}

\bibitem[Zhu09]{Zhu09n-groupoids-and-stacky-groupoids}
Zhu, C.
\newblock \emph{$n$-groupoids and stacky groupoids}.
\newblock Int. Math. Res. Not. IMRN \textbf{21} (2009):4087--4141.
\newblock \href{http://dx.doi.org/10.1093/imrn/rnp080}{\texttt{doi:10.1093/imrn/rnp080}}.
\newblock \href{http://arxiv.org/abs/0801.2057}{\texttt{arXiv:0801.2057}}

\end{thebibliography}
 \def\polhk#1{\setbox0=\hbox{#1}{\ooalign{\hidewidth
  \lower1.5ex\hbox{`}\hidewidth\crcr\unhbox0}}} \def\cprime{$'$}

\end{document}